\documentclass[11pt,epsfig,amsfonts]{amsart}
\usepackage{stmaryrd}
\usepackage{xcolor}
\usepackage{epsfig}
\usepackage{enumerate}
\usepackage{enumitem}
\usepackage{amsmath}
\usepackage{amssymb}
\usepackage{amscd}
\usepackage{graphicx}
\usepackage{lineno}
\usepackage{pstricks}
\usepackage{tocvsec2}
\usepackage{mathtools}
\allowdisplaybreaks[4]

\usepackage [backref, colorlinks, citecolor=red, anchorcolor=black, linkcolor=blue]{hyperref}

\usepackage{amsfonts}
\usepackage[top=35mm, bottom=35mm, left=30mm, right=30mm]{geometry}
\usepackage{verbatim}
\usepackage{graphicx, amssymb, amsmath,amsthm}
\allowdisplaybreaks[4]
\usepackage{appendix}
 \usepackage{mathrsfs}
 \usepackage{cite}

\usepackage{xcolor}

\usepackage{xcolor}

\newtheorem{definition}{Definition}[section]
\newtheorem{lemma}{Lemma}[section]

\newtheorem{theorem}{Theorem}
\newtheorem{corollary}{Corollary}
\newtheorem{remark}{Remark}[section]

\newtheorem{proposition}{Proposition}[section]

\numberwithin{equation}{section}
\usepackage{epsfig}
\newtheorem{claim}{Claim}[section]

\numberwithin{equation}{section}

\renewcommand*{\backref}[1]{}
\renewcommand*{\backrefalt}[4]{\quad \tiny
  \ifcase #1 (\textbf{NOT CITED.})%
  \or    (Cited on Section~#2.)%
  \else   (Cited on Section~#2.)%
  \fi}

\makeatletter

\def\MRbibitem{\@ifnextchar[\my@lbibitem\my@bibitem}

\def\mybiblabel#1#2{\@biblabel{{\hyperref{http://www.ams.org/mathscinet-getitem?mr=#1}{}{}{#2}}}}

\def\myhyperanchor#1{\Hy@raisedlink{\hyper@anchorstart{cite.#1}\hyper@anchorend}}

\def\my@lbibitem[#1]#2#3#4\par{%
  \item[\mybiblabel{#2}{#1}\myhyperanchor{#3}\hfill]#4%
  \@ifundefined{ifbackrefparscan}{}{\BR@backref{#3}}%
  \if@filesw{\let\protect\noexpand\immediate
    \write\@auxout{\string\bibcite{#3}{#1}}}\fi\ignorespaces%
}

\def\my@bibitem#1#2#3\par{%
  \refstepcounter\@listctr
  \item[\mybiblabel{#1}{\the\value\@listctr}\myhyperanchor{#2}\hfill]#3%
  \@ifundefined{ifbackrefparscan}{}{\BR@backref{#2}}%
  \if@filesw\immediate\write\@auxout
    {\string\bibcite{#2}{\the\value\@listctr}}\fi\ignorespaces%
}

\makeatother

\subjclass[2020]{Primary: 37D35, 37C45.  Second: 37B40.}
\keywords{fiber specification, fiber Bowen's topological entropy, fiber measure theoretical entropy, variational principle}

\author{Nian Liu}
\address[Nian Liu]
{Department of Mathematics\\
 Pennsylvania State University\\
 State College, PA 16801, USA}
\email[N.~Liu]{nkl5330@psu.edu}

\author{Xue Liu}
\address[Xue Liu]
{School of Mathematics\\ Southeast University\\
 Nanjing 211189, PR China}
\email[X.~Liu]{xueliuseu@seu.edu.cn}

\begin{document}

\begin{abstract}
In this paper, we establish a variational principle, between the fiber Bowen's topological entropy on conditional level sets of Birkhoff average and fiber measure-theoretical entropy, for the skew product transformation driven by a uniquely ergodic homeomorphism system satisfying Anosov and topological mixing on fibers property. We prove it by utilizing a fiber specification property.
Moreover, we prove that such skew product transformation has specification property defined by Gundlach and Kifer\cite{Kifer2000specification}. Employing their main results, every H\"older continuous potential has a unique equilibrium state, and we also establish a variational principle between the fiber measure-theoretic entropy and the fiber Bowen's topological entropy on conditional level sets of local entropy for such unique equilibrium state.
Examples of systems under consideration are given, such as fiber Anosov maps on 2-dimension tori driven by any irrational rotation on circle and random composition of 2x2 area preserving positive matrices driven by uniquely ergodic subshift.
\end{abstract}


\title{On the variational principle for a class of skew product transformations}
\maketitle

\tableofcontents
\parskip 10pt         

\section{Introduction}\label{section 1}
The present paper is devoted to the study of the multifractal structure of Birkhoff averages for skew product transformations.

\subsection{Motivation}Given a continuous transformation $f$ on a compact metric space $M$, and a continuous observable $\varphi:M\to\mathbb{R}$, then the space $M$ has a natural multifractal decomposition
\begin{equation*}
  M=\cup_{\alpha\in\mathbb{R}} M_{\varphi,\alpha}\cup I_\varphi,
\end{equation*}where $ M_{\varphi,\alpha}$ denotes the level sets of Birkhoff average
\begin{equation*}
   M_{\varphi,\alpha}=\{x\in M:\ \lim_{n\to\infty}\frac{1}{n}\sum_{i=0}^{n-1}\varphi(f^i(x))=\alpha\},
\end{equation*}and $I_{\varphi}$ denotes the collection of points for which the Birkhoff average does not exist. The multifractal analysis aims to investigates the dimension complexity of these sets, where the dimension characteristics include Hausdorff dimension, Bowen's topological entropy\cite{Bowen1973Topentropy}, and topological pressure given by Pesin and Pitskel \cite{PesinPitskel1984toppressure}. These dimension characteristic are all special cases of Carathéodory dimension structure, see \cite{Pesinbook} or \cite{climenhagaPesin}. The Hausdorff dimension and Bowen's topological entropy of $M_{\varphi,\alpha}$ was extensively studied in past decades, see \cite{PesinWeiss1997, TakensVerbitski1999specification, BarreiraSaussol2001, BarreiraSussolSchmeling2002, FengLauWu2002, Olsen2003, JohanssonThomasPollicott2010, Climenhaga2013PressureSimultaneousset, TianVarandas2017} and reference therein.
In particular, for a continuous observable $\varphi$ and for those $\alpha$ such that $M_{\varphi,\alpha}\not=\emptyset$, one established the following variational principle between $h_{top}(M_{\varphi,\alpha})$, the Bowen topological entropy of $M_{\varphi,\alpha}$, and the measure-theoretical entropy $h_\mu(f)$
\begin{equation*}
  h_{top}(M_{\varphi,\alpha})=\sup\{h_\mu(f):\ \mu\in\mathcal{M}_f(M) \mbox{ and }\int\varphi d\mu=\alpha\},
\end{equation*}for topological mixing subshift of finite type\cite{FanFeng2001}, then for systems with specification property\cite{TakensVerbitsky2003variationalprinciple}, for systems with $g$-almost product property\cite{PfisterSullivan2007galmost}, and for systems with non-uniform specification property\cite{TianVarandas2017}. For Banach valued continuous observable, similar variational principle was also established for system with specification property\cite{FanLiaoPeryiere2008}. As a generalisation of above variational principle, the variational principle regarding topological pressure of $M_{\varphi,\alpha}$ was also established for systems we mentioned above \cite{Thompson2009Variationalprinciplepressure, PeiYuErcaiChen2010, YinErcaiZhou2016}. Meanwhile, for the above systems,  the measure negligible set $I_\varphi$ carries full topological pressure \cite{Thompson2010irregularspecification, Thompson2012almostspecification, Tian2017}.  In the proof of above results, specification or ``weak" specification property plays a key role, and it is called orbit-gluing approach in some reference\cite{Climenhaga2014}. Roughly speaking, the specification property guarantees that any finite collection of arbitrary long orbit segments can be shadowed by the orbit of a point within given precision as long as one allows for enough time gap between segments.

In the scope of skew product transformation or random dynamical systems (which can also be transferred to skew product transformation), the conditional level sets of Birkhoff average has been investigated from Hausdorff dimension point of view for several systems. In \cite{FengShu2009}, for the product of a ergodic transformation on a Lebesgue space and a full shift of finite type, authors relates the Hausdorff dimension of conditional level sets of Birkhoff average and the relative measure-theoretic entropy via a variational principle. In \cite{Shulin2010}, the author proved that the Hausdorff dimension of conditional level sets of Birkhoff average for random perturbation of mixing smooth conformal repeller is approximating to the Hausdorff dimension of Birkhoff average for the unperturbed system. One of difficulty in studying the conditional level sets of Birkhoff average from topological entropy point of view is the missing of orbits gluing technique.
 Recently, in \cite{HLL}, for Anosov systems driven by external force satisfying topological mixing on fibers property, authors state a random specification property, which motivate us the fiber specification property. The fiber specification property introduced in this paper provides an orbit-gluing technique along the orbit of external force, which gives possibility to analysis the multifratcal structure of conditional level sets of Birkhoff average from the topological entropy point of view for skew product transformation generated by such external forced system.

For the purpose of getting external forced system involved, we introduce the fiber topological entropy on non-compact sets by employing the Carathéodory dimension structure (abbr. C-structure). The fiber Bowen's  topological entropy is a generalisation of the topological entropy defined by Bowen\cite{Bowen1973Topentropy} or equivalently by Pesin\cite{Pesinbook} on non-compact sets for deterministic system. While the classical fiber topological entropy resembles the box dimension, the fiber Bowen's topological entropy defined by C-structure resembles the Hausdorff dimension. In following subsections, we will introduce the settings of system, fiber specification property, fiber Bowen's topological entropy and our main results in details.
\subsection{Settings of system and fiber specification property}\label{section settings}
 Let $M$ be a connected closed smooth Riemannian manifold, and $d_M$ be the induced Riemannian metric on $M$. Let $\theta:\Omega\rightarrow \Omega$ be a homeomorphism on a compact metric space $(\Omega,d_\Omega)$. The product space $\Omega\times M$ is a compact metric space with distance $d((\omega_1,x_1),(\omega_2,x_2))=d_\Omega(\omega_1,\omega_2)+d_M(x_1,x_2)$ for any $\omega_1,\omega_2\in\Omega$ and $x_1,x_2\in M$. Let $\mathcal{H}=$diff$^2(M)$ be the space of $C^2$ diffeomorphisms on $M$ equipped with $C^2$-topology, and $F:\Omega\rightarrow \mathcal{H}$ be a continuous map. The skew product $\Theta:\Omega\times M\rightarrow \Omega\times M$ induced by $F$ and $\theta$ is defined by:
\begin{equation*}
  \Theta(\omega,x)=(\theta\omega,F_\omega x),\ \forall\omega\in\Omega,\ x\in M.
\end{equation*}
where we rewrite $F(\omega) $ as $F_\omega$. Then inductively:
\begin{equation*}
  \Theta^n(\omega,x)=(\theta^n\omega,F_\omega^nx):=
  \begin{cases}
  (\theta^n\omega,F_{\theta^{n-1}\omega}\circ\cdots\circ F_\omega x), & \mbox{if } n>0 \\(\omega,x), & \mbox{if } n=0 \\ (\theta^n\omega,(F_{\theta^n\omega})^{-1}\circ\cdots\circ (F_{\theta^{-1}\omega})^{-1}x), & \mbox{if } n<0.
  \end{cases}
\end{equation*}In this paper, we consider the target system $(\Omega\times M,\Theta)$ satisfying:
\begin{enumerate}[label=\textbf{(C.\arabic*)}]
  \item \label{C1} the driven system $(\Omega,\mathcal{B}_P(\Omega),P,\theta)$ is a uniquely ergodic homeomorphism equipped with the unique invariant Borel probability measure $P$ and $\sigma$-algebra $\mathcal{B}_P(\Omega)$, which is the completion of Borel $\sigma$-algebra with respect to $P$;
  \item \label{C2} Anosov on fibers:  for every $(\omega,x)\in \Omega\times M$, there is a splitting of the tangent bundle of $M_\omega=\{\omega\}\times M$ at $x$
\begin{equation*}
  T_xM_\omega=E_{(\omega,x)}^s\oplus E_{(\omega,x)}^u
\end{equation*}which depends continuously on $(\omega,x)\in M\times \Omega$ with $\dim E^s_{(\omega,x)},\ \dim E_{(\omega,x)}^u>0$ and satisfies that
\begin{equation*}
  D_xF_\omega E_{(\omega,x)}^u=E_{\Theta(\omega,x)}^u,\ \ D_xF_\omega E_{(\omega,x)}^s=E_{\Theta(\omega,x)}^s
\end{equation*}and
\begin{equation*}
  \begin{cases}
    |D_xF_\omega\xi|\geq e^{\lambda_0}|\xi|, & \ \forall \xi\in E_{(\omega,x)}^u, \\
    |D_xF_\omega\eta|\leq e^{-\lambda_0}|\eta|, &\ \forall \eta\in E_{(\omega,x)}^s,
  \end{cases}
\end{equation*}where $\lambda_0>0$ is a constant;
\item \label{C3} topological mixing on fibers: for any nonempty open sets $U,V\subset M$, there exists  $N>0$ such that for any $n\geq N$ and $\omega\in\Omega$
\begin{equation*}
  \Theta^n(\{\omega\}\times U)\cap (\{\theta^n\omega\}\times V)\not=\emptyset.
\end{equation*}
\end{enumerate}Examples of Anosov and topological mixing on fibers skew product driven by uniquely ergodic homeomorphism are given in Section \ref{section example}.

By Remark 1.1.8 in \cite{ArnoldBook1998}, we can identity the skew product $\Theta$ and the corresponding random dynamical system (abbr. RDS), which is $F:\mathbb{Z}\times \Omega\times M\to M$ by $F(n,\omega,x)=F_\omega^nx$. Here, we still use notation $F$ by avoiding abuse of notation.

\begin{definition}\label{def fiber specification}
 For any $\omega\in\Omega$, an $\omega-$specification $S_\omega=(\omega,\tau,P_\omega)$ consists of a finite collection of intervals $\tau=\{I_1,...,I_k\}$, $I_i=[a_i,b_i]\subset \mathbb{Z}$, and a map $P_{\omega}:\cup_{i=1}^kI_i\rightarrow M$ such that for $t_1,t_2\in I\in \tau,$
  \begin{equation*}
    F_{\theta^{t_1}\omega}^{t_2-t_1}( P_\omega(t_1))=P_\omega(t_2).
  \end{equation*}
An $\omega$-specification $S_\omega$ is called $m-$spaced if $a_{i+1}>b_i+m$ for all $i\in\{1,...,k-1\}$.

   The system $\Theta$ is said to have the fiber specification property if for any $\epsilon>0$, there exists $m=m(\epsilon)>0$ such that for any $\omega\in\Omega$, any $m$-spaced $\omega$-specification $S_\omega=(\omega,\tau,P_\omega)$, we can find a point $x\in M$, which is $(\omega,\epsilon)$-shadowing $S_\omega$, i.e.,
\begin{equation*}
d_M(P_\omega(t),F_\omega^tx)<\epsilon\mbox{ for any }t\in I\in \tau.
\end{equation*}
\end{definition}

\subsection{Fiber Bowen's topological entropy}

In this section, we introduce the fiber Bowen's  topological entropy on noncompact set, which is a special case of Carathéodory dimension structure. The definition of C-structure can be found in Chap. 1 of \cite{Pesinbook} or Sec. 5 of \cite{climenhagaPesin}. We note that for a fixed $\omega\in\Omega$, the systems along the orbit of $\omega$ can be viewed as an non-autonomous systems, whose Bowen's topological entropy is given in \cite{kolyada}.

For $n\in\mathbb{N}$, $\omega\in \Omega$, we denote the fiber Bowen's metric
\begin{equation*}
  d_\omega^n(x,y)=\max_{0\leq i\leq n-1}\{d_M(F_\omega^ix,F_\omega^iy)\}\mbox{ for }x,y\in M.
\end{equation*}Let $B_n(\omega,x,\epsilon)\subset M$ be the open ball of radius $\epsilon$ around $x\in M$ with respect to the metric $d_\omega^n$. For any $\epsilon>0$, $s\in\mathbb{R}$, on  the fiber $\{\omega\}\times M$, for any nonempty set $Z\subset M$, we define
\begin{equation}\label{def M(Z,s,w,N,epsilon)}
 m(Z,s,\omega,N,\epsilon)= \inf_{\Gamma_\omega^\epsilon}\sum_{i}e^{-sn_i},
\end{equation}where the infimum is taken over all finite or countable collections $\Gamma_\omega^\epsilon=\{B_{n_i}(\omega,x_i,\epsilon)\}$ with $Z\subset \cup_iB_{n_i}(\omega,x_i,\epsilon)$ and $\min\{n_i\}\geq N$. Note that $m(Z,s,\omega,N,\epsilon)$ does not decrease as $N$ increases, the following limit exists
\begin{align}\label{def m(z,s,omega,N,epsilon)}
   m(Z,s,\omega,\epsilon)&=\lim_{N\rightarrow\infty} m(Z,s,\omega,N,\epsilon).
\end{align}There exists a number $h_{top}(F,Z,\omega,\epsilon)\in[0,+\infty)$ such that
  \begin{equation}\label{def m(z,s,omega,epsilon)}
     m(Z,s,\omega,\epsilon)=\begin{cases}
                              +\infty, & \mbox{if } s<h_{top}(F,Z,\omega,\epsilon) \\
                              0, & \mbox{if } s>h_{top}(F,Z,\omega,\epsilon) .
                            \end{cases}
  \end{equation} Note that $m(Z,h_{top}(F,Z,\omega,\epsilon),\omega,\epsilon)$ could be $+\infty,\ 0$ or positive finite number.
\begin{lemma}\label{lemma htopZ exist}
  The value $h_{top}(F,Z,\omega,\epsilon)$ does not decrease as $\epsilon$ decreases. Therefore, the following limit exists
  \begin{equation}\label{def htopZ}
    h_{top}(F,Z,\omega)=\lim_{\epsilon\rightarrow 0} h_{top}(F,Z,\omega,\epsilon)=\sup_{\epsilon>0}h_{top}(F,Z,\omega,\epsilon).
  \end{equation}
\end{lemma}The proof of Lemma \ref{lemma htopZ exist} is addressed in Sec. \ref{5.1}. We call $h_{top}(F,Z,\omega)$ the fiber Bowen's  topological entropy of $F$ restricted on $Z$ and the fiber $\{\omega\}\times M$. The following lemma is a corollary of \cite[Theorem 1.1]{Pesinbook}.
\begin{lemma}\label{lemma property htop}
  The fiber Bowen's topological entropy has the following properties:
  \begin{enumerate}
    \item $h_{top}(F,Z_1,\omega)\leq h_{top}(F,Z_2,\omega)$ if $\emptyset \subsetneq Z_1\subset Z_2\subset M$.
    \item $h_{top}(F,Z,\omega)=\sup_i h_{top}(F,Z_i,\omega)$ if $Z$ is finite or countable union of nonempty sets $ Z_i\subset M$.
  \end{enumerate}
\end{lemma}
For any subset $\{\omega\}\times Z\subset \Omega\times M$, denote $h_{top}(\Theta,\{\omega\}\times Z)$ to be the Bowen's topological entropy of topological dynamical system $(\Omega\times M,\Theta)$ on $\{\omega\}\times Z$ defined by Bowen \cite{Bowen1973Topentropy} or equivalently by using $C$-structure in \cite{TakensVerbitsky2003variationalprinciple} on noncompact sets. Note that for any $(\omega,x)\in \Omega\times M$, $n\in\mathbb{N}$, one has
\begin{equation*}
  \{\omega\}\times B_n(\omega,x,\epsilon)\subset B_n^\Theta((\omega,x),\epsilon):=\{(\beta,y)\in\Omega\times M:\ \max_{0\leq i<n}d(\Theta^i(\omega,x),\Theta^i(\beta,y))<\epsilon\}.
\end{equation*}Therefore, any covering $\{B_{n_i}(\omega,x_i,\epsilon)\}$ of $Z$ is corresponding to a covering $\{B_{n_i}^\Theta((\omega,x_i),\epsilon)\}$ of $\{\omega\}\times Z$. As a consequence, $h_{top}(\Theta,\{\omega\}\times Z)\leq h_{top}(F,Z,\omega)$ for all $\omega\in\Omega$.

\subsection{Conditional level sets of Birkhoff average}
In this section, we assume that $(\Omega\times M,\Theta)$ is Anosov and topological mixing on fibers system driven by a uniquely ergodic homeomorphism system $(\Omega,\mathcal{B}_P(\Omega),P,\theta)$.
For any continuous observable $\varphi:\Omega\times M\to \mathbb{R}$ and $\alpha\in\mathbb{R}$, denote the classical level set of Birkhoff average
\begin{equation*}
  K_{\varphi,\alpha}=\{(\omega,x)\in\Omega\times M:\ \lim_{n\rightarrow\infty}\frac{1}{n}\sum_{i=0}^{n-1}\varphi(\Theta^i(\omega,x))=\alpha\}
\end{equation*}and conditional level set of Birkhoff average $K_{\varphi,\alpha}(\omega)=\{x\in M: (\omega,x)\in K_{\varphi,\alpha}\}$. The proof of the following two lemmas are addressed in Sec. \ref{5.1}. 
\begin{lemma}\label{lemma Kalpha nonempty}
  For each $\alpha\in\mathbb{R}$, it is clear that $K_{\varphi,\alpha}\subset \Omega\times M$ is a $\Theta-$invariant Borel subset. Denote $\Omega_\alpha=\{\omega\in\Omega:\ K_{\varphi,\alpha}(\omega)\not=\emptyset\}$, then $\Omega_\alpha$ is a measurable set and either $P(\Omega_\alpha)=1$ or $P(\Omega_{\alpha})=0$.
\end{lemma}
Denote $L_\varphi:=\{\alpha\in\mathbb{R}:\ P(\Omega_\alpha)=1\}$, and denote $I_\Theta(\Omega\times M)$ and $I_\Theta^e(\Omega\times M)$ to be the collection of all $\Theta$-invariant  and $\Theta$-invariant ergodic Borel probability measures on $\Omega\times M$ respectively.
\begin{lemma}\label{lemma I alpha nonemptyset}
 For any $\alpha\in L_\varphi$, the set $I_\Theta(\Omega\times M,\varphi,\alpha):=\{\mu\in I_\Theta(\Omega\times M):\ \int\varphi d\mu=\alpha\}$ is a nonempty, convex and closed subset of $I_\Theta(\Omega\times M).$ As a consequence, $L_\varphi\subset \{\int \varphi d\mu: \ \mu\in I_\Theta(\Omega\times M)\}$. Moreover, $L_\varphi$ is nonempty bounded subset of $\mathbb{R}.$
\end{lemma}

For a topological dynamical system $(X,T)$ with the specification property, it is pointed out in \cite[Lemma 2.5]{Thompson2009Variationalprinciplepressure} that for any continuous observable $\psi:X\to\mathbb{R}$,
\begin{equation*}
  \left\{\alpha\in\mathbb{R}:\ \{\alpha\in\mathbb{R}:\ \lim_{n\to\infty}\frac{1}{n}\sum_{i=0}^{n-1}\psi(T^i(x))=\alpha\}\not=\emptyset\right\}=\left\{\int\psi d\mu:\ \mu\mbox{ is }T\mbox{-invariant}\right\},
\end{equation*}due to the fact that every $T$-invariant probability measure has a generic point. In this paper, it is not sure whether $\Theta$-invariant measure has generic points. Moreover, the existence of generic points is not enough to prove that $L_\varphi$ coincides $\{\int\varphi d\mu:\ \mu\in I_{\Theta}(\Omega\times M)\}$, since we need $P(\Omega_\alpha)=1$ for $\alpha\in L_\varphi$. However, our proof of Theorem \ref{thm variational principle} indicates that $L_\varphi=\{\int\varphi d\mu:\ \mu\in I_{\Theta}(\Omega\times M)\}$, see Remark \ref{remark final}.

\subsection{Main results}
We state our main results in this section.
In the first theorem, the unique ergodicity of the metric dynamical system $(\Omega,\theta)$ is not necessary.
\begin{theorem}\label{thm random specification}
  Assume that $(\Omega\times M,\Theta)$ is Anosov on fibers and topological mixing on fibers system driven by a homeomorphism $(\Omega,\theta)$, then $\Theta$ has the fiber specification property. Moreover, RDS $F$ corresponding to $\Theta$ has the specification property given by Gundlach and Kifer as in Remark \ref{remark kifer gundlach}.
\end{theorem}
Denote $Pr(\Omega\times M)$ and $Pr(M)$ to be the space of Borel probability measures on $\Omega\times M$ and $M$ respectively, which are compact spaces with respect to the weak$^*$ topology, and $Pr_P(\Omega\times M)\subset Pr(\Omega\times M)$ to be the Borel probability measure with marginal $P$ on $\Omega$, i.e. for any $\mu\in Pr_P(\Omega\times M)$, $\mu\circ \pi_\Omega^{-1}=P$, where $\pi_\Omega:\Omega\times M\to \Omega$ denotes the projection into the first coordinate. For any $\mu\in Pr_P(\Omega\times M)$, by the disintegration theorem (see \cite[Proposition 3.6]{Hans02}), there exists a measurable mapping called disintegration $\mu_\cdot:\Omega\to Pr(M)$ by $\omega\mapsto\mu_\omega$ such that $\int_{\Omega\times M} h(\omega,x)d\mu=\int_\Omega\int_Mh(\omega,x)d\mu_\omega(x) dP(\omega)$ for any continuous function $h:\Omega\times M\rightarrow \mathbb{R}$. Moreover, the disintegration is $P-$a.s. unique.

 When the uniquely ergodicity of $(\Omega,\theta)$ is present, one must have
\begin{equation}\label{invariant marginal P}
  I_\Theta(\Omega\times M)\subset Pr_P(\Omega\times M),
\end{equation}since $\mu\circ \pi_\Omega^{-1}$ is $\theta$-invariant for any $\mu\in I_\Theta(\Omega\times M)$. Relation \eqref{invariant marginal P} enable us to use lots of results in ergodic theory of RDS.

For a continuous potential $\varphi\in C(\Omega\times M,\mathbb{R})$, let $\pi_F(\varphi)$ be the fiber  topological pressure defined in \eqref{eq fiber top pressure} by using separated set, which is related to fiber measure theoretic entropy by the  following fiber (or relative) variational principle (see \cite[Theorem 1.2.13]{KL06})
\begin{equation}\label{fiber vp}
\begin{split}
  \pi_F(\varphi)=&\sup\left\{h_\mu(F)+\int \varphi d\mu:\ \mu\in I_\Theta(\Omega\times M)\cap Pr_P(\Omega\times M)\right\}\\
  \overset{\eqref{invariant marginal P}}=&\sup\left\{h_\mu(F)+\int \varphi d\mu:\ \mu\in I_\Theta(\Omega\times M)\right\},
\end{split}
\end{equation}where $h_\mu(F)=h_{\mu}^{(r)}(\Theta)$ is the fiber measure theoretic entropy of $F$ or relative measure theoretic entropy of $\Theta$ (see section \ref{section measure entropy}). When $\varphi=0$, $h_{top}(F):=\pi_F(0)$ is the classical fiber topological entropy. We also note that the generator of RDS $F_\omega\in$Diff$^2(M)$ and continuous depending on $\omega$. By the Margulis-Ruelle inequality for RDS \cite[Theorem 1]{Rulleinequality}, in the settings of this paper, one has
\begin{equation*}
  h_{top}(F)<\infty.
\end{equation*}

 With the help of fiber specification property in Theorem \ref{thm random specification}, we obtain the following variational principle, which can be viewed as a generalisation of variational principle established in \cite{TakensVerbitsky2003variationalprinciple}.
\begin{theorem}\label{thm variational principle}
  Assume $(\Omega\times M,\Theta)$ is Anosov and topological mixing on fibers system driven by a uniquely ergodic homeomorphism system $(\Omega,\mathcal{B}_P(\Omega),P,\theta)$. Let $\varphi\in C(\Omega\times M,\mathbb{R})$, for $\alpha\in L_\varphi$, then the following holds for $P$-a.s. $\omega\in\Omega$:
  \begin{equation*}
    h_{top}(F,K_{\varphi,\alpha}(\omega),\omega)=\max\{h_\mu(F):\ \mu\in I_\Theta(\Omega\times M,\varphi,\alpha)\}.
  \end{equation*}
\end{theorem}

  Theorem \ref{thm variational principle} and \eqref{fiber vp} imply that the fiber Bowen's topological entropy of $M$ coincides the classical fiber topological entropy neglecting a $P$-zero measure set for systems under consideration, i.e., $ h_{top}(F,M,\omega)=h_{top}(F)$ for $P$-a.s. $\omega\in\Omega$. Theorem \ref{thm variational principle} also reveals new phenomenon for hyperbolic toral automorphism, see Sec. \ref{subsection 2.1} and Corollary \ref{corollary 3}.

For $\varphi\in  C(\Omega\times M,\mathbb{R})$ and $q\in\mathbb{R}$, define $\pi_{F,\varphi}:\mathbb{R}\to\mathbb{R}$ by $\pi_{F,\varphi}(q)=\pi_F(q\varphi)$, which is the fiber topological pressure of $q\varphi$. The Legendre transform $\pi_{F,\varphi}^*$ is defined by
\begin{equation*}
  \pi_{F,\varphi}^*(\alpha)=\inf_{q\in\mathbb{R}}\{\pi_{F,\varphi}(q)-q\alpha\}\mbox{ for }\alpha\in\mathbb{R}.
\end{equation*}
The following theorem is consequence of Theorem \ref{thm variational principle} and Lemma \ref{lemma int subset inteq}.
\begin{theorem}\label{thm 3}
   Assume $(\Omega\times M,\Theta)$ is Anosov and topological mixing on fibers system driven by a uniquely ergodic homeomorphism system $(\Omega,\mathcal{B}_P(\Omega),P,\theta)$. Let $\varphi\in C(\Omega\times M,\mathbb{R})$, then
   \begin{enumerate}
     \item[(a)] for any $\alpha\in L_\varphi$, one has
         \begin{equation}\label{eq leq Legendre}
           h_{top}(F,K_{\varphi,\alpha}(\omega),\omega)\leq \pi_{F,\varphi}^*(\alpha)\mbox{ for }P-a.s.\ \omega\in\Omega.
         \end{equation}
     \item[(b)] for $\alpha\in int(L_\varphi),$
         one has
         \begin{equation}\label{eq eq Legendre}
           h_{top}(F,K_{\varphi,\alpha}(\omega),\omega)= \pi_{F,\varphi}^*(\alpha)\mbox{ for }P-a.s.\ \omega\in\Omega,
         \end{equation}where $int$ denotes the interior.
   \end{enumerate}
\end{theorem}

 The last result is about the conditional level sets of local entropy for equilibrium states. We recall the local entropy formula, which is a corollary of \cite[Theorem 2.1]{zhutwonotes}.
\begin{lemma}\label{lemma local entropy formula}
 Assume $(\Omega\times M,\Theta)$ is Anosov and topological mixing on fibers system driven by a uniquely ergodic homeomorphism system $(\Omega,\mathcal{B}_P(\Omega),P,\theta)$.
  For $\mu\in I^e_\Theta(\Omega\times M)$, note that $h_{\mu}(F)\leq h_{top}(F)<\infty$, then
  \begin{align*}
    h_\mu(F)&=\underbar{h}_\mu(F;\omega,x) :=\lim_{\epsilon\to 0}\liminf_{n\to\infty}-\frac{1}{n}\log\mu_\omega(B_n(\omega,x,\epsilon))\\
    &=\bar{h}_\mu(F;\omega,x):=\lim_{\epsilon\to 0}\limsup_{n\to\infty}-\frac{1}{n}\log\mu_\omega(B_n(\omega,x,\epsilon))
  \end{align*}for $\mu$-a.s. $(\omega,x)\in\Omega\times M$, where $\omega \mapsto\mu_\omega$ is the disintegration of
 $\mu$ with respect to $P$ by noticing \eqref{invariant marginal P}.
\end{lemma}

Theorem 3.9 in \cite{Kifer2000specification} and Theorem \ref{thm random specification} imply that the target systems has unique equilibrium state for H\"older potentials. Here, we recall the definition of equilibrium state. If there exists a maximizing measure in the fiber variational principle \eqref{fiber vp}, then this measure is called an fiber equilibrium state for $\varphi$. We note that the fiber equilibrium states in the fiber variational principle also obtain the maximal in the classical variation principle since
\begin{equation*}
  \sup\{h_\mu(\Theta)+\int \varphi d\mu: \mu\in I_\Theta(\Omega\times M)\}=\sup\{h_\mu(F)+\int\varphi d\mu:\ \mu\in I_\Theta(\Omega\times M)\}+h_P(\theta),
\end{equation*}where the equality due to \eqref{invariant marginal P} and the Abramov-Rohlin formula
\begin{equation}\label{AR formula}
  h_\mu(\Theta)=h_\mu(F)+h_{P}(\theta) \mbox{ for } \mu\in I_\Theta(\Omega\times M).
\end{equation} Therefore, we can identify equilibrium states and fiber equilibrium states in this paper.
\begin{corollary}\label{corollary eq}
  Assume that $(\Omega\times M,\Theta)$ is Anosov on fibers and topological mixing on fibers system driven by a uniquely ergodic homeomorphism $(\Omega,\mathcal{B}_P(\Omega),P,\theta)$, then for any H\"older continuous potential $\varphi:\Omega\times M\to\mathbb{R}$, there exists a unique equilibrium state for $\varphi$, named $\mu=\mu_\varphi\in Pr_P(\Omega\times M)$. Moreover, such unique equilibrium state is $\Theta$-invariant, ergodic and satisfies the Gibbs property, i.e., for any $\epsilon\in (0,\frac{\eta}{4})$, there exist positive constants $A_\epsilon, B_\epsilon$ such that for $P$-a.s. $\omega\in\Omega$,
  \begin{equation}\label{Aepsilon Bepsilon}
    A_\epsilon\leq \mu_\omega(B_n(\omega,x,\epsilon))\cdot\pi_F(\varphi)(\omega,\epsilon,n)\cdot e^{-\sum_{i=0}^{n-1}\varphi(\Theta^i(\omega,x))}\leq B_\epsilon\mbox{ for all }x\in M,
  \end{equation} where $\pi_F(\varphi)(\omega,\epsilon,n)$ is defined by using separated set in \eqref{eq pi F varphi}, and $\eta>0$ is a fiber expansive constant in Lemma \ref{Lemma expansive}.
\end{corollary}

Combining Corollary \ref{corollary eq}, Theorem \ref{thm variational principle} and Theorem \ref{thm 3}, we give the multifractal analysis of local entropy for equilibrium states.
\begin{corollary}\label{corollary 2}
   Assume $(\Omega\times M,\Theta)$ is Anosov and topological mixing on fibers system driven by a uniquely ergodic homeomorphism system $(\Omega,\mathcal{B}_P(\Omega),P,\theta)$. Let $\varphi:\Omega\times M\to\mathbb{R}$ be a H\"older potential and $\mu=\mu_\varphi$ be the equilibrium state given in Corollary \ref{corollary eq}. For $\alpha\in\mathbb{R}$, we denote
   \begin{equation*}
     E_\alpha=\{(\omega,x)\in\Omega\times M:\ \bar{h}_\mu(F;\omega,x)=\underbar{h}_\mu(F;\omega,x)=\pi_F(\varphi)-\alpha\},
   \end{equation*}where $\bar{h}_\mu(F;\omega,x)$ and $\underbar{h}_\mu(F;\omega,x)$ are local entropy defined in Lemma \ref{lemma local entropy formula}. Then
   \begin{enumerate}
     \item for $\mu$-a.s. $(\omega,x)\in\Omega\times M$,
         \begin{equation*}
           \bar{h}_\mu(F;\omega,x)=\underbar{h}_\mu(F;\omega,x)=h_\mu(F)=\pi_F(\varphi)-\int\varphi d\mu;
         \end{equation*}
     \item for $\alpha\in L_\varphi$,
     \begin{equation*}
       h_{top}(F,E_\alpha(\omega),\omega)=\max\{h_\mu(F):\ \mu\in I_\Theta(\Omega\times M,\varphi,\alpha)\}\mbox{ for }P-a.s.\ \omega\in\Omega;
     \end{equation*}
     \item for any $\alpha\in L_\varphi$,
         \begin{equation*}
           h_{top}(F,E_\alpha(\omega),\omega)\leq \pi_{F,\varphi}^*(\alpha)\mbox{ for }P-a.s.\ \omega\in\Omega;
         \end{equation*}
     \item for $\alpha\in int( L_\varphi),$
         \begin{equation*}
           h_{top}(F,E_\alpha(\omega),\omega)= \pi_{F,\varphi}^*(\alpha)\mbox{ for }P-a.s.\ \omega\in\Omega.
         \end{equation*}
   \end{enumerate}
\end{corollary}
The paper is organised as follows. In Sec. \ref{section example}, we give several examples under consideration. In Sec. \ref{section preliminary}, we introduce definitions and state several preliminary lemmas. The proof of our main theorems is addressed in Sec. \ref{section proofs}, and we collect all proof of lemmas in Sec. \ref{section proof of lemmas}.

\section{Examples}\label{section example}
In this section, we give some examples of Anosov and topological mixing on fibers system driven by a uniquely ergodic homeomorphism system.
\subsection{Fiber Anosov maps on 2-d tori}\label{subsection 2.1}
In section 8.1.2 of \cite{HLL}, Huang, Lian and Lu proved the following skew product system is Anosov and topological mixing on fibers.

The skew product $\Theta:\mathbb{T}\times \mathbb{T}^2 \to\mathbb{T} \times\mathbb{T}^2$ is given by
\begin{equation*}
  \Theta(\omega,x)=(\theta\omega,F_\omega(x)):=(\omega+\alpha,Tx+h(\omega)),
\end{equation*}where $\alpha\in\mathbb{R}\backslash\mathbb{Q}$, $T$ is any hyperbolic toral automorphism, and $h:\mathbb{T}\to\mathbb{T}^2$ is a continuous map. Note that the irrational rotation on circle $\theta:\omega\mapsto \omega+\alpha$ is a uniquely ergodic system with the unique $\theta$-invariant measure $P=$Lebesgue measure. Therefore, this system is one of target systems.

Since that the classical specification property implies topological mixing and the factor system of topological mixing system is also topological mixing, but $(\mathbb{T},\theta)$ as factor of $(\mathbb{T}\times \mathbb{T}^2,\Theta)$ is not topological mixing. Therefore, the system $(\mathbb{T}\times\mathbb{T}^2,\Theta)$ does not have classical specification property. As a consequence, the variational principle in \cite{TakensVerbitsky2003variationalprinciple} can not apply to $\Theta$.

Note that $T$ is linear map, for any $\omega\in \mathbb{T}$ and $n\in\mathbb{N}$, the fiber Bowen's metric
\begin{align*}
  d_\omega^n(x,y) & =\max_{1\leq i\leq n-1}\{d_{\mathbb{T}^2}(F_\omega^i(x),F_\omega^i(y))\}\\
  &=\max_{1\leq i\leq n-1}\{d_{\mathbb{T}^2}(T^i(x),T^i(y))\}=d_n^T(x,y),
\end{align*}where $d_n^T$ is the usual Bowen's metric induced by $T$ on $\mathbb{T}^2$. As a consequence, for all $\omega\in \Omega$, we have $B_n(\omega,x,\epsilon)=B_n^T(x,\epsilon):=\{y\in M:\ d_n^T(x,y)<\epsilon\}$, and therefore,
\begin{equation*}
  h_{top}(F,Z,\omega)=h_{top}(T,Z)
\end{equation*}for any nonempty subset $Z\subset \mathbb{T}^2$, where $h_{top}(T,Z)$ is the Bowen topological entropy defined on noncompact set for deterministic system $(\mathbb{T}^2,T)$ (see \cite{Bowen1973Topentropy, Pesinbook, TakensVerbitsky2003variationalprinciple}). Therefore, for such $\Theta:\mathbb{T}\times \mathbb{T}^2\to \mathbb{T}\times \mathbb{T}^2$, $h_{top}(F,K_{\varphi,\alpha}(\omega),\omega)$ in Theorem \ref{thm variational principle}, \ref{thm 3} and $h_{top}(F,E_\alpha(\omega),\omega)$ in Corollary \ref{corollary 2} can be replace by $h_{top}(T,K_{\varphi,\alpha}(\omega))$ and $h_{top}(T,E_\alpha(\omega))$ respectively. For example:
\begin{corollary}\label{corollary 3}
  Let $\varphi\in C(\mathbb{T}^2\times \mathbb{T})$ and $\alpha\in L_\varphi$. Then for Lebesgue-a.s. $\omega\in\mathbb{T}$,
  \begin{align*}
    h_{top}(T,K_{\varphi,\alpha}(\omega))&=\sup\{h_\mu(F):\mu\in I_\Theta(\mathbb{T}\times\mathbb{T}^2,\varphi,\alpha)\}\\
    &=\sup\{h_\mu(\Theta):\mu\in I_\Theta(\mathbb{T}\times\mathbb{T}^2,\varphi,\alpha)\},
  \end{align*}where $h_\mu(F)=h_\mu(\Theta)$ is due to the Abramov-Rohlin formula $h_\mu(\Theta)=h_\mu(F)+h_P(\theta)$ and the fact $h_P(\theta)=h_{Leb}(\theta)=0$.
\end{corollary}

\subsection{Random composition of $2\times 2$ area-preserving positive matrices driven by uniquely ergodic subshift}

Let
  \begin{equation*}
    \left\{B_i=\begin{pmatrix}
            a_i & b_i  \\
            c_i & d_i  \\
          \end{pmatrix}\right\}_{1\leq i\leq k}
  \end{equation*}be $2\times2$ matrices with $a_i,b_i,c_i,d_i\in\mathbb{Z}^+$, and $|a_id_i-c_ib_i|=1$ for any $i\in\{1,...,k\}$.  Let $\Omega=\{1,...,k\}^\mathbb{Z}$ with the left shift operator $\theta$ be the symbolic dynamical system with $k$ symbols.
  For any $\omega=(...,\omega_{-1},\omega_0,\omega_1,...)\in \Omega$, we define $F_\omega=B_{\omega_0}.$ Then the skew product $\tilde{\Theta}:\Omega\times\mathbb{T}^2\rightarrow \Omega\times\mathbb{T}^2$ defined by
  \begin{equation*}
    \tilde{\Theta}(\omega, x)=\big(\theta\omega, F_\omega (x)\big)
  \end{equation*}
  is an Anosov and topological mixing on fibers system by Proposition 8.2 and Theorem 8.2 in \cite{HLL}.

  Now we let $\Omega_s\subset \Omega$ be any uniquely ergodic subshift. For instance, when $k\geq 2$, let $(\Omega_s,\theta)$ be the Arnoux-Rauzy subshift, which is uniquely ergodic and minimal (see Section 2 in \cite{Strictlyergodicsubshift}). Especially, when $k=2$, the Arnoux-Rauzy subshift is the Sturmian subshift. Then the following system is one of target system:
  \begin{equation*}
    \Theta:\Omega_s\times \mathbb{T}^2\to \Omega_s\times \mathbb{T}^2\mbox{ by }\Theta=\tilde{\Theta}|_{\Omega_s\times \mathbb{T}^2}.
  \end{equation*}
\section{Preliminary Definitions and Lemmas}\label{section preliminary}
In this section, we introduce some preliminary lemmas for systems $(\Omega\times M,\Theta)$ satisfying \ref{C1} and \ref{C2}. Note that for lemmas in Sec. \ref{section hyperbolic dynamics}-\ref{subsection specification}, the uniquely ergodic condition in \ref{C1} is not needed. When the uniquely ergodic condition \ref{C1} is present, we will employ the fact \eqref{invariant marginal P} to apply the ergodic theory of RDS.
\subsection{Hyperbolic dynamics}\label{section hyperbolic dynamics}
 The local stable and unstable manifolds are defined as the following:
\begin{align*}
  W_\epsilon^s(\omega,x) &=\{y\in M|\ d_M(F_\omega^n(x),F_\omega^n(y))\leq \epsilon\mbox{ for all }n\geq 0\},  \\
   W_\epsilon^u(\omega,x)& =\{y\in M|\ d_M(F_\omega^n(x),F_\omega^n(y))\leq \epsilon\mbox{ for all }n\leq 0\}.
\end{align*}
The following lemmas can be found in \cite{HLL}, and it is special version of Lemmas in \cite{Gund99}.
\begin{lemma}\cite[Lemma 3.1]{HLL}\label{lemma stable unstable manifolds}
For any $\lambda\in(0,\lambda_0),$ there exists $\epsilon_0>0$ such that for any $\epsilon\in(0,\epsilon_0],$ the followings hold:
\begin{enumerate}
  \item $W_\epsilon^\tau(\omega,x)$ are $C^2$ embedded discs for all $(\omega,x)\in \Omega\times M$ with $T_xW^\tau(\omega,x)=E^\tau(\omega,x)$ for $\tau=u,s$. 
  \item For $n\geq 0$ and $y\in W^s_\epsilon(\omega,x)$, $d_M(f_\omega^nx,f_\omega^ny)\leq e^{-n\lambda}d_M(x,y)$, and for $y\in W^u_\epsilon(\omega,x)$, $d_M(f_\omega^{-n}x,f_\omega^{-n}y)\leq e^{-n\lambda}d_M(x,y).$
  \item $W^s_\epsilon(\omega,x), W^u_\epsilon(\omega,x)$ vary continuously on $(\omega,x)$ in $C^1$ topology.
\end{enumerate}
\end{lemma}
\begin{lemma}\cite[Lemma 3.2]{HLL}\label{local product structure}
  For any $\epsilon\in(0,\epsilon_0]$, where $\epsilon_0$ comes from Lemma \ref{lemma stable unstable manifolds}, there is a $\delta\in(0,\epsilon)$ such that for any $x,y\in M$ with $d_M(x,y)<\delta$, $W_\epsilon^s(\omega,x)\cap W_\epsilon^u(\omega,y) $ consists of a single point.
\end{lemma}

\begin{lemma}\cite[Lemma 3.3]{HLL}\label{Lemma expansive}
  The system $\Theta$ is fiber-expansive, i.e. there exists a constant $\eta>0$ such that for any $\omega\in\Omega$, if $d_M(F_\omega^n(x),F_\omega^n(y))<\eta$ for all $n\in\mathbb{Z}$, then $x=y$.
\end{lemma}

Let $\eta>0$ be a fiber expansive constant of the system $\Theta$. The following lemma is a corollary of Lemma 3.6 in \cite{HLL}.
\begin{lemma}\label{lemma pick L epsilon}
  For any $\epsilon\in (0,\eta)$, there exists $L(\epsilon)\in\mathbb{N}$ such that for any $x,y\in M$ and $\omega\in\Omega$, we have
  \begin{equation*}
    \max_{|n|\leq L(\epsilon)}d_M(F_\omega^n(x),F_\omega^n(y))\leq \eta\mbox{ implies }d_M(x,y)<\epsilon.
  \end{equation*}
\end{lemma}

\subsection{Specification for RDS}\label{subsection specification}
We compare the fiber specification property and the specification for RDS $F$ given by Gundlach and Kifer \cite{Kifer2000specification} in this section.

Let us first recall the definition of specification for RDS $F$ in \cite{Kifer2000specification}.
  Let $\epsilon_0:\Omega\to\mathbb{R}^+$ be the expansive characteristic for RDS $F$, i.e., for any $\omega\in\Omega$
   \begin{equation*}
     \mbox{if $d_M(F_\omega^n(x),F_\omega^n(y))\leq \epsilon_0(\theta^n\omega)$ for $n\in\mathbb{Z}$, then $x=y$.}
   \end{equation*}
 The RDS $F$ has $k-$specification if for each constant $c>0$ and $P-$a.s. $\omega\in\Omega$, there exists an $\mathbb{N}$-valued random variable $L_c=L_c(\omega)\geq 1$ such that for any points $\{x_i\}_{i=0}^k\in M$ and integers $-\infty\leq b_0<a_1<b_1<\cdots a_k<b_k<a_{k+1}\leq \infty$ satisfying $a_{i+1}\geq b_i+L_c(\theta^{b_i}\omega),$ $i=0,1,...,k$ one can find $z\in M$ such that
  \begin{align}
    \max_{a_i\leq j\leq b_i}d_M(F_\omega^j(x_i),F_\omega^j(z))&\leq c\epsilon_0(\theta^j\omega)\ \ \forall i=1,...,k\label{kifer s1}\\
    d_M(F_\omega^j(x_0),F_\omega^j(z))&\leq c\epsilon_0(\theta^j\omega)\ \ \forall j\geq a_{k+1} \mbox{ and }\forall j\leq b_0.\label{kifer s2}
  \end{align}If $F$ satisfies $k-$specification for any $k\in \mathbb{N}_+$ with $L_c$ independent of $k$, then we say that $F$ has specification property.
\begin{remark}\label{remark kifer gundlach}
  In the settings of this present paper, the expansive characteristic for RDS $ F $ is a constant (see Lemma \ref{Lemma expansive}). If we define $P_\omega(t)=F_\omega^t(x_i)$ if $t\in[a_i,b_i]$ for $i\in\{1,...,k\}$, then $(\omega,\cup_{i=1}^k[a_i,b_i],P_\omega)$ is an $\omega$-specification as in Def. \ref{def fiber specification}. Therefore, the fiber specification property defined in this paper is equivalent to only \eqref{kifer s1} holds for all $\omega\in\Omega$, any $k\in\mathbb{N}_+$ and $L_c$ independent of $k$ and $\omega$.
\end{remark}
\subsection{Fiber measure theoretic entropy}\label{section measure entropy}
For Lemmas in Sec. \ref{section measure entropy}-\ref{section 3.6}, the condition \ref{C1} is present, and  \eqref{invariant marginal P} is employed.
For any $\mu\in I_\Theta(\Omega\times M)$, the classical measure theoretic entropy of dynamical system $(\Omega\times M,\Theta,\mu)$ is denoted by $h_\mu(\Theta)$.
Denote $\pi_\Omega:\Omega\times M\rightarrow\Omega$ to be the projection into the first coordinate. The conditional entropy of a finite measurable partition $\mathcal{R}$ of $\Omega\times M$ given $\sigma$-algebra $\pi_\Omega^{-1}(\mathcal{B}_P(\Omega))$ is defined by
\begin{equation*}
  H_\mu(\mathcal{R}|\pi_\Omega^{-1}(\mathcal{B}_P(\Omega)))=\int H_{\mu_\omega}(\mathcal{R}(\omega))dP(\omega),
\end{equation*}where $\omega\mapsto \mu_\omega$ is the measure disintegration of $\mu$ with respect to $P$ and $H_{\mu_\omega}(\mathcal{A})$ denotes the usual entropy of a finite measurable partition $\mathcal{A}$ of $M$.
The fiber entropy of $F$ (or the relative entropy of $\Theta$) with respect to $\mu\in I_\Theta(\Omega\times M)$ is defined by
\begin{equation*}
  h_\mu(F)=h_{\mu}^{(r)}(\Theta)=\sup_{Q}h_{\mu}(F,Q),
\end{equation*}where
\begin{equation*}
  h_{\mu}(F,Q)=\lim_{n\rightarrow\infty}\frac{1}{n}H_\mu(\bigvee_{i=0}^{n-1}(\Theta)^{-i}Q|\pi^{-1}(\mathcal{B}_P(\Omega))),
\end{equation*}and the supreme is taken over all finite measurable partitions $Q$ of $\Omega\times M$.
Note that $h_\mu(F)$ remains the same by taking the supreme only over partition of $\Omega\times M$ into finite measurable partition $Q_i$ of the form $Q_i=\Omega\times P_i$, where $\{P_i\}$ is a finite measurable partition of $M$.
\subsection{Measure Approximation}\label{section 3.4}


 The following lemma can be found in \cite{Young1990} p.535.
\begin{lemma}[Measure approximation]\label{lemma measure approxi}
 For topological dynamical system $(\Omega\times M,\Theta)$, given $\varphi\in  C(\Omega\times M,\mathbb{R})$, for any $\mu\in I_\Theta(\Omega\times M)$, and $\delta>0$, there exist $\nu\in I_\Theta(\Omega\times M)$ and $\{\nu_i\}_{i=1}^k\subset I^e_\Theta(\Omega\times M)$ with the following properties:
  \begin{enumerate}
    \item $\nu=\sum_{i=1}^{k}\lambda_i\nu_i$, where $\lambda_i>0$, $\sum_{i=1}^{k}\lambda_i=1$;
    \item $h_\nu(\Theta)\geq h_\mu(\Theta)-\delta$, which implies $h_\nu(F)\geq h_\mu(F)-\delta$ by the Abramov-Rohlin formula $h_\mu(\Theta)=h_\mu(F)+h_P(\theta)$ and $h_\nu(\Theta)=h_\nu(F)+h_P(\theta)$;
    \item $|\int_{\Omega\times M} \varphi d\mu-\int_{\Omega\times M}\varphi d\nu|<\delta.$
  \end{enumerate}
\end{lemma}

\subsection{Random Katok entropy theorem}
For $\epsilon>0$, $\delta\in(0,1)$, $\omega\in\Omega$ and $\mu\in I_\Theta(\Omega\times M)$, denotes
\begin{equation*}
  S(\omega,n,\epsilon,\delta)=\min\{S(\omega,n,\epsilon,K)|\ K\subset M,\ \mu_\omega(K)\geq 1-\delta\},
\end{equation*}where $S(\omega,n,\epsilon,K)$ denotes the smallest cardinality of any $(\omega,\epsilon,n)$-spanning set of $K$, and a subset $G\subset K$ is called an $(\omega,\epsilon,n)$-spanning set of $ K$ if for any $y\in K$, there exists a $x\in G$ such that $d_{\omega}^n(x,y)\leq  \epsilon$. The following is the random version of Katok entropy theorem. The first two equality can be found in  \cite[Theorem 3.1]{zhutwonotes} and \cite[Theorem A]{Lizhimingdelta} respectively. The third and fourth equalities are due to the fiber-expansive property, and we supply a proof in Sec. \ref{5.2} for the sake of completeness.
\begin{lemma}\label{lemma katok entropy}
   For any $\mu\in I^e_\Theta(\Omega\times M)$ and $\delta\in(0,1)$. The mapping $\omega\mapsto S(\omega,n,\epsilon,\delta)$ is measurable. Note that $h_{\mu}(F)\leq h_{top}(F)<\infty$, then
  \begin{align}\label{eq katok delta}
    h_\mu(F)&=\lim_{\epsilon\rightarrow 0}\limsup_{n\rightarrow\infty}\frac{1}{n}\log S(\omega,n,\epsilon,\delta)=\lim_{\epsilon\rightarrow 0}\liminf_{n\rightarrow\infty}\frac{1}{n}\log S(\omega,n,\epsilon,\delta) \\
    &=\limsup_{n\rightarrow\infty}\frac{1}{n}\log S(\omega,n,\eta,\delta)=\liminf_{n\rightarrow\infty}\frac{1}{n}\log S(\omega,n,\eta,\delta) \mbox{ for } P-a.s.\  \omega\in\Omega,\nonumber
  \end{align}where $\eta>0$ is a fiber-expansive constant as in Lemma \ref{Lemma expansive}.
\end{lemma}

\subsection{Upper semi-continuity of entropy map and equilibrium states}\label{section 3.6}
  A subset $Q\subset M$ is called an $(\omega,\epsilon,n)$-separated set of $M$ if for any two different points $x,y\in Q$, $d_\omega^n(x,y)>\epsilon$.
For any $\varphi\in C(\Omega\times M)$, $n\geq 1$ and constant $\epsilon>0$, set
\begin{equation}\label{eq pi F varphi}
  \pi_F(\varphi)(\omega,\epsilon,n)=\sup\{Z_n(\omega,\varphi,Q):\ Q\mbox{ is an maximal $(\omega,\epsilon,n)$-separated set of $M$}\},
\end{equation}where
\begin{equation*}
  Z_n(\omega,\varphi,Q)=\sum_{x\in Q}\exp(S_n\varphi(\omega,x))=\sum_{x\in Q}\exp\left(\sum_{i=0}^{n-1}\varphi(\Theta^i(\omega,x))\right).
\end{equation*} The fiber topological pressure of $F$ (or the relative topological pressure of $\Theta$) with respect to $\varphi\in C(\Omega\times M,\mathbb{R})$ is
\begin{equation}\label{eq fiber top pressure}
  \pi_F(\varphi)=\lim_{\epsilon\to0}\limsup_{n\to \infty}\frac{1}{n}\int\log\pi_F(\varphi)(\omega,\epsilon,n)dP(\omega).
\end{equation}
The following lemma is a direct corollary of \cite[Proposition 1.2.6]{KL06} by noticing that $P$ is an ergodic measure with respect to $\theta$.
\begin{lemma}\label{lemma fiber top pressure}
  For any $\varphi\in  C(\Omega\times M,\mathbb{R})$,
  \begin{align*}
    \pi_F(\varphi)&=\lim_{\epsilon\to 0}\limsup_{n\to \infty}\frac{1}{n}\log\pi_F(\varphi)(\omega,\epsilon,n)\\
    &=\lim_{\epsilon\to 0}\liminf_{n\to \infty}\frac{1}{n}\log\pi_F(\varphi)(\omega,\epsilon,n)\mbox{ for $P$-a.s. $\omega\in\Omega$.}
  \end{align*}
\end{lemma}

Recall that the system $\Theta$ is fiber-expansive, see Lemma \ref{Lemma expansive}, then the following lemma is a corollary of \cite[Theorem 1.3.5]{KL06}.
\begin{lemma}\label{lemma upper semicontinuity}
  The entropy map
  \begin{equation*}
    \mu \mapsto h_\mu(F)\mbox{ for }\mu\in I_\Theta(\Omega\times M)
  \end{equation*} is upper semi-continuous with respect to the weak$^*$ topology on $I_\Theta(\Omega\times M).$ In this case, the equilibrium state for any continuous $\varphi:M\to\mathbb{R}$ exists. We denote $ES_\varphi$ to be the collection  of all equilibrium states for $\varphi\in C(\Omega\times M,\mathbb{R})$.
\end{lemma}

By  Lemma \ref{lemma upper semicontinuity} and the Abramov-Rohlin formula , we obtain the upper semi-continuity of entropy map $h_\cdot(\Theta):I_\Theta(\Omega\times M)\to\mathbb{R}$ defined by $\mu\mapsto h_\mu(\Theta)$. The following lemma only needs the upper semi-continuity of entropy map $\mu\mapsto h_\mu(\Theta)$, and it is a consequence of \cite[Corollary 2]{Jenkinson01}.
\begin{lemma}\label{lemma int subset inteq}
  For any $\varphi\in  C(\Omega\times M,\mathbb{R})$, we have
  \begin{equation*}
    int\left\{\int\varphi d\mu:\ \mu\in I_\Theta(\Omega\times M)\right\}\subset \left\{\int\varphi d\mu:\ \mu\in\cup_{q\in\mathbb{R}}ES_{q\varphi}\right\},
  \end{equation*}where $int$ denotes the interior of subset of $\mathbb{R}$.
\end{lemma}



\section{Proof of Main results}\label{section proofs}
\subsection{Proof of Theorem \ref{thm random specification}}
In this section, we prove Theorem \ref{thm random specification}.  We start with the following proposition.

\begin{proposition}\label{N intersection of su}
  For any $\epsilon\in(0,\epsilon_0]$, where $\epsilon_0$ comes from Lemma \ref{lemma stable unstable manifolds}, there exists an integer $N\in\mathbb{N}$ depending on $\epsilon$, such that for any $x,y\in M$, $\omega\in\Omega$, $n\geq N$,
  \begin{equation*}
    F_\omega^n(W_\epsilon^u(\omega,x))\cap W_\epsilon^s(\theta^n\omega,y)\not=\emptyset.
  \end{equation*}
\end{proposition}
\begin{proof}[Proof of Proposition \ref{N intersection of su}]
 Given $\epsilon>0$, by Lemma \ref{local product structure}, there is $\epsilon_1\in (0,\epsilon/2)$ such that
   \begin{equation}\label{eq pick epsilon1}
     \mbox{$W_{\epsilon/2}^s(\omega,x)\cap W_{\epsilon/2}^u(\omega,y)$ consists of a single point for all $\omega\in\Omega$ whenever $d_M(x,y)<\epsilon_1$.}
   \end{equation}Applying Lemma \ref{local product structure} again, there is a $\epsilon_2\in (0,\epsilon_1/2)$ such that
   \begin{equation}\label{pick epsilon2}
     \mbox{$W_{\epsilon_1/2}^s(\omega,x)\cap W_{\epsilon_1/2}^u(\omega,y)$ consists of a single point for all $\omega\in\Omega$ whenever $d_M(x,y)<\epsilon_2$.}
   \end{equation}

  Now let $\{x_i\}_{i=1}^n$ be a $\frac{\epsilon_2}{4}$-dense subset of $M$. Then by topological mixing on fibers property, there exists a number $N\in\mathbb{N}$ such that for any $k\geq N$ and $i,j\in\{1,...,n\}$, we have
  \begin{equation}\label{intersection nonempty}
    F_\omega^k\left(B_M(x_i,\frac{\epsilon_2}{4})\right)\cap B_M(x_j,\frac{\epsilon_2}{4})\not=\emptyset\mbox{ for all }\omega\in\Omega.
  \end{equation}
  We fix this $N$. Next, we prove this $N$ is the desired number in the statement of Proposition \ref{N intersection of su}.

  For any $k\geq N$, and $x,y\in M$, since $\{x_i\}_{i=1}^n$ is $\frac{\epsilon_2}{4}$-dense, there exists $x^\prime,y^\prime\in \{x_i\}_{i=1}^n$ such that $d_M(x,x^\prime)<\frac{\epsilon_2}{4}$ and $d_M(y,y^\prime)<\frac{\epsilon_2}{4}$. Therefore, $B_M(x^\prime,\frac{\epsilon_2}{4})\subset B_M(x,\epsilon_2)$ and $B_M(y^\prime,\frac{\epsilon_2}{4})\subset B_M(y,\epsilon_2)$. Moreover, equation $\eqref{intersection nonempty}$ implies
   $F_\omega^k(B_M(x,\epsilon_2))\cap B_M(y,\epsilon_2)\not=\emptyset.$
We pick $p\in F_\omega^k(B_M(x,\epsilon_2))\cap B_M(y,\epsilon_2)$. According to $\eqref{pick epsilon2}$, there exists a point $q$ such that
  \begin{equation}\label{pick q}
   q\in W^s_{\epsilon_1/2}(\theta^k\omega,y)\cap W_{\epsilon_1/2}^u(\theta^k\omega,p).
  \end{equation} Since $q\in W^u_{\epsilon_1/2}(\theta^k\omega,p)$, we have
  \begin{equation*}
    d_M(F_{\theta^k\omega}^{-k}p,F_{\theta^k\omega}^{-k}q)\leq e^{-k\lambda}\frac{\epsilon_1}{2}< \frac{\epsilon_1}{2}.
  \end{equation*}Therefore, we have
  \begin{align*}
    d_M(x,F_{\theta^k\omega}^{-k}q)&\leq d_M(x,F_{\theta^k\omega}^{-k}p)+d_M(F_{\theta^k\omega}^{-k}p,F_{\theta^k\omega}^{-k}q)\\
    &\leq \epsilon_2+\frac{\epsilon_1}{2}< \epsilon_1.
  \end{align*}According to \eqref{eq pick epsilon1}, $z\in W^u_{\epsilon/2}(\omega,x)\cap W^s_{\epsilon/2}(\omega,F_{\theta^k\omega}^{-k}q)$ exists. Note that
  \begin{align*}
    F_\omega^k(z)&\in F_\omega^k(W^u_{\epsilon/2}(\omega,x))\cap F_\omega^k(W^s_{\epsilon/2}(\omega,F_{\theta^k\omega}^{-k}q))\subset F_\omega^k(W^u_{\epsilon/2}(\omega,x))\cap W_{\epsilon/2}^s(\theta^k\omega,q)\\
    &\overset{\eqref{pick q}}\subset F_\omega^k(W^u_{\epsilon/2}(\omega,x))\cap W_{\epsilon}^s(\theta^k\omega,y).
  \end{align*} Therefore,   $F_\omega^k(W_\epsilon^u(\omega,x))\cap W_\epsilon^s(\theta^k\omega,y)\not=\emptyset.$
\end{proof}

Now we are ready to prove Theorem \ref{thm random specification}.
 For any fixed $\epsilon>0$, we first define $N=N(\epsilon)$, the desired space of the fiber specification property corresponding to $\epsilon$. Let $\beta=\frac{1}{2} \min\{\epsilon/2,\epsilon_0\}$, where $\epsilon_0$ comes from Lemma \ref{lemma stable unstable manifolds}. Define $\gamma=\beta/8$, and let $N $ be in Proposition \ref{N intersection of su} such that for any $x,y\in M$, $n\geq N$,
\begin{equation}
  F_\omega^n(W_{\gamma}^u(\omega,x))\cap W_\gamma^s(\theta^n\omega,y)\not=\emptyset\mbox{ for any $\omega\in\Omega$.}\label{choice of N specification}
\end{equation} Moreover, we pick $N$ sufficiently large such that $e^{-N\lambda_0}\leq \frac{1}{2}$. From now on, we fix this $N$.

Next, we prove that any $N$ spaced $\omega$-specification $S_\omega=(\omega,\tau,P_\omega)$ as in Definition \ref{def fiber specification} is $(\omega,\epsilon)$-shadowed by a point in $M$. We define
\begin{equation}\label{x a 1}
  x_{a_1}=P_\omega(a_1).
\end{equation}For $k\in\{1,2,...,m\}$, we define $x_{a_k}$ inductively: once $x_{a_k}$ is defined for $k\in\{1,2,...,m-1\}$, pick
\begin{equation}\label{construction of x_a_k}
  x_{a_{k+1}}\in F_{\theta^{b_k}\omega}^{a_{k+1}-b_k}\left(W_\gamma^u(\Theta^{b_k-a_k}(\theta^{a_k}\omega,x_{a_k}))\right)\cap W_\gamma^s(\theta^{a_{k+1}}\omega,P_\omega(a_{k+1})),
\end{equation}where $x_{a_{k+1}}$ exists due to $a_{k+1}-b_k>N$ and $(\ref{choice of N specification})$.
Define $x:=F_{\theta^{a_m}\omega}^{-a_m}(x_{a_m})$, and we are going to show $x$ is $(\omega,\beta/2)-$shadowing the $\omega-$specification $S_\omega$, i.e.,
\begin{equation}\label{shadowingspe}
  d_M(F_\omega^t(x),P_\omega(t))<\beta/2\ \mbox{ for }t\in\cup_{i=1}^m I_i.
\end{equation}
For any fixed $t\in \cup_{i=1}^m I_i$, there exists a $j\in\{1,2,...,m\}$ such that $a_j\leq t\leq b_j$. Then
\begin{equation}\label{eq 4.9}
  d_M(F_\omega^t(x),P_\omega(t))\leq d_M(F_\omega^t(x),F^{t-a_j}_{\theta^{a_j}\omega}(x_{a_j}))+d_M(F^{t-a_j}_{\theta^{a_j}\omega}(x_{a_j}), P_\omega(t)).
\end{equation}By \eqref{construction of x_a_k}, $x_{a_j}\in W^s_\gamma(\theta^{a_j}\omega,P_\omega(a_j))$, and therefore,
\begin{equation}\label{estimate1}
  d_M(F^{t-a_j}_{\theta^{a_j}\omega}(x_{a_j}), P_\omega(t))=d_M(F^{t-a_j}_{\theta^{a_j}\omega}(x_{a_j}), F^{t-a_j}_{\theta^{a_j}\omega}(P_\omega(a_j)))\leq \gamma<\beta/4.
\end{equation}To estimate $ d_M(F_\omega^t(x),F^{t-a_j}_{\theta^{a_j}\omega}(x_{a_j}))$, we are going to show that
\begin{equation}\label{phixinu}
  F_\omega^{b_j}(x)\in W_{2\gamma}^u(\Theta^{b_j-a_j}(\theta^{a_j}\omega,x_{a_j}))=W^u_{2\gamma}(\theta^{b_j}\omega,F_{\theta^{a_j}\omega}^{b_j-a_j}(x_{a_j})).
\end{equation}We only consider the case that $j\in\{1,2,...,m-1\}$ since when $j=m$, it is clear that
\begin{equation*}
  F_\omega^{b_m}(x)=F_\omega^{b_m}(F_{\theta^{a_m}\omega}^{-a_m}(x_{a_m}))=F_{\theta^{a_m}\omega}^{b_m-a_m}(x_{a_m})\in W_{2\gamma}^u(\Theta^{b_m-a_m}(\theta^{a_m}\omega,x_{a_m})).
\end{equation*}
Now for $j\in\{1,2,...,m-1\}$, we show that for any $p\in\{1,2,...,m-j\}$
\begin{equation}\label{eq induction}
  F^{b_j}_\omega \circ F^{-a_{j+p}}_{\theta^{a_{j+p}}\omega}(x_{a_{j+p}})\in W_{\gamma+\gamma\cdot e^{-\lambda N}+\cdots+\gamma\cdot e^{-\lambda(p-1) N}}^u(\Theta^{b_j-a_{j}}(\theta^{a_{j}}\omega,x_{a_{j}})).
\end{equation}by using induction. Notice that \eqref{phixinu} is a corollary of \eqref{eq induction} when $p=m-j$. In fact, for $p=1$, directly by \eqref{construction of x_a_k}, we obtain
\begin{equation*}
  F_{\omega}^{b_j}\circ F_{\theta^{a_{j+1}}\omega}^{-a_{j+1}}(x_{a_{j+1}})\in W_\gamma^u(\Theta^{b_j-a_j}(\theta^{a_j}\omega,x_{a_j})).
\end{equation*}Suppose now \eqref{eq induction} holds for $p=l\in\{1,2,...,m-j-1\}$. By the construction \eqref{construction of x_a_k}, we have
\begin{equation}\label{EQ in W U2}
    F_{\omega}^{b_{j+l}}\circ F_{\theta^{a_{j+l+1}}\omega}^{-a_{j+l+1}}(x_{a_{j+l+1}})\in W_\gamma^u(\Theta^{b_{j+l}-a_{j+l}}(\theta^{a_{j+l}}\omega,x_{a_{j+l}})).
\end{equation}Equation \eqref{EQ in W U2} implies that
\begin{equation*}
 F^{b_j-b_{j+l}}_{\theta^{b_{j+l}}\omega} \circ F_{\omega}^{b_{j+l}}\circ F_{\theta^{a_{j+l+1}}\omega}^{-a_{j+l+1}}(x_{a_{j+l+1}})\in W_{\gamma\cdot e^{-\lambda l N}}^u(\Theta^{b_j-a_{j+l}}(\theta^{a_{j+l}}\omega,x_{a_{j+l}})).
\end{equation*}Therefore,
\begin{align*}
  F^{b_j}_\omega \circ F^{-a_{j+l+1}}_{\theta^{a_{j+l+1}}\omega}(x_{a_{j+l+1}})&\in W_{\gamma\cdot e^{-\lambda l N}}^u(\Theta^{-(a_{j+l}-b_j)}(\theta^{a_{j+l}}\omega,x_{a_{j+l}}))\\
  &\subset W_{\gamma+\gamma\cdot e^{-\lambda N}+\cdots+\gamma e^{-\lambda lN}}^u(\Theta^{b_j-a_{j}}(\theta^{a_{j}}\omega,x_{a_{j}})),
\end{align*}where the last subset is due to the induction step
\begin{equation*}
   F^{b_j}_\omega \circ F^{-a_{j+l}}_{\theta^{a_{j+l}}\omega}(x_{a_{j+l}})=F^{-(a_{j+l}-b_j)}_{\theta^{a_{j+l}}\omega}(x_{a_{j+l}})\in W_{\gamma+\gamma\cdot e^{-\lambda N}+\cdots+\gamma\cdot e^{-\lambda(l-1) N}}^u(\Theta^{b_j-a_{j}}(\theta^{a_{j}}\omega,x_{a_{j}})).
\end{equation*} Therefore, equation \eqref{eq induction} and in particular \eqref{phixinu} hold.
As a consequence of \eqref{phixinu}, we have
\begin{align}
  d_M(F_\omega^t(x),F^{t-a_j}_{\theta^{a_j}\omega}(x_{a_j}))&=d_M(F^{-(b_j-t)}_{\theta^{b_j}\omega}\circ F^{b_j}_\omega(x),F^{-(b_j-t)}_{\theta^{b_j}\omega}\circ F^{b_j-a_j}_{\theta^{a_j}\omega}(x_{a_j}))\nonumber\\
 &< 2\gamma=\beta/4.\label{estimate2}
\end{align}Combining \eqref{eq 4.9}, $(\ref{estimate1})$ and $(\ref{estimate2})$, we conclude
\begin{equation*}
  d_M(F^t_\omega(x),P_\omega(t))\leq \beta/4+\beta/4=\beta/2<\epsilon.
\end{equation*}The proof of the fiber specification property is complete.

Next, let us prove the RDS corresponding to $\Theta$ has the specification property given by Gundlach and Kifer for all $\omega\in\Omega$. By Remark \ref{remark kifer gundlach}, for any $\epsilon>0$, given point $\{x_i\}_{i=0}^k\subset M$ and integers $-\infty\leq b_0<a_1<b_1<\cdots a_k<b_k<a_{k+1}\leq \infty$ satisfying $a_{i+1}\geq b_i+ N(\epsilon)$, where $N(\epsilon)$ is the number $N$ in fiber specification property corresponding to $\epsilon$. We need to prove that there exists a point $z\in M$ satisfying both
 \begin{align}
    \max_{a_i\leq j\leq b_i}d_M(F_\omega^j(x_i),F_\omega^j(z))&\leq \epsilon,\ \ \forall i=1,...,k\label{kifer 1 p}\\
    d_M(F_\omega^j(x_0),F_\omega^j(z))&\leq \epsilon,\ \ \forall j\geq a_{k+1} \mbox{ and }\forall j\leq b_0.\label{kifer 2 p}
  \end{align}
 Denote $a_0:=b_0$ and $b_{k+1}:=a_{k+1}$ and define $P_\omega:\cup_{i=0}^{k+1}[a_i,b_i]\rightarrow M$ by
 \begin{equation*}
   P_\omega(t)=\begin{cases}
                 F_\omega^{b_0}(x_0), & \mbox{if } t=a_0=b_0; \\
                 F_\omega^j(x_i) , & \mbox{if } j\in[a_i,b_i]\mbox{ for }i\in\{1,...,k\}; \\
                 F_\omega^{a_{k+1}}(x_0), & \mbox{if }t=a_{k+1}=b_{k+1}.
               \end{cases}
 \end{equation*}Then $(\omega,\cup_{i=0}^{k+1}[a_i,b_i],P_\omega)$ is a $N(\epsilon)$-spaced $\omega-$specification. Therefore, by the previous proof, there exists a point $z\in M$, which is $(\omega,\epsilon)$-shadowing this $\omega$-specification. Therefore, \eqref{kifer 1 p} is proved. Following the construction \eqref{construction of x_a_k}, we have
 \begin{equation*}
   F_\omega^{a_{k+1}}(z)=F_\omega^{b_{k+1}}(z)\in W_\gamma^s(\theta^{a_{k+1}}\omega,P_\omega(a_{k+1}))= W_\gamma^s(\theta^{a_{k+1}}\omega,F_\omega^{a_{k+1}}(x_0))\subset W_\epsilon^s(\theta^{a_{k+1}}\omega,F_\omega^{a_{k+1}}(x_0)).
 \end{equation*}As a consequence, for all $j\geq a_{k+1}$, we have
 \begin{equation*}
   d_M(F_\omega^j(x_0),F_\omega^j(z))\leq e^{-\lambda(j-a_{k+1})}d_M(F_\omega^{a_{k+1}}(x_0),F_\omega^{a_{k+1}}(z))\leq \epsilon.
 \end{equation*}Similar as \eqref{phixinu},
 we have
 \begin{equation*}
   F_\omega^{b_0}(z)\in W^u_{2\gamma}(\Theta^{b_0-a_0}(\theta^{a_0}\omega,P_\omega(a_0)))=W_{2\gamma}^u(\theta^{b_0}\omega,F_\omega^{b_0}(x_0))\subset W_{\epsilon}^u(\theta^{b_0}\omega,F_\omega^{b_0}(x_0)).
 \end{equation*}Hence, if $j\leq b_0$, we have
 \begin{equation*}
   d_M(F_\omega^j(x_0),F_\omega^j(z))\leq e^{-\lambda(b_0-j)}d_M(F_\omega^{b_0}(z),F_\omega^{b_0}(x_0))\leq \epsilon.
 \end{equation*}\eqref{kifer 2 p} is proved. The proof of Theorem \ref{thm random specification} is complete.

\subsection{Proof of Theorem \ref{thm variational principle}}
In this section, we prove Theorem \ref{thm variational principle}. In order not to obscure the main structure of the proof, we address the proof of lemmas in Sec. \ref{5.3}.
\subsubsection{Upper estimates on fiber Bowen's  topological entropy of $K_{\varphi,\alpha}$}
To give upper estimates on the fiber Bowen's  topological entropy, we need an intermediate value.
Given $\varphi\in C(\Omega\times M,\mathbb{R})$, for any $\alpha\in L_\varphi$, $\delta>0$, and $n\in \mathbb{N}$, we let
\begin{equation}
  P(\alpha,\delta,n)=\left\{(\omega,x)\in\Omega\times M:\ \left|\frac{1}{n}\sum_{i=0}^{n-1}\varphi(\Theta^i(\omega,x))-\alpha\right|<\delta\right\},
\end{equation}and denote
\begin{equation*}
  P(\alpha,\delta,n,\omega)=\{x\in M:\ (\omega,x)\in P(\alpha,\delta,n)\}.
\end{equation*}Clearly, for $\alpha\in L_\varphi$, $\delta>0$ and $\omega\in\Omega_\alpha$, $P(\alpha,\delta,n,\omega)\not=\emptyset$ for sufficiently large $n$.

For any $\epsilon>0$, we let $N(\alpha,\delta,n,\epsilon,\omega)$ be the minimal number of balls $B_n(\omega,x,\epsilon)=\{y\in M:\ d_\omega^n(x,y)<\epsilon\}$ which is necessary for covering the set $P(\alpha,\delta,n,\omega)$. If $P(\alpha,\delta,n,\omega)=\emptyset$, we let $N(\alpha,\delta,n,\epsilon,\omega)=1$. It is clear that $N(\alpha,\delta,n,\epsilon,\omega)$ does not increase as $\delta$ decreases, and $N(\alpha,\delta,n,\epsilon,\omega)$ does not decrease as $\epsilon$ decreases.
We let $M(\alpha,\delta,n,\epsilon,\omega)$ be the largest cardinality of maximal $(\omega,\epsilon,n)$-separated sets of $P(\alpha,\delta,n,\omega)$. We put $M(\alpha,\delta,n,\epsilon,\omega)=1$ if $P(\alpha,\delta,n,\omega)=\emptyset.$ It is clear that for all $\omega\in\Omega$,
\begin{equation*}
  N(\alpha,\delta,n,\epsilon,\omega)\leq M(\alpha,\delta,n,\epsilon,\omega)\leq  N(\alpha,\delta,n,\frac{\epsilon}{2},\omega).
\end{equation*}
Therefore, for $\alpha\in L_\varphi$, the following limit exists
\begin{align}
  \Lambda_{\varphi,\alpha}(\omega)&= \lim_{\epsilon\rightarrow 0}\lim_{\delta\rightarrow 0}\liminf_{n\rightarrow \infty}\frac{1}{n}\log N(\alpha,\delta,n,\epsilon,\omega) \\
  &=\lim_{\epsilon\rightarrow 0}\lim_{\delta\rightarrow 0}\liminf_{n\rightarrow \infty}\frac{1}{n}\log M(\alpha,\delta,n,\epsilon,\omega)\mbox{ for all }\omega\in\Omega.\nonumber
\end{align}According to the definition of fiber topological entropy $h_{top}(F)=\pi_F(0)$ and Lemma \ref{lemma fiber top pressure}, we can see that $\Lambda_{\varphi,\alpha}(\omega)\leq h_{top}(F)<\infty$ for $\omega$ in a  $P-$full measure set, named $\Omega_\Lambda$. The next lemma is about the measurability, and its proof is addressed in Sec. \ref{5.3}.
\begin{lemma}\label{lemma infsup}
  For each $\alpha\in L_\varphi$, $\delta>0$, $\epsilon>0$ and fixed $n$, the mapping $\omega\mapsto M(\alpha,\delta,n,\epsilon,\omega)$ is measurable. 
  As a consequence, the mapping $\omega\mapsto\Lambda_{\varphi,\alpha}(\omega)$ is measurable.
  \end{lemma}
 In this section, we prove the following proposition.
\begin{proposition}\label{lemma direction less than}
  For any $\alpha\in L_\varphi$, 
  for any $\omega\in\Omega_\alpha\cap \Omega_\Lambda$, one has
  \begin{equation*}
    h_{top}(F,K_{\varphi,\alpha}(\omega),\omega)\leq \Lambda_{\varphi,\alpha}(\omega)\leq \sup\{h_\mu(F):\ \mu\in I_\Theta(\Omega\times M,\varphi,\alpha)\}.
  \end{equation*}We note that $P(\Omega_\alpha)=1$ by the definition of $L_\varphi$.
\end{proposition}
\begin{proof}[Proof of Proposition \ref{lemma direction less than}]
Given $\alpha\in L_\varphi$, we first prove $h_{top}(F,K_{\varphi,\alpha}(\omega),\omega)\leq \Lambda_{\varphi,\alpha}(\omega)$ for each $\omega\in\Omega_\alpha\cap \Omega_\Lambda$. From now on, we fix any $\omega\in\Omega_\alpha\cap \Omega_\Lambda$. For any $\delta>0$ and $k\in\mathbb{N}$, we denote
  \begin{align*}
    G(\alpha,\delta,k,\omega)&=\bigcap_{n=k}^\infty P(\alpha,\delta,n,\omega)=\bigcap_{n=k}^\infty\left\{x\in X:\ |\frac{1}{n}\sum_{i=0}^{n-1}\varphi(\Theta^i(\omega,x))-\alpha|<\delta\right\}.
  \end{align*}
  Recall that $K_{\varphi,\alpha}(\omega)=\{x\in M:\ \lim_{n\rightarrow\infty}\frac{1}{n}\sum_{i=0}^{n-1}\varphi(\Theta^i(\omega,x))=\alpha\}$.
  Therefore, for any $\delta>0$, we have
  \begin{equation*}
    K_{\varphi,\alpha}(\omega)\subset \cup_{k=1}^\infty G(\alpha,\delta,k,\omega).
  \end{equation*} If
    \begin{equation}\label{mid less}
     h_{top}(F,G(\alpha,\delta,k,\omega),\omega,\epsilon)\leq \Lambda_{\varphi,\alpha}(\omega)
    \end{equation}holds true for those $k\in\mathbb{N}$ with $G(\alpha,\delta,k,\omega)\not=\emptyset$, for sufficient small $\epsilon>0$ and $\delta>0$, then Lemma \ref{def htopZ} and Lemma \ref{lemma property htop} implies that
    \begin{align*}
        & h_{top}(F,K_{\varphi,\alpha}(\omega),\omega,\epsilon)  \\ &\leq h_{top}(F,\cup_{k=1}^\infty G(\alpha,\delta,k,\omega),\omega,\epsilon)\\
     &=\sup\{h_{top}(F,G(\alpha,\delta,k,\omega),\omega,\epsilon):\ k\in\mathbb{N}\mbox{ with }G(\alpha,\delta,k,\omega)\not=\emptyset\}\\
     &\leq \Lambda_{\varphi,\alpha}(\omega).
    \end{align*}Therefore, $h_{top}(F,K_{\varphi,\alpha}(\omega),\omega)\leq \Lambda_{\varphi,\alpha}(\omega)$. Now, we are going to prove \eqref{mid less}.

    To prove \eqref{mid less}, by \eqref{def m(z,s,omega,epsilon)}, it is sufficient to prove that for any sufficiently small $\epsilon,\delta>0$, and $s>\Lambda_{\varphi,\alpha}(\omega)$, $m(G(\alpha,\delta,k,\omega),s,\omega,\epsilon )=0$ for those $k\in\mathbb{N}$ with $G(\alpha,\delta,k,\epsilon)\not=\emptyset$. Denote $\gamma(\omega)=\frac{s-\Lambda_{\varphi,\alpha}(\omega)}{2}>0$.
Fix any $k\in\mathbb{N}$ such that $G(\alpha,\delta,k,\omega)\not=\emptyset$. Recall that $N(\alpha,\delta,n,\epsilon,\omega)$ is the minimal number of balls $B_n(\omega,x,\epsilon)$ which is necessary for covering the set $P(\alpha,\delta,n,\omega)$. When $n\geq k$, these balls covers $G(\alpha,\delta,k,\omega)$ as well since $G(\alpha,\delta,k,\omega)\subset P(\alpha,\delta,n,\omega)$ for $n\geq k$. Hence, we have
    \begin{equation}\label{mG leq Ne}
      m(G(\alpha,\delta,k,\omega),s,\omega,n,\epsilon)
      \leq N(\alpha,\delta,n,\epsilon,\omega)e^{-sn},
    \end{equation}where $ m(G(\alpha,\delta,k,\omega),s,\omega,n,\epsilon)$ is defined as in \eqref{def m(z,s,omega,N,epsilon)}. Recall that
    \begin{equation*}
      \Lambda_{\varphi,\alpha}(\omega)=\lim_{\epsilon\to 0}\lim_{\delta\to 0}\liminf_{n\to \infty}\frac{1}{n}\log N(\alpha,\delta,n,\epsilon,\omega).
    \end{equation*}There exists $\epsilon_0(\omega)>0$ and $\delta_0(\omega)>0$ such that for all $\epsilon\in(0,\epsilon_0(\omega))$ and $\delta\in(0,\delta_0(\omega))$, there exists a monotone sequence of integers $n_l=n_l(\epsilon,\delta,\gamma(\omega))\to\infty$ as $l\to \infty$ satisfying
    \begin{equation}\label{N leq exp}
      N(\alpha,\delta,n_l,\epsilon,\omega)\leq \exp(n_l(\Lambda_{\varphi,\alpha}(\omega)+\gamma(\omega))).
    \end{equation}We fix any $\epsilon\in (0,\epsilon_0(\omega))$ and $\delta\in (0,\delta_0(\omega))$, then \eqref{mG leq Ne} and \eqref{N leq exp} imply that
    \begin{equation*}
      m(G(\alpha,\delta,k,\omega),s,\omega,n_l,\epsilon)\leq  e^{n_l(\Lambda_{\varphi,\alpha}(\omega)+\gamma(\omega))}\cdot e^{-sn_l}=e^{-n_l\gamma(\omega)}\to0\mbox{ as $n_l\to \infty$.}
    \end{equation*} By \eqref{def m(z,s,omega,N,epsilon)}, we have
     \begin{equation*}
        m(G(\alpha,\delta,k,\omega),s,\omega,\epsilon)=\lim_{n\to \infty} m(G(\alpha,\delta,k,\omega),s,\omega,n,\epsilon)=\lim_{l\to \infty} m(G(\alpha,\delta,k,\omega),s,\omega,n_l,\epsilon)=0.
     \end{equation*}
     Therefore, \eqref{mid less} holds. The proof of the first inequality is complete.

    Next, we prove $\Lambda_{\varphi,\alpha}(\omega)\leq \sup\{h_\mu(F):\ \mu\in I_\Theta(\Omega\times M,\varphi,\alpha)\}$ for any $\omega\in\Omega_\alpha\cap \Omega_\Lambda.$   Fix any $\omega^*\in\Omega_\alpha\cap\Omega_\Lambda$, and we only consider the case that $\Lambda_{\varphi,\alpha}(\omega^*)>0$. It is sufficient to show that for any $r\in(0,\Lambda_{\varphi,\alpha}(\omega^*))$, there exists a measure $\mu\in I_\Theta(\Omega\times M,\varphi,\alpha)$ satisfying
     \begin{equation}\label{h mu geq lambda}
       h_\mu(F)\geq \Lambda_{\varphi,\alpha}(\omega^*)-r.
     \end{equation}

     Fix any $r\in (0,\Lambda_{\varphi,\alpha}(\omega^*))$. By the definition of $\Lambda_{\varphi,\alpha}(\omega^*)$, there exists $\epsilon_1=\epsilon_1(\omega^*)>0$ such that for any $\epsilon\in (0,\epsilon_1(\omega^*)]$, one has
    \begin{equation}\label{pick epsilon1}
      \lim_{\delta\to 0}\liminf_{n\rightarrow\infty}\frac{1}{n}\log M(\alpha,\delta,n,\epsilon,\omega^*)\geq \Lambda_{\varphi,\alpha}(\omega^*)-\frac{1}{3}r.
    \end{equation}
    For any $k\geq 1$, we can find a sequence $\delta_k(\omega^*)>0$ depending on $ \epsilon_1(\omega^*)$ with $\delta_k(\omega^*)\searrow 0$ so that
    \begin{equation}\label{eq pick deltakomegastar}
      \liminf_{n\rightarrow\infty}\frac{1}{n}\log M(\alpha,\delta_k(\omega^*),n,\epsilon_1(\omega^*),\omega^*)\geq \Lambda_{\varphi,\alpha}(\omega^*)-\frac{2}{3}r.
    \end{equation}For $k\geq 1$, we pick a strictly increasing sequence $m_k(\omega^*)\in\mathbb{N} $ depending on $\delta_k(\omega^*)$ and $\epsilon_1(\omega^*)$ such that
    \begin{equation*}
      P(\alpha,\delta_k(\omega^*),m_{k}(\omega^*),\omega^*)\not=\emptyset
    \end{equation*}    and
    \begin{equation}\label{pick of mk}
      M_k(\omega^*):=M(\alpha,\delta_k(\omega^*),m_k(\omega^*),\epsilon_1(\omega^*),\omega^*)\geq \exp\left(m_k(\omega^*)\cdot(\Lambda_{\varphi,\alpha}(\omega^*)-r)\right).
    \end{equation}Recall $M(\alpha,\delta_k(\omega^*),m_k(\omega^*),\epsilon_1(\omega^*),\omega^*)$ is the largest cardinality of a maximal $(\omega^*,\epsilon_1(\omega^*),m_k(\omega^*))$ separated set of $P(\alpha,\delta_k(\omega^*),m_k(\omega^*),\omega^*)$, and pick $C_k(\omega^*)$ to be one of such $(\omega^*,\epsilon_1(\omega^*),m_k(\omega^*))$-separated set with $\#C_k(\omega^*)=M_k(\omega^*)$. We define a sequence of probability measures $\sigma_k\in Pr(\Omega\times M)$ by
    \begin{equation}\label{eq construction of sigma k}
      \sigma_k=\frac{1}{M_k(\omega^*)}\sum_{x\in C_k(\omega^*)}\delta_{(\omega^*,x)}.
    \end{equation}
    Let
    \begin{equation*}
      \sigma_{k,\omega^*}:=\frac{1}{M_k(\omega^*)}\sum_{x\in C_k(\omega^*)}\delta_{x}\in Pr(M).
    \end{equation*}Then $\sigma_k$ has disintegration
    \begin{equation}\label{disintegration of sigma k}
      \sigma_k= \delta_{\omega^*}\times \sigma_{k,\omega^*}.
    \end{equation}
     Define $\mu_k\in Pr(\Omega\times M)$ by
    \begin{align}\label{construct measure muk}
      \mu_k&:=\frac{1}{m_k(\omega^*)}\sum_{i=0}^{m_k(\omega^*)-1}(\Theta^i)_*\sigma_k=\frac{1}{M_k(\omega^*)}\sum_{x\in C_k(\omega^*)}\frac{1}{m_k(\omega^*)}\sum_{i=0}^{m_k(\omega^*)-1}\delta_{\Theta^{i}(\omega^*,x)}.
    \end{align}Keep in mind that this $\omega^*\in\Omega_\alpha\cap\Omega_\Lambda$ is fixed throughout this proof. By the compactness of $Pr(\Omega\times M)$ with respect to the weak$^*$ topology, a subsequence of $\mu_k$ weak$^*$ converges to a limit in $Pr(\Omega\times M)$, named $\mu_{k_l}\to\mu$ as $l\to\infty.$
    \begin{claim}\label{claim Ipalpha}
      We claim that $\mu\in I_\Theta(\Omega\times M,\varphi,\alpha)$.
    \end{claim}
\begin{proof}[Proof of Claim \ref{claim Ipalpha}]
Firstly, by using Krylov-Bogolyubov type argument, $\mu$ is $\Theta-$invariant. It is left to show that $\int_{\Omega\times M}\varphi(\omega,x)d\mu(\omega,x)=\alpha$. For every $l\geq 1$, one has
\begin{align}
  \left|\int_{\Omega\times M}\varphi(\omega,x) d\mu_{k_l}-\alpha\right|&\leq \frac{1}{M_{k_l}(\omega^*)}\sum_{x\in C_{k_l}(\omega^*)}\left|\frac{1}{m_{k_l}(\omega^*)}\sum_{i=0}^{m_{k_l}(\omega)-1}\varphi(\Theta^i(\omega^*,x))-\alpha\right|\nonumber\\
  &\leq \delta_{k_l}(\omega^*)\label{use separated set}
\end{align}since $x\in C_{k_l}(\omega^*)\subset P(\alpha,\delta_{k_l}(\omega^*),m_{k_l}(\omega^*),\omega^*)$. Therefore, $\lim_{l\rightarrow\infty}\int \varphi d\mu_{k_l}=\alpha$ since $\delta_{k_l}(\omega^*)\to 0$ as $l\to\infty$. By the weak$^*$ convergence, we also have $\lim_{l\rightarrow\infty}\int\varphi d\mu_{k_l}=\int\varphi d\mu$. Hence $\int\varphi d\mu=\alpha.$
\end{proof}

Next, we prove that such $\mu$ satisfies \eqref{h mu geq lambda} by using the same strategy as the proof of \cite[Proposition 2.2]{Kifervariationalprinciple}. Note that $\omega^*$ is fixed, and $\epsilon_1(\omega^*)$ and $m_k(\omega^*)$ are positive numbers chosen satisfying \eqref{eq pick     deltakomegastar} and \eqref{pick of mk} respectively. We can choose a finite measurable partition $\mathcal{Q}=\{\Omega\times P_i\}_{i=1}^{l}$ of $\Omega\times M$ with $\max_{1\leq i\leq l}$diam$(P_i)< \epsilon_1(\omega^*)$, such that
\begin{equation}\label{boundary 0}
  \mu(\partial\mathcal{Q})=\sum_{i=1}^{l}\int_\Omega\mu_\omega(\partial P_i)dP(\omega)=0.
\end{equation}
We denote $\mathcal{P}=\{P_i\}_{i=1}^l$. Due to \eqref{disintegration of     sigma k}, we have the following equation
\begin{equation*}
  H_{\sigma_k}(\vee_{i=0}^{m_k(\omega^*)-1}(\Theta^i)^{-1}\mathcal{Q}|\pi^{-1}(\mathcal{B}_P(\Omega)))=H_{\sigma_{k,\omega^*}}(\vee_{i=0}^{m_k(\omega^*)-1}(F_{\omega^*}^i)^{-1}\mathcal{P}).
\end{equation*}We note that the measure $\sigma_{k,\omega^*}$ is weighted on $M_k(\omega^*)=M(\alpha,\delta_k(\omega^*),m_k(\omega^*),\epsilon_1(\omega^*),\omega^*)$ numbers of $(\omega^*,\epsilon_1(\omega^*),m_k(\omega^*))$-separated points and no member of $\vee_{i=0}^{m_k(\omega^*)-1}(F_{\omega^*}^i)^{-1}\mathcal{P}$ contains more than one such point. 
Therefore, we have
\begin{align}
  H_{\sigma_k}(\vee_{i=0}^{m_k(\omega^*)-1}(\Theta^i)^{-1}\mathcal{Q}|\pi^{-1}(\mathcal{B}_P(\Omega)))&=H_{\sigma_{k,\omega^*}}(\vee_{i=0}^{m_k(\omega^*)-1}(F_{\omega^*}^i)^{-1}\mathcal{P})\nonumber \\
&=\log M(\alpha,\delta_k(\omega^*),m_k(\omega^*),\epsilon_1(\omega^*),\omega^*).\label{H=logM}
\end{align}Now let $q\in\mathbb{N}$ with $1<q<m_k(\omega^*)$ and $n\in\{0,1,...,q-1\}$. Let $a(n)$ be the greatest integer smaller than $(m_k(\omega^*)-n)q^{-1}$ so that $m_k(\omega^*)=n+a(n)q+r$ with $0\leq r<q.$ Then
\begin{equation*}
  \bigvee_{i=0}^{m_k(\omega^*)-1}(\Theta^i)^{-1}\mathcal{Q}=\left(\bigvee_{i=0}^{n-1}(\Theta^i)^{-1}\mathcal{Q}\right)\vee\left(\bigvee_{j=0}^{a(n)-1}(\Theta^{n+jq})^{-1}(\vee_{i=0}^{q-1}(\Theta^i)^{-1}\mathcal{Q})\right)\vee\left(\bigvee_{i=0}^{r-1}(\Theta^{i+n+a(n)q})^{-1}\mathcal{Q}\right).
\end{equation*}Taking into account the subadditivity of conditional entropy (see, e.g., \cite[Theorem 4.3]{Waltersergodic}), we obtain
\begin{align*}
 H_{\sigma_k}(\vee_{i=0}^{m_k(\omega^*)-1}(\Theta^i)^{-1}\mathcal{Q}|\pi^{-1}(\mathcal{B}_P(\Omega)))  &\leq \sum_{j=0}^{a(n)-1}H_{(\Theta^{n+jq})_*\sigma_k}(\vee_{i=0}^{q-1}(\Theta^i)^{-1}\mathcal{Q}|\pi^{-1}(\mathcal{B}_P(\Omega)))+(n+r)\log l\\
  &\leq \sum_{j=0}^{a(n)-1}H_{(\Theta^{n+jq})_*\sigma_k}(\vee_{i=0}^{q-1}(\Theta^i)^{-1}\mathcal{Q}|\pi^{-1}(\mathcal{B}_P(\Omega)))+2q\log l.
\end{align*}Summing over $n=0,1,...,q-1$, we get
\begin{align*}
  q H_{\sigma_k}(\vee_{i=0}^{m_k(\omega^*)-1}(\Theta^i)^{-1}\mathcal{Q}|\pi^{-1}(\mathcal{B}_P(\Omega)))   & \leq \sum_{j=0}^{m_k(\omega^*)-1}H_{(\Theta^j)_*\sigma_k}(\vee_{i=0}^{q-1}(\Theta^i)^{-1}\mathcal{Q}|\pi^{-1}(\mathcal{B}_P(\Omega)))+2q^2\log l\\
  &\leq m_k(\omega^*)H_{\mu_k}(\vee_{i=0}^{q-1}\Theta^i)^{-1}\mathcal{Q}|\pi^{-1}(\mathcal{B}_P(\Omega)))+2q^2\log l.
\end{align*}The above inequality and \eqref{H=logM} imply that
\begin{equation*}
  q\log M(\alpha,\delta_k(\omega^*),m_k(\omega^*),\epsilon_1(\omega^*),\omega^*)\leq m_k(\omega^*)H_{\mu_k}(\vee_{i=0}^{q-1}\Theta^i)^{-1}\mathcal{Q}|\pi^{-1}(\mathcal{B}_P(\Omega)))+2q^2\log l.
\end{equation*}Dividing by $m_k(\omega^*)$ and sending $k\rightarrow\infty$ along the subsequence $k_l\rightarrow\infty$, we have
\begin{align}
  q\limsup_{k_l\rightarrow\infty}\frac{1}{m_{k_l}(\omega^*)}\log M(\alpha,\delta_{k_l}(\omega^*),m_{k_l}(\omega^*),\epsilon_1(\omega^*),\omega^*)&\leq \limsup_{k\rightarrow\infty}H_{\mu_{k_l}}(\vee_{i=0}^{q-1}\Theta^i)^{-1}\mathcal{Q}|\pi^{-1}(\mathcal{B}_P(\Omega)))\nonumber\\
  &\leq H_\mu(\vee_{i=0}^{q-1}(\Theta^i)^{-1}\mathcal{Q}|\pi^{-1}(\mathcal{B}_P(\Omega))),\label{qlisuoleqH}
\end{align}where the second inequality is due to \cite[Lemma 3.2 (ii)]{LedrWalter}. To apply Lemma 3.2 (ii) in \cite{LedrWalter}, we are using the equation \eqref{boundary 0} and the assumption that $\Omega$ is a compact metric space. Now
\begin{align*}
 \Lambda_{\varphi,\alpha}(\omega^*)-r  &\overset{\eqref{pick of mk}}\leq \liminf_{k\rightarrow\infty}\frac{1}{m_k(\omega^*)}\log M(\alpha,\delta_k(\omega^*),m_k(\omega^*),\epsilon_1(\omega^*),\omega^*)\\
 &\leq \limsup_{l\rightarrow\infty}\frac{1}{m_{k_l}(\omega^*)}\log M(\alpha,\delta_{k_l}(\omega^*),m_{k_l}(\omega^*),\epsilon_1(\omega^*),\omega^*)\\
 &\overset{\eqref{qlisuoleqH}}\leq \frac{1}{q}H_\mu(\vee_{i=0}^{q-1}(\Theta^i)^{-1}\mathcal{Q}|\pi^{-1}(\mathcal{B}_P(\Omega))).
\end{align*}Sending $q\rightarrow\infty$, we arrive
\begin{equation*}
  \Lambda_{\varphi,\alpha}(\omega^*)-r\leq h_\mu(F,Q)\leq h_\mu(F),
\end{equation*}i.e., \eqref{h mu geq lambda} is proved. The proof of Proposition \ref{lemma direction less than} is complete.
\end{proof}
\begin{remark}
When the external force space $(\Omega,\theta)$ is trivial, our method is a modification of method in \cite{TakensVerbitsky2003variationalprinciple}.  As Daniel Thompson in \cite{Thompson2009Variationalprinciplepressure} pointed out, the argument in the second part of proof of Theorem 4.1 in \cite{TakensVerbitsky2003variationalprinciple} need to be corrected.  We correct Takens and Verbitskiy's proof by using the following treatments: in \eqref{pick epsilon1}, we pick $\epsilon_1(\omega)$ instead of picking $\epsilon_k(\omega)\to 0$, and we use the separated set instead of centers of covering set to construct measure $\sigma_k$ in \eqref{eq construction of sigma k}.
\end{remark}

\subsubsection{Lower estimates on fiber Bowen's topological entropy of $K_{\varphi,\alpha}$}
In this subsection, we are going to show that for $P-$a.s. $\omega\in\Omega$, :
	\begin{equation}\label{eq second goal}
		h_{top}(F,K_{\varphi,\alpha}(\omega),\omega)\geq \sup\{h_{\mu}(F):\ \mu\in I_\Theta(\Omega\times M,\varphi,\alpha)\}.
	\end{equation}
Note that the fiber entropy mapping $\mu\mapsto h_\mu(F)$ is upper semi-continuous by Lemma \ref{lemma upper semicontinuity}. By Lemma \ref{lemma I alpha nonemptyset}, $I_\Theta(\Omega\times M,\varphi,\alpha)$ is closed. Hence, there exists $\mu_0\in I_\Theta(\Omega\times M,\varphi,\alpha)$ such that
    \begin{equation}\label{def mu0}
      h_{\mu_0}(F)=\sup\{h_{\mu}(F):\ \mu\in I_\Theta(\Omega\times M,\varphi,\alpha)\}\leq h_{top}(F)<\infty.
    \end{equation}Inequality \eqref{eq second goal} is trivial true if $h_{\mu_0}(F)=0$, therefore we  only consider the case that $h_{\mu_0}(F)>0$.

    To prove \eqref{eq second goal}, it is sufficient to show for any sufficiently small $\gamma\in(0,\frac{1}{10}h_{\mu_0}(F))$, there exists a $P-$full measure set $\tilde{\Omega}_\gamma$ such that
 \begin{equation}\label{eq geq h-5gamma}
   h_{top}(F,K_{\varphi,\alpha}(\omega),\omega)\geq h_{\mu_0}(F)-5\gamma,\ \forall\omega\in\tilde{\Omega}_{\gamma}
 \end{equation}
 The proof is a 4-step process.   We start with the following technical Lemma, whose proof is given in Sec. \ref{5.3}.
    \begin{lemma}\label{Lemma lower estimates}
	For any $\gamma>0$ and  positive number $\delta<\frac{1}{2}\{\gamma,\eta\}$, there exists a $P$-full measure set $\Omega_{\delta}$ such that for any $N^\prime\in\mathbb{N}$, there exists a measurable function $\hat{n}:\Omega_{\delta}\to\mathbb{N}$ satisfying for all $\omega\in\Omega_\delta$, $\hat{n}(\omega)\geq N^\prime$, $P(\alpha,4\delta,\hat{n}(\omega),\omega)\not=\emptyset$ and
\begin{equation}\label{eq 1 hatn M geq}
  		\frac{1}{\hat{n}(\omega)}\log M(\alpha,4\delta,\hat{n}(\omega),\frac{\eta}{2},\omega)\geq 
h_{\mu_0}(F)-4\gamma,
\end{equation}where $\eta>0$ is the fiber-expansive constant in Lemma \ref{Lemma expansive} and $M(\alpha,4\delta,\hat{n}(\omega),\frac{\eta}{2},\omega)$ is the largest cardinality of maximal $(\omega,\frac{\eta}{2},\hat{n}(\omega))$-separated sets of $P(\alpha,4\delta,\hat{n}(\omega),\omega)$.
    \end{lemma}
\begin{remark}\label{remark 4.2}
  In fact, the Lemma \ref{Lemma lower estimates} still holds if $\mu_0$ is replaced by any $\Theta$-invariant measure $\mu$ and $\alpha$ is replaced by $\int \varphi d\mu$ correspondingly.
\end{remark}
Given $\gamma\in (0,\frac{1}{10}h_{\mu_0}(F))$, we first construct the desired $P$-full measure set $\tilde{\Omega}_\gamma$ in \eqref{eq geq h-5gamma} by using Lemma \ref{Lemma lower estimates}. For $k\in\mathbb{N}$, by Theorem \ref{thm random specification}, there exists $m_k=m(\frac{\eta}{2^{4+k}})$ such that any $m_k$-spaced $\omega-$specification is $(\omega,\frac{\eta}{2^{4+k}})$-shadowed by a point in $M$. Let $\{\delta_k\}_{k=1}^\infty\searrow 0$ be a strictly decreasing sequence of positive numbers with $\delta_1<\frac{1}{2}\min\{\gamma,\eta\}$. By Lemma \ref{Lemma lower estimates},  there exist a $P$-full measure set $\Omega_{\delta_{k}}$, such that there exists a measurable function $\hat{n}_k:\Omega_{\delta_k}\to\mathbb{N}$ satisfying for all $\omega\in\Omega_{\delta_k}$, $\hat{n}_k(\omega)\geq 2^{m_k}$, $P(\alpha,4\delta_k,\hat{n}_k(\omega),\omega)\not=\emptyset$, and
\begin{equation}\label{constrcution Omega k}
  \frac{1}{\hat{n}_k(\omega)}\log M(\alpha,4\delta_k,\hat{n}_k(\omega),\frac{\eta}{2},\omega)\geq h_{\mu_0}(F)-4\gamma.
\end{equation}

 Note that $\hat{n}_k(\cdot):\cap_{j\geq 1}\Omega_{\delta_j}\to\mathbb{N}$ is measurable and $P(\cap_{j\geq 1}\Omega_{\delta_j})=1$. Let $\{\xi_k\}_{k=1}^\infty$ be a sequence of positive numbers strictly decreasing to 0 with $\xi_1<\frac{1}{2}$.
 Now for any $k\in\mathbb{N},$ by using Lusin's theorem, there exists a compact set  $\Omega_k\subset\cap_{j\geq 1}\Omega_{\delta_j}$ with $P(\Omega_k)\geq 1-\frac{\xi_k}{2}$, such that $\hat{n}_k(\omega)$ is continuous on $\Omega_k.$ Denote
 \begin{equation}\label{def hat n k M}
   \hat{n}_k^M=\max_{\omega\in\Omega_k}\hat{n}_k(\omega).
 \end{equation}
 By Birkhoff's Ergodic theorem, there exists a $P$-full measure set $\Omega_k^\prime,$ such that for any $\omega\in\Omega_k^\prime$, we have
\begin{equation}\label{eq lim 1-xik}
  \lim_{n\to\infty}\frac{1}{n}\sum_{i=0}^{n-1}1_{\Omega_k}(\theta^i\omega)=P(\Omega_k)\geq 1-\frac{\xi_k}{2}.
\end{equation}
We define $\tilde{\Omega}_\gamma:=\cap_{k\geq1}\Omega_k^\prime$, and in the left of this section, we are going to show \eqref{eq geq h-5gamma} by constructing Moran-like fractal on each fiber $\{\omega\}\times M$ for $\omega\in\tilde{\Omega}_\gamma$.
From now on, we fix any $\omega\in\tilde{\Omega}_\gamma$.

\emph{Step 1. Construction of intermediate sets.} We start by choosing two sequences of positive integers depending on $\omega$, named $\{L_k(\omega)\}_{k\in\mathbb{N}}$ and $\{N_k(\omega)\}_{k\in\mathbb{N}}.$

We define $\{L_k(\omega)\}_{k\in\mathbb{N}}$ inductively. For $k=1,$ by \eqref{eq lim 1-xik}, we can pick $L_1(\omega)\in\mathbb{N}$ such that
\begin{equation}\label{omega1,xi1}	\frac{1}{n}\sum_{i=0}^{n-1}1_{\Omega_1}(\theta^i\omega)\geq 1-\xi_1,\ \forall n\geq L_1(\omega).
\end{equation}
Once $L_k(\omega)$ is defined, by \eqref{eq lim 1-xik}, we can pick $L_{k+1}(\omega)$ satisfying $L_{k+1}(\omega)> L_k(\omega),$ and \begin{equation}\label{pick L k+1}
 \dfrac{1}{n}\displaystyle\sum_{i=0}^{n-1}1_{\Omega_{k+1}}(\theta^i\omega)\geq 1-\xi_{k+1},\forall n\geq L_{k+1}(\omega).
 \end{equation}
Now, let $N_1(\omega)=L_2(\omega)$, and we pick $N_k(\omega)\in\mathbb{N}$ such that
\begin{equation}\label{choose of Nk}
	\begin{cases}
		N_k(\omega)\geq \max\{2^{\hat{n}_{k+1}^M+m_{k+1}},L_{k+1}(\omega)\}, \forall k\geq 2,\\ N_{k+1}(\omega)\geq 2^{\sum_{j=1}^k(N_j(\omega)\cdot(\hat{n}_j^M+m_j)\cdot\prod_{i=j}^{k}(1-\xi_i)^{-1})},\forall k\geq 1.
	\end{cases}
\end{equation}The \eqref{choose of Nk} ensures that
\begin{align*}
  &\lim_{k\to\infty}\frac{\hat{n}_{k+1}^M+m_{k+1}}{N_k(\omega)} =0,\ \lim_{k\to \infty}\frac{\sum_{j=1}^k(N_j(\omega)\cdot(\hat{n}_j^M+m_j)\cdot\prod_{i=j}^{k}(1-\xi_i)^{-1})}{N_{k+1}(\omega)}=0.
\end{align*}

First, we construct intermediate set $D_1(\omega)$. Let $l^{0,1}(\omega)$ be the first integer $i\geq 0$ such that $\theta^i\omega\in\Omega_1.$ We denote $T_0^1(\omega)=l^{0,1}(\omega).$ Let $l_1^1(\omega)$ be the first integer $i\geq 0$ such that
\begin{displaymath}
	\theta^{T^1_0(\omega)+\hat{n}_1(\theta^{T^1_0(\omega)}\omega)+m_1+i}\omega\in\Omega_1,
\end{displaymath}
where $\hat{n}_1(\theta^{T_0^1(\omega)}\omega)$ is well defined and bounded by $\hat{n}_1^M$ since $\theta^{T_0^1(\omega)}\omega\in\Omega_1.$ Denote $T_1^1(\omega)=T_0^1(\omega)+\hat{n}_1(\theta^{T_0^1}\omega)+m_1+l_1^{1}(\omega)$. For $k\in\{1,\dots,N_1(\omega)-2\}$, suppose $T_k^1(\omega)$ is already defined and $\theta^{T_k^1(\omega)}\omega\in\Omega_1$, we let $l_{k+1}^1(\omega)$ be the first integer $i\geq 0$ such that
\begin{displaymath} \theta^{T_k^1(\omega)+\hat{n}_1(\theta^{T_k^1(\omega)}\omega)+m_1+i}\in\Omega_1
\end{displaymath}
and denote $$T_{k+1}^1(\omega)=T_k^1(\omega)+\hat{n}_1({\theta^{T_k^1}\omega})+m_1+l_{k+1}^1(\omega).$$ Finally, we define $$T_{N_1(\omega)}^1(\omega)=T_{N_1(\omega)-1}^1(\omega)+\hat{n}_1(\theta^{T^1_{N_1-1}}\omega).$$
We point out that  in orbit $\{\omega,\cdots,\theta^{T_{N_1(\omega)}^1(\omega)}\omega\}$ , there are at least $T_0^1(\omega)+l_1^1(\omega)+\cdots+l_{N_1(\omega)-1}^1(\omega)$ points not lying in $\Omega_1.$ We note that $T_{N_1(\omega)}^1(\omega)\geq N_1(\omega)\geq L_2(\omega)\geq L_1(\omega)$, therefore, by (\ref{omega1,xi1}),
\begin{equation}\label{estimate 1 bad point}
  	T_0^1(\omega)+l_1^1(\omega)+\cdots+l_{N_1(\omega)-1}^1(\omega)\leq \xi_1 T_{N_1(\omega)}^1(\omega).
\end{equation}

By our construction, we have $\theta^{T_k^1(\omega)}\omega\in\Omega_1\subset\cap_{j\geq 1}\Omega_{\delta_j}$ for $k\in\{0,\dots,N_1(\omega)-1\}$. Therefore, our construction \eqref{constrcution Omega k}, $\Omega_1\subset\cap_{j\geq 1}\Omega_{\delta_j}$ and \eqref{def hat n k M} imply that $P(\alpha,4\delta_1,\hat{n}_1(\theta^{T_k^1(\omega)}\omega),\theta^{T_k^1(\omega)}\omega)\neq \emptyset;$ $2^{m_1}\leq \hat{n}_1(\theta^{T_k^1(\omega)}\omega)\leq\hat{n}_1^M$, and
\begin{equation*}
  \frac{1}{\hat{n}_1(\theta^{T_k^1}\omega)}\log M(\alpha,4\delta_1,\hat{n}_1(\theta^{T_k^1}\omega),\frac{\eta}{2},\theta^{T_k^1}\omega)\geq h_{\mu_0}(F)-4\gamma.
\end{equation*}
Let $C_1(\theta^{T_k^1(\omega)}\omega)\subset P(\alpha,4\delta_1,\hat{n}_1(\theta^{T_k^1(\omega)}\omega),\theta^{T_k^1(\omega)}\omega)$ be a maximal $(\theta^{T_k^1(\omega)}\omega,\frac{\eta}{2},\hat{n}_1(\theta^{T_k^1(\omega)}\omega))$-separated set with $\# C_1(\theta^{T_k^1(\omega)}\omega)=M(\alpha,4\delta_1,\hat{n}_1(\theta^{T_k^1(\omega)}\omega),\frac{\eta}{2},\theta^{T_k^1(\omega)}\omega).$

Now for any $N_1(\omega)$-tuple $(x_0^1,\dots,x_{N_1(\omega)-1}^1)\in C_1(\theta^{T_0^1(\omega)}\omega)\times\cdots\times C_1(\theta^{T_{N_1(\omega)-1}^1(\omega)}\omega)$, by Theorem \ref{thm random specification}, there exists a point $y=y(x_0^1,\dots, x^1_{N_1(\omega)-1})\in M$, which is $(\omega,\frac{\eta}{2^{4+1}})$-shadowing pieces of orbits
\begin{displaymath}
	\begin{split}
&\Big\{x_0^1,F_{\theta^{T_0^1(\omega)}\omega}(x_0^1),\dots,F^{\hat{n}_1(\theta^{T_0^1(\omega)}\omega)-1}_{\theta^{T_0^1(\omega)}\omega}(x_0^1)\Big\},\\
&\dots, \Big\{x_{N_1(\omega)-1}^1,F_{\theta^{T_{N_1(\omega)-1}^1(\omega)}\omega}(x_{N_1(\omega)-1}^1),\dots,F^{\hat{n}_1(\theta^{T_{N_1(\omega)-1}^1(\omega)}\omega)-1}_{\theta^{T_{N_1(\omega)-1}^1(\omega)}\omega}(x_{N_1(\omega)-1}^1)\Big\}
	\end{split} \end{displaymath}
 with gaps $m_1+l_1^1(\omega),m_1+l_2^1(\omega),\dots,m_1+l_{N_1-1}^1(\omega)$, i.e., for any $k\in\{0,\dots,N_1(\omega)-1\}$,
 \begin{equation}\label{eq shadowingproperty y}
d_{\theta^{T_k^1(\omega)}\omega}^{\hat{n}_1(\theta^{T_k^1(\omega)}\omega)}\big(x_k^1, F_{\omega}^{T_k^1(\omega)}y\big)<\frac{\eta}{2^{4+1}}.
 \end{equation}

 Let $D_1(\omega)=\{y=y(x_0^1,\dots,x_{N_1(\omega)-1}^1):x_k^1\in C(\theta^{T_k^1(\omega)}\omega),k\in \{0,\dots,N_1(\omega)-1\}\}$, which is the first intermediate set. The following lemma says that different tuples give different shadowing points, and its proof is given in Sec. \ref{5.3}.
 \begin{lemma}\label{lemma separated metric}
 	 If $(x_0^1,\dots,x_{N_1(\omega)-1}^1), (z_0^1,\dots,z_{N_1(\omega)-1}^1)\in C_1(\theta^{T_0^1(\omega)}\omega)\times\cdots\times C_1(\theta^{T_{N_1(\omega)-1}^1(\omega)}\omega)$ are different tuples, then
 	 \begin{displaymath}
 d_{\omega}^{T_{N_1(\omega)}^1(\omega)}\left(y((x_0^1,\dots,x_{N_1(\omega)-1}^1)),y((z_0^1,\dots,z_{N_1(\omega)-1}^1))\right)> \frac{\eta}{2}-\frac{\eta}{2^{4+1}}\times 2\geq \frac{3\eta}{8}.
 	 \end{displaymath}
  Hence $\# D_1(\omega)=\prod_{i=0}^{N_1(\omega)-1}\# C_1(\theta^{T_i^1(\omega)}\omega).$
 \end{lemma}

Next, we inductively define $D_k(\theta^{T_0^k(\omega)}\omega)$ for $k\geq 2$.
Suppose now that $T_{N_k(\omega)}^k(\omega)> N_k(\omega)$ is already defined for $k\geq 1,$ let $l^{k,k+1}(\omega)$ be the first integer $i\geq 0$ such that
\begin{equation}\label{gap larger}
	\theta^{T_{N_k(\omega)}^k(\omega)+m_{k+1}+i}\omega\in\Omega_{k+1}.
\end{equation}
Denote $T_0^{k+1}(\omega)=T_{N_k(\omega)}^k(\omega)+m_{k+1}+l^{k,k+1}(\omega)$. Let $l_1^{k+1}(\omega)$ be the first integer $i\geq 0$ such that
\begin{displaymath}
	\theta^{T_0^{k+1}(\omega)+\hat{n}_{k+1}(\theta^{T_0^{k+1}(\omega)}\omega)+m_{k+1}+i}\omega\in \Omega_{k+1}
\end{displaymath}
and denote
\begin{displaymath}
	T_1^{k+1}(\omega)=T_0^{k+1}(\omega)+\hat{n}_{k+1}(\theta^{T_0^k}\omega)+m_{k+1}+l_{1}^{k+1}(\omega).
\end{displaymath}
 Once $T_{j}^{k+1}(\omega)$ is defined for $j\in\{1,\dots,N_{k+1}(\omega)-2\}$, we let $l_{j+1}^{k+1}(\omega)$ be the first integer $i\geq 0$ such that
 \begin{displaymath}
 	\theta^{T_j^{k+1}(\omega)+\hat{n}_{k+1}(\theta^{T_j^{k+1}(\omega)}\omega)+m_{k+1}+i}\omega\in\Omega_{k+1},
 \end{displaymath}
 and denote
 \begin{displaymath}
 	T_{j+1}^{k+1}(\omega)=T_j^{k+1}(\omega)+\hat{n}_{k+1}(\theta^{T_j^{k+1}}\omega)+m_{k+1}+l_{j+1}^{k+1}(\omega).
 \end{displaymath}For $j=N_{k+1}(\omega)$, we define
 \begin{equation*}
 T_{N_{k+1}(\omega)}^{k+1}(\omega)=T_{N_{k+1}(\omega)-1}^{k+1}(\omega)+\hat{n}_{k+1}(\theta^{T_{N_{k+1}(\omega)-1}^{k+1}(\omega)}\omega).
 \end{equation*}It is clear that $T_{N_{k+1}(\omega)}^{k+1}(\omega)>N_{k+1}(\omega)$ since $\hat{n}_{k+1}\geq 1$ and $m_{k+1}>0$. By construction,
  in the orbit $\{\theta^{T_{N_k(\omega)}^k(\omega)}\omega,...,\theta^{T_{N_{k+1}(\omega)-1}^{k+1}-1}\omega\}$, there are at least $l^{k,k+1}(\omega)+\sum_{i=1}^{N_{k+1}(\omega)-1}l_{i}^{k+1}(\omega)$ points not lying in $\Omega_{k+1}$ and $T_{N_k(\omega)}^k(\omega)\geq N_k(\omega)\geq L_{k+1}(\omega).$ Therefore \eqref{pick L k+1} implies
\begin{equation}\label{estimate of k bad point} \Big(l^{k,k+1}(\omega)+\sum_{i=1}^jl_{i}^{k+1}(\omega)\Big)\leq\xi_{k+1}T_{j}^{k+1}(\omega),\forall j\in\{1,\dots,N_{k+1}(\omega)-1\}.
\end{equation}
 By our construction, $\theta^{T^{k+1}_j(\omega)}\omega\in\Omega_{k+1}\subset\cap_{l\geq 1}\Omega_{\delta_l}$ for $j\in\{0,\dots,N_{k+1}(\omega)-1\}$, therefore our construction \eqref{constrcution Omega k} and \eqref{def hat n k M} implies that
$P(\alpha,4\delta_{k+1},\hat{n}_{k+1}(\theta^{T_j^{k+1}(\omega)}\omega),\theta^{T_j^{k+1}(\omega)}\omega)\neq \emptyset,$ $2^{m_{k+1}}\leq \hat{n}_{k+1}(\theta^{T_j^{k+1}(\omega)}\omega)\leq \hat{n}_{k+1}^M,$ and
\begin{equation}\label{frac 1 hat n M entropy}
  \frac{1}{\hat{n}_{k+1}(\theta^{T_j^{k+1}(\omega)}\omega)}\log M(\alpha,4\delta_{k+1},\hat{n}_{k+1}(\theta^{T_j^{k+1}(\omega)}\omega),\frac{\eta}{2},\theta^{T_j^{k+1}(\omega)}\omega)\geq h_{\mu_0}(F)-4\gamma.
\end{equation}
For $j\in\{0,1,...,N_{k+1}(\omega)-1\}$, let $C_{k+1}(\theta^{T^{k+1}_j(\omega)}\omega)$ be a maximal $(\theta^{T_j^{k+1}(\omega)}\omega,\frac{\eta}{2},\hat{n}_{k+1}(\theta^{T_j^{k+1}(\omega)}\omega))$-separated set of $P(\alpha,4\delta_{k+1},\hat{n}_{k+1}(\theta^{T_j
^{k+1}(\omega)}\omega),\theta^{T_j^{k+1}(\omega)}\omega)$ with $$\# C_{k+1}(\theta^{T_j^{k+1}(\omega)}\omega)=M(\alpha,4\delta_{k+1},\hat{n}_{k+1}(\theta^{T_j^{k+1}(\omega)}\omega),\frac{\eta}{2},\theta^{T_j^{k+1}(\omega)}\omega).$$ For any $N_{k+1}(\omega)$-tuple $(x_0^{k+1},\dots,x_{N_{k+1}(\omega)-1}^{k+1})\in C_{k+1}(\theta^{T_0^{k+1}(\omega)}\omega)\times\cdots\times C_{k+1}(\theta^{T_{N_{k+1}(\omega)-1}^{k+1}(\omega)}\omega)$, by Theorem \ref{thm random specification}, there exists a point $y=y(x_0^{k+1},\dots,x_{N_{k+1}(\omega)-1}^{k+1})\in M$, which is  $(\theta^{T_0^{k+1}(\omega)}\omega,\frac{\eta}{2^{4+k+1}})$-shadowing the pieces of orbits
\begin{displaymath}
	\begin{split}
	&\Big\{x_0^{k+1}, F_{\theta^{T_0^{k+1}(\omega)}\omega}(x_0^{k+1}),\dots,F^{\hat{n}_{k+1}(\theta^{T_0^{k+1}(\omega)}\omega)-1}_{\theta^{T_0^{k+1}(\omega)}\omega}(x_0^{k+1})\Big\},\dots,\\
&\Big\{x_{N_{k+1}(\omega)-1}^{k+1}, F_{\theta^{T_{N_{k+1}(\omega)-1}^{k+1}(\omega)}\omega}(x_{N_{k+1}(\omega)-1}^{k+1}),\dots,F^{\hat{n}_{k+1}(\theta^{T_{N_{k+1}(\omega)-1}^{k+1}(\omega)}\omega)-1}_{\theta^{T_{N_{k+1}(\omega)-1}^{k+1}(\omega)}\omega}(x_{N_{k+1}(\omega)-1}^{k+1})\Big\}
	\end{split}
\end{displaymath}
with gaps $m_{k+1}+l_1^{k+1}(\omega),\dots,m_{k+1}+l_{N_{k+1}(\omega)-1}^{k+1}(\omega)$, respectively, i.e. for any $j\in\{0,\dots,N_{k+1}(\omega)-1\},$ we have
\begin{equation} \label{eq y shadowing xk+1} d_{\theta^{T_j^{k+1}(\omega)}\omega}^{\hat{n}_{k+1}(\theta^{T_j^{k+1}(\omega)}\omega)}\Big(x_j^{k+1}, F_{\theta^{T_0^{k+1}(\omega)}\omega}^{T_{j}^{k+1}(\omega)-T_0^{k+1}(\omega)}y\Big)<\frac{\eta}{2^{4+k+1}}.
\end{equation}
We collect all such points into
\begin{equation*}
  D_{k+1}(\theta^{T_0^{k+1}(\omega)}\omega)=\{y(x_0^{k+1},\dots,x_{N_{k+1}(\omega)-1}^{k+1}):\ x_{i}^{k+1}\in C_{k+1}(\theta^{T_{i}^{k+1}(\omega)}\omega)\mbox{ for }i=1,...,N_{k+1}(\omega)-1\}.
\end{equation*}Similar as the proof of Lemma \ref{lemma separated  metric}, we have the following lemma.
\begin{lemma}\label{lemma separated metric 2}
	If $(x_0^{k+1},\dots,x_{N_{k+1}(\omega)-1}^{k+1}), (z_0^{k+1},\dots,z_{N_{k+1}(\omega)-1}^{k+1})\in\prod_{i=0}^{N_{k+1}(\omega)-1} C_{k+1}(\theta^{T_i^{k+1}(\omega)}\omega)$ are different tuples, then
	\begin{displaymath}
		d_{\theta^{T_0^{k+1}(\omega)}\omega}^{T_{N_{k+1}}^{k+1}(\omega)-T_0^{k+1}(\omega)}\left(y(x_0^{k+1},\dots,x_{N_{k+1}(\omega)-1}^{k+1}),y(z_0^{k+1},\dots,z_{N_{k+1}(\omega)-1}^{k+1})\right)> \frac{\eta}{2}-\frac{\eta}{2^{4+k+1}}\times 2.
	\end{displaymath}
Hence
\begin{equation}\label{number Dk+1}
  \# D_{k+1}(\theta^{T_0^{k+1}(\omega)}\omega)=\prod_{i=0}^{N_{k+1}(\omega)-1}\# C_{k+1}(\theta^{T_i^{k+1}(\omega)}\omega).
\end{equation}
\end{lemma}In this way, we have constructed intermediate sets $D_1(\omega)$ and $D_k(\theta^{T_0^k(\omega)}\omega)$ for $k\geq 2$.

\emph{Step 2. Construction of $H_k(\omega),$ centers of balls forming the $k$-th level of fiber Moran-like fractal.} Let $H_1(\omega)=D_1(\omega).$ Once $H_k(\omega)$ for $k\geq 1$ is constructed, we can construct $H_{k+1}(\omega)$ inductively as follows. Pick any $x\in H_k(\omega)$ and $y\in D_{k+1}(\theta^{T_0^{k+1}}\omega)$. We note that $$T_{0}^{k+1}(\omega)-T_{N_k(\omega)}^k(\omega)=m_{k+1}+l^{k,k+1}(\omega)\geq m_{k+1}$$ by construction \eqref{gap larger}. Therefore, by Theorem \ref{thm random specification}, there exists a point $z=z(x,y)\in M$, which is  $(\omega,\frac{\eta}{2^{4+k+1}})$-shadowing pieces of orbits
\begin{displaymath}
	\Big\{x,F_{\omega}(x),\dots,F_{\omega}^{T_{N_k(\omega)}^k(\omega)-1}(x)\Big\},\Big\{y,F_{\theta^{T_0^{k+1}(\omega)}\omega}(y)\dots,F_{\theta^{T_0^{k+1}(\omega)}\omega}^{T_{N_{k+1}(\omega)}^{k+1}(\omega)-T_0^{k+1}(\omega)-1}(y)\Big\}
\end{displaymath}
with the space $m_{k+1}+l^{k,k+1}(\omega),$ i.e.,
\begin{equation}\label{z shadow xy}
	d_{\omega}^{T_{N_k(\omega)}^k(\omega)}(x,z)<\frac{\eta}{2^{4+k+1}},\quad d_{\theta^{T_0^{k+1}(\omega)}\omega}^{T_{N_{k+1}(\omega)}^{k+1}(\omega)-T_0^{k+1}(\omega)}\big(y,F_{\omega}^{T_0^{k+1}(\omega)}z\big)<\frac{\eta}{2^{4+k+1}}.
\end{equation}
 Collect all these points into the set $H_{k+1}(\omega)=\{z=z(x,y):x\in H_k(\omega),y\in D_{k+1}(\theta^{T_0^{k+1}}\omega)\}.$ The proof of the following lemma is given in Sec. \ref{5.3}.

\begin{lemma}\label{lemma Hk separated}
 For any $k\in\mathbb{N}$, $x\in H_k(\omega),y,y^{\prime}\in D_{k+1}(\theta^{T_0^{k+1}}\omega),$ if $y\neq y^{\prime},$ one has
\begin{align}
&d_{\omega}^{T_{N_k(\omega)}^k(\omega)}(z(x,y),z(x,y^{\prime}))<\frac{\eta}{2^{4+k+1}}\times 2=\frac{\eta}{2^{4+k}},\label{eq 1 lemma 6.7}\\
&d_{\omega}^{T_{N_{k+1}(\omega)}^{k+1}(\omega)}(z(x,y),z(x,y^{\prime}))> \frac{\eta}{2}-\frac{\eta}{2^{4+k+1}}\times 4\geq \frac{3\eta}{8}.\label{eq 2 lemma 6.7}
\end{align}
\end{lemma}
As a consequence of Lemma \ref{lemma Hk separated}, we have
\begin{equation}\label{number Hk}
  \# H_{k+1}(\omega)=\# H_{k}(\omega)\cdot \# D_{k+1}(\theta^{T_0^{k+1}(\omega)}\omega)=\#D_1(\omega)\cdot\prod_{i=2}^{k+1}D_{i}(\theta^{T_0^i(\omega)}\omega).
\end{equation}
Note that points in $H_1(\omega)=D_1(\omega)$ are $(\omega,\frac{3\eta}{8},T_{N_1(\omega)}^1(\omega))$-separated by Lemma \ref{lemma separated  metric}. As a consequence of \eqref{eq 2 lemma 6.7}, for any $k\in\mathbb{N}$, points in $H_k(\omega)$ are $(\omega,\frac{3\eta}{8},T_{N_{k+1}(\omega)}^{k+1}(\omega))$-separated, i.e.
\begin{equation}\label{eq 3 lemma 6.7}
  d_{\omega}^{T_{N_{k}(\omega)}^{k}(\omega)}(z,z^\prime)> \frac{3\eta}{8},\ \forall z,z^\prime\in H_k(\omega).
\end{equation}

\emph{Step 3. Construction of fiber Moran-like fractal $\mathfrak{X}=\mathfrak{X}(\omega).$} For any $k\geq 1$, we let the $k$-th level of Moran-like fractal
\begin{equation}\label{def k-th mf}
  \mathfrak{X}_k(\omega)=\bigcup_{x\in H_k(\omega)}\overline{B}_{T_{N_k(\omega)}^k(\omega)}\Big(\omega,x,\frac{\eta}{2^{4+k}}\Big),
\end{equation}
where
\begin{displaymath}
	\overline{B}_{T_{N_k(\omega)}^k(\omega)}\Big(\omega,x,\frac{\eta}{2^{4+k}}\Big)=\Big\{y\in M:d_{\omega}^{T_{N_k(\omega)}^k(\omega)}(x,y)\leq\frac{\eta}{2^{4+k}}\Big\}.
\end{displaymath}
The structure the $k$-th level of Moran-like fractal is given in the next lemma, whose proof is placed in Sec. \ref{5.3}.
\begin{lemma}\label{lemma decreased Xk}
For every $k\geq 1$, the following statements hold
\begin{enumerate}
	\item for any $x,x^{\prime}\in H_{k}(\omega), x\neq x^{\prime},$ we have
$\overline{B}_{T_{N_k(\omega)}^k(\omega)}\Big(\omega,x,\frac{\eta}{2^{4+k}}\Big)\cap \overline{B}_{T_{N_k(\omega)}^k(\omega)}\Big(\omega,x^{\prime},\frac{\eta}{2^{4+k}}\Big)=\emptyset.$
	\item  if $z=z(x,y)\in H_{k+1}(\omega)$ for $x\in H_k(\omega)$ and $y\in D_{k+1}(\theta^{T_0^{k+1}(\omega)}\omega),$ then
	\begin{displaymath}
		\overline{B}_{T_{N_{k+1}(\omega)}^{k+1}(\omega)}\Big(\omega,z,\frac{\eta}{2^{4+k+1}}\Big)\subseteq \overline{B}_{T_{N_k(\omega)}^k(\omega)}\Big(\omega,x,\frac{\eta}{2^{4+k}}\Big).
	\end{displaymath}
As a consequence, $\mathfrak{X}_{k+1}(\omega)\subseteq\mathfrak{X}_k(\omega).$
\end{enumerate}
\end{lemma}

Finally, we define the Moran-like fractal to be
\begin{equation}\label{construction of hua H}
  \mathfrak{X}(\omega)=\bigcap_{k\geq 1}\mathfrak{X}_k(\omega).
\end{equation}
By the compactness of $M, \mathfrak{X}(\omega)$ is a nonempty closed subset of $M.$ The following lemma shows that the Moran-like fractal is a subset of the conditional level set of Birkhoff average, and its proof is given in Sec. \ref{5.3}.
\begin{lemma}\label{lemma contained}
The fiber fractal $\mathfrak{X}(\omega)\subseteq K_{\varphi,\alpha}(\omega)$, i.e., for any $x\in\mathfrak{X}(\omega),$ one has $$\lim_{n\to\infty}\frac{1}{n}\sum_{k=0}^{n-1}\varphi(\Theta^i(\omega,x))=\alpha.$$
\end{lemma}

\begin{remark}\label{remark final}
  According to Remark \ref{remark 4.2} and the above proof, we can construct a fractal in  $K_{\varphi,\alpha}(\omega)$ for $\alpha= \int\varphi d\mu_0$ for any $\mu_0\in I_\Theta(\Omega\times M)$ and $P$-a.s. $\omega\in\Omega$. This indicates that $L_\varphi$ coincides $\{\int\varphi d\mu:\ \mu\in I_{\Theta}(\Omega\times M)\}$.
\end{remark}

\emph{Step 4. Construct probability measure concentrated on $\mathfrak{X}(\omega)$ to apply the entropy distribution principle type argument.} For every $k\geq 1$, we define a probability measure $\mu_{k,\omega}\in Pr(M)$ by
\begin{displaymath}
	\mu_{k,\omega}=\frac{1}{\# H_k(\omega)}\sum_{x\in H_k(\omega)}\delta_x.
\end{displaymath}
 Since $\mu_{k,\omega}$ is concentrated on $H_k(\omega)$, we have $\mu_{k,\omega}(\mathfrak{X}_k(\omega))=1$ by construction \eqref{def k-th mf}. The proof of the following three lemmas are given in Sec. \ref{5.3}.
\begin{lemma}\label{lemma limit measure}
	For any continuous function $\psi:M\to\mathbb{R}$, the following limit exists
	\begin{displaymath}
		\lim_{k\to\infty}\int_M\psi(x)d\mu_{k,\omega}(x).
	\end{displaymath}
\end{lemma}

 By the Riesz representation theorem, there exists some  $\mu_{\omega}\in Pr(M)$ such that
 \begin{displaymath}
 	\int\psi d\mu_{\omega}=\lim_{k\to\infty}\int\psi d\mu_{k,\omega}, \forall \psi\in C(M).
 \end{displaymath}
  \begin{lemma}\label{X full measure}
  	$\mu_{\omega}(\mathfrak{X}(\omega))=1.$
  \end{lemma}We note that $\omega$ is fixed at the beginning of this proof. We use notation $\mu_\omega$ to denote that this measure is concentrated on $\mathfrak{X}(\omega)$ rather than a random probability measure.

  Now, with the help of the following Lemma \ref{lemma entropy distribution}, we can apply the entropy distribution principle type argument.
  \begin{lemma}\label{lemma entropy distribution}
  	There exists $N(\omega)\in\mathbb{N},$ such that for any $n\geq N(\omega),$
  \begin{equation*}
    \mbox{if $B_n\Big(\omega,x,\frac{\eta}{2^{4}}\Big)\cap\mathfrak{X}(\omega)\neq \emptyset$, then $ \mu_{\omega}\Big(B_n\Big(\omega,x,\frac{\eta}{2^4}\Big)\Big)\leq e^{-n(h_{\mu_0}(F)-5\gamma)}$}.
  \end{equation*}Recall that $\mu_0$ is the measure in \eqref{def mu0}.
  \end{lemma}
 Let   $\Gamma_{\omega}^{\eta/2^4}=\{B_{n_i}(\omega,x_i,\frac{\eta}{2^4})\}_i$ be any finite cover of $\mathfrak{X}(\omega)$ with $\min\{n_i\}\geq N(\omega)$ as in Lemma \ref{lemma entropy distribution}. Without loss of generality, we assume that $B_{n_i}(\omega,x_i,\frac{\eta}{2^4})\cap\mathfrak{X}(\omega)\neq \emptyset$ for every $i$. It follows by Lemma \ref{lemma entropy distribution} that
 \begin{displaymath}
 	\sum_ie^{-n_i(h_{\mu_0}(F)-5\gamma)}\geq \sum_{i}\mu_{\omega}\Big(B_{n_i}\Big(\omega,x_i,\frac{\eta}{2^{4}}\Big)\Big)\geq \mu_{\omega}(\mathfrak{X}(\omega))=1>0.
 \end{displaymath}
Therefore by definition \eqref{def M(Z,s,w,N,epsilon)} and \eqref{def m(z,s,omega,N,epsilon)},
\begin{displaymath}
	m(\mathfrak{X}(\omega), h_{\mu_0}(F)-5\gamma,\omega,\frac{\eta}{2^4})\geq \inf_{\Gamma_{\omega}^{\eta/2^4}}\sum_ie^{-n_i(h_{\mu_0}(F)-5\gamma)}\geq 1>0.
\end{displaymath}
Hence, by definition \eqref{def m(z,s,omega,epsilon)}, we have
\begin{displaymath}
	h_{top}(F,\mathfrak{X}(\omega),\omega,\frac{\eta}{2^4})\geq h_{\mu_0}(F)-5\gamma.
\end{displaymath}
By Lemma \ref{lemma htopZ exist} and Lemma \ref{lemma contained}, we obtain
\begin{equation}\label{eq aim 5 gamma}
   h_{\mu_0}(F)-5\gamma\leq \sup_{\epsilon>0}h_{top}(F,\mathfrak{X}(\omega),\omega,\epsilon)=h_{top}(F,\mathfrak{X}(\omega),\omega)\leq h_{top}(F,K_{\varphi,\alpha}(\omega),\omega).
\end{equation}

We have showed that \eqref{eq aim 5 gamma} holds for arbitrary $\omega\in\tilde{\Omega}_\gamma$, i.e., \eqref{eq geq h-5gamma} holds. We note that $P(\tilde{\Omega}_\gamma)=1$ since $\tilde{\Omega}_\gamma$ is the intersection of countable $P$-full measure set. Finally, we pick a countable sequence $\gamma_k\to 0,$ and let $\Omega^{\prime}=\cap_{k}\tilde{\Omega}_{\gamma_k}.$ It follows that $P(\Omega^{\prime})=1,$ and for any $\omega\in\Omega^{\prime},$
\begin{displaymath}
	h_{top}(F,K_{\varphi,\alpha}(\omega),\omega)\geq h_{\mu_0}(F)=\sup\{h_{\mu}(F):\mu\in I_\Theta(\Omega\times M,\varphi,\alpha)\}.
\end{displaymath}
The proof of Theorem \ref{thm variational principle} is complete.
\subsection{Proof of Theorem \ref{thm 3}}
We first prove \eqref{eq leq Legendre}. For any given $\alpha\in L_\varphi$, we notice that for any $q\in\mathbb{R}$
\begin{align*}
   &\ \ \sup\{h_\mu(F):\ \mu\in I_\Theta(\Omega\times M,\varphi,\alpha)\}\nonumber\\
   &=\sup\{h_\mu(F)+\int q\varphi d\mu:\ \mu\in I_\Theta(\Omega\times M,\varphi,\alpha)\}-q\alpha\nonumber\\
   &\leq \sup\{h_\mu(F)+\int q\varphi d\mu:\ \mu\in I_\Theta(\Omega\times M)\}-q\alpha\nonumber\\
   &=\pi_{F,\varphi}(q)-q\alpha.\label{pressurelesslegendre}
\end{align*}Therefore, we have
\begin{equation}\label{eq entropy less legendre}
  \sup\{h_\mu(F):\ \mu\in I_\Theta(\Omega\times M,\varphi,\alpha)\}\leq \inf_{q\in\mathbb{R}}\{\pi_{F,\varphi}(q)-q\alpha\}= \pi_{F,\varphi}^*(\alpha).
\end{equation}Then \eqref{eq leq Legendre} is a consequence of \eqref{eq entropy less legendre} and Theorem \ref{thm variational principle}.

Now, for any $\alpha\in int(L_\varphi)=int\{\int \varphi d\mu: \ \mu\in I_\Theta(\Omega\times M)\}$, by Lemma \ref{lemma int subset inteq}, there exists $q^*\in\mathbb{R}$ and $\mu^*\in ES_{q^*\varphi}$ such that $\alpha=\int \varphi d\mu^*$. Therefore,
\begin{equation}\label{eq geq pi q star}
  \begin{split}
\sup\{h_\mu(F):\ \mu\in I_\Theta(\Omega\times M,\varphi,\alpha)\}&\geq  h_{\mu^*}(F)\\
  &= h_{\mu^*}(F)+\int q^*\varphi d\mu^*-q^*\alpha\\
  &=\pi_{F,\varphi}(q^*)-q^*\alpha.
  \end{split}
\end{equation}
Combining \eqref{eq geq pi q star} and \eqref{eq entropy less legendre}, we obtain for any $\alpha\in int(L_\varphi)$,
\begin{equation}\label{pressure eq legendre}
  \sup\{h_\mu(F):\ \mu\in I_\Theta(\Omega\times M,\varphi,\alpha)\}=\pi_{F,\varphi}^*(\alpha).
\end{equation}Then \eqref{eq eq Legendre} is a consequence of \eqref{pressure eq legendre} and Theorem \ref{thm variational principle}.

\subsection{Proof of Corollary \ref{corollary eq}}
  In this proof, we need to verify that any H\"older potential $\varphi$ and our system satisfy conditions of \cite[Theorem 3.9]{Kifer2000specification}, and compute positive constants $A_\epsilon$, $B_\epsilon$ following \cite{Kifer2000specification}. To coincide with the form in \cite{Kifer2000specification}, we denote $\epsilon=c\eta$ for some $c\in(0,\frac{1}{4})$.

  For any H\"older continuous function $\varphi:\Omega\times M\to\mathbb{R}$, we need to verify wether $\varphi\in V^\pm(F)$ as in the statement of Theorem 3.9 in \cite{Kifer2000specification}. $\varphi$ is automatically measurable on $(\omega,x)$ and continuous on $x$, and
  \begin{equation*}
    \int_\Omega\sup_{x\in M}|\varphi(\omega,x)|dP(\omega)\leq \|\varphi\|_{C^0}<\infty.
  \end{equation*}For $n\in\mathbb{Z}^-$, we define the backward Bowen's metric $d_\omega^{n-1}(x,y)=\max\{d_M(F_\omega^{i}(x),F_\omega^i(y)):\ n\leq i\leq 0\}$ for $x,y\in M$. Now let us estimate
  \begin{align*}
    K_\varphi^+(\omega)=\sup\{|\sum_{i=0}^{n}\varphi(\Theta^i(\omega,x))-\sum_{i=0}^{n}\varphi(\Theta^i(\omega,y))|:\ n\in\mathbb{Z}^+,\ d_\omega^{n+1}(x,y)\leq\eta\},\\
    K_\varphi^-(\omega)=\sup\{|\sum_{i=n}^{0}\varphi(\Theta^i(\omega,x))-\sum_{i=n}^{0}\varphi(\Theta^i(\omega,y))|:\ n\in\mathbb{Z}^-,\ d_\omega^{n-1}(x,y)\leq\eta\}.
  \end{align*}where $\eta>0$ is a fiber-expansive constant in Lemma \ref{Lemma expansive}. We estimate $K_\varphi^+(\omega)$ first, and one can compute $K_\varphi^-(\omega)$ similarly. For $ m\in\mathbb{N}$, denote
  \begin{equation*}
    var_m(\varphi,F,\omega)=\sup\{|\varphi(\omega,x)-\varphi(\omega,y)|:\ d_M(F_\omega^i(x),F_\omega^i(y))\leq \eta\mbox{ for all }|i|\leq m\}.
  \end{equation*} Exactly same proof as Lemma 3.6 in \cite{HLL} implies that there exists a positive constant $C$ such that for any $\omega\in\Omega$,
  \begin{equation}\label{eq contraction}
    \mbox{if $d_M(F_\omega^i(x),F_\omega^i(y))\leq \eta\mbox{ for all }|i|\leq m$, then $d_M(x,y)\leq Ce^{-m\lambda}$}.
  \end{equation}
 Therefore, $var_m(\varphi,F,\omega)\leq C^\prime e^{-m\lambda r}$ for some positive constant $C^\prime$ uniformly for all $\omega\in\Omega$, where $r>0$ is the H\"older exponent of $\varphi$. Then for all $\omega\in\Omega$, $$K:=\sum_{m\in\mathbb{N}}\sup_{\omega\in\Omega}var_m(\varphi,F,\omega)<\infty.$$
  Now if $d_\omega^{n+1}(x,y)\leq \eta$, and $0\leq k\leq n$, then $d_M(F_{\theta^k\omega}^j(F_\omega^k(x)),F_{\theta^k\omega}^j(F_\omega^k(y)))\leq \eta$ for $|j|\leq m_k:=\min\{k,n-k\}$ and $|\varphi(\Theta^k(\omega,x))-\varphi(\Theta^k(\omega,y))|\leq var_{m_k}(\varphi,F,\theta^k\omega)$. Thus
  \begin{align*}
|\sum_{i=0}^{n}\varphi(\Theta^i(\omega,x))-\sum_{i=0}^{n}\varphi(\Theta^i(\omega,y))|&
\leq 2\sum_{m=0}^{[\frac{n}{2}]+1}\sup_{\omega\in\Omega}var_m(\varphi,F,\omega)\leq 2K<\infty.
  \end{align*}As a consequence,
  \begin{equation}\label{eq bound K var}
    \sup_{\omega\in\Omega}\max\{K_\varphi^+(\omega),K_\varphi^-(\omega)\}\leq 2K<\infty.
  \end{equation} Therefore, $\varphi\in V^\pm(F)$ by directly checking the definition of $V^\pm(F)$ on P.96 of \cite{Kifer2000specification}.

For $n\in \mathbb{N}$, we define the following metric
  \begin{equation*}
    d_\omega^{\pm n}(x,y)=\max_{-n\leq i<n}\{d_M(F_\omega^i(x),F_\omega^i(y))\}\mbox{ for }x,y\in M.
  \end{equation*}   By \eqref{eq contraction}, we have
  \begin{equation*}
    \lim_{n\rightarrow\pm\infty}\sup\{d_M(x,y):d_{\theta^n\omega}^{\pm[\frac{|n|}{2}]}(x,y)\leq \eta\}\leq \lim_{n\rightarrow\pm\infty}Ce^{-[\frac{|n|}{2}]\lambda}/\eta=0,
  \end{equation*}which implies the strong expansivity of RDS $F$ as Definition  2.4 in \cite{Kifer2000specification}.

  Remark \ref{remark kifer gundlach} and Theorem \ref{thm random specification} imply that the RDS $F$ satisfies specification with $L_c$ independent of $k$ and $\omega$. Combining with \eqref{eq bound K var}, we obtain the equation (3.15) in the statement of Theorem 3.9 in \cite{Kifer2000specification}. Hence, the unique equilibrium state exists and it has the Gibbs property.

  Next, we estimate $A_\epsilon$ and $B_\epsilon$ in \eqref{Aepsilon Bepsilon}. The definition of the following quantity can be found in \cite[Section 2]{Kifer2000specification}. Define
  \begin{equation*}
     \beta(\omega)=\sup\{\delta>0:d_\omega^{\pm L(\frac{1}{2}c\eta)}(x,y)\leq \alpha_{L(\frac{1}{2}c\eta)}^{(2c\eta)}\mbox{ provided }d_M(x,y)\leq \delta\},
  \end{equation*}where $L(\frac{1}{2}c\eta)$ is given in Lemma \ref{lemma pick L epsilon}, and
  \begin{equation*}
\alpha_{L(\frac{1}{2}c\eta)}^{(2c\eta)}=\inf_{\omega\in\Omega}\sup\{d_M(x,y):\ d_\omega^{\pm L(\frac{1}{2}c\eta)}(x,y)\leq 2c\eta\}>0
  \end{equation*}by the continuity of $F_\omega$ on $\omega$. Again, by the continuity of $F_\omega$ on $\omega$, $\beta:=\inf_{\omega\in\Omega}\beta(\omega)>0$. By Definition 3.1 and Proposition 3.2 in \cite{Kifer2000specification}, the unique equilibrium state has Gibbs property
  \begin{equation*}
   A_\epsilon:=e^{-8K}(D_c)^{-2} \leq \mu_\omega(B_n(\omega,x,\epsilon))\cdot\pi_F(\varphi)(\omega,\epsilon,n)\cdot e^{-S_n\varphi(\omega,x)}\leq B_\epsilon:= e^{8K}(D_c)^2D^{(1)}D^{(2)},
  \end{equation*}where
  \begin{align*}
 D_c&=(N_M(2c\eta))^{2L_c}e^{2K}\exp(2L_c\|\varphi\|_{C^0});\\
 D^{(1)}&=N_M(\beta),\ D^{(2)}=e^{4K}N_M(\beta),
  \end{align*}and $ N_M(\beta)$ (resp. $N_M(2c\eta)$)is the minimal number of open balls of diameter $\beta$ (resp. $2c\eta$) covering $M$.

\subsection{Proof of Corollary \ref{corollary 2}}
For any H\"older potential $\varphi$, let $\mu=\mu_\varphi$ be the equilibrium state given in Corollary \ref{corollary eq}. As a consequence of  Corollary \ref{corollary eq}, for any $\epsilon\in (0,\frac{\eta}{4})$, there exist positive constants $A_\epsilon, B_\epsilon$ such that for $P$-a.s. $\omega\in\Omega$,
  \begin{equation}\label{eq gibbs 2}
    A_\epsilon\leq \mu_\omega(B_n(\omega,x,\epsilon))\cdot\pi_F(\varphi)(\omega,\epsilon,n)\cdot e^{-\sum_{i=0}^{n-1}\varphi(\Theta^i(\omega,x)})\leq B_\epsilon\mbox{ for all }x\in M.
  \end{equation}On the one hand, \eqref{eq gibbs 2} implies that for $P-$a.s. $\omega\in\Omega$ and any $x\in M$,
  \begin{align*}
    \lim_{\epsilon\to 0}\liminf_{n\to\infty}-\frac{1}{n}\log\mu_\omega(B_n(\omega,x,\epsilon))&\geq  \lim_{\epsilon\to 0}\liminf_{n\to\infty}\frac{1}{n}\log \pi_F(\varphi)(\omega,\epsilon,n)+\liminf_{n\to\infty}-\frac{1}{n}\sum_{i=0}^{n-1}\varphi(\Theta^i(\omega,x)) \\&= \lim_{\epsilon\to 0}\liminf_{n\to\infty}\frac{1}{n}\log \pi_F(\varphi)(\omega,\epsilon,n)-\limsup_{n\to\infty}\frac{1}{n}\sum_{i=0}^{n-1}\varphi(\Theta^i(\omega,x)),
  \end{align*} and
  \begin{align*}
    \lim_{\epsilon\to 0}\limsup_{n\to\infty}-\frac{1}{n}\log\mu_\omega(B_n(\omega,x,\epsilon))&\leq  \lim_{\epsilon\to 0}\limsup_{n\to\infty}\frac{1}{n}\log \pi_F(\varphi)(\omega,\epsilon,n)+\limsup_{n\to\infty}-\frac{1}{n}\sum_{i=0}^{n-1}\varphi(\Theta^i(\omega,x))  \\&=\lim_{\epsilon\to 0}\limsup_{n\to\infty}\frac{1}{n}\log \pi_F(\varphi)(\omega,\epsilon,n)-\liminf_{n\to\infty}\frac{1}{n}\sum_{i=0}^{n-1}\varphi(\Theta^i(\omega,x)),
  \end{align*}Taking Lemma \ref{lemma fiber top pressure} and Lemma \ref{lemma local entropy formula} into account, we have
  \begin{equation*}
   \underline{h}_{\mu}(F;\omega,x)\geq \pi_F(\varphi)-\limsup_{n\to\infty}\frac{1}{n}\sum_{i=0}^{n-1}\varphi(\Theta^i(\omega,x));
  \end{equation*}and
  \begin{equation*}
     \bar{h}_{\mu}(F;\omega,x)\leq  \pi_F(\varphi)-\liminf_{n\to\infty}\frac{1}{n}\sum_{i=0}^{n-1}\varphi(\Theta^i(\omega,x)).
  \end{equation*}
Therefore, $\liminf_{n\to\infty}\frac{1}{n}\sum_{i=0}^{n-1}\varphi(\Theta^i(\omega,x))=\limsup_{n\to\infty}\frac{1}{n}\sum_{i=0}^{n-1}\varphi(\Theta^i(\omega,x))$ implies $\underbar{h}_{\mu}(F;\omega,x)=\bar{h}_{\mu}(F;\omega,x)$.
On the other hand, \eqref{eq gibbs 2} also implies that for $P-$a.s. $\omega\in\Omega$ and any $x\in M$,
\begin{align*}
        \liminf_{n\to\infty}\frac{1}{n}\sum_{i=0}^{n-1}\varphi(\Theta^i(\omega,x))&\geq   \lim_{\epsilon\to 0}\liminf_{n\to\infty}\frac{1}{n}\log \pi_F(\varphi)(\omega,\epsilon,n)+\lim_{\epsilon\to 0}\liminf_{n\to\infty}\frac{1}{n}\log\mu_\omega(B_n(\omega,x,\epsilon))\\
        &=       \lim_{\epsilon\to 0}\liminf_{n\to\infty}\frac{1}{n}\log \pi_F(\varphi)(\omega,\epsilon,n)-\lim_{\epsilon\to 0}\limsup_{n\to\infty}-\frac{1}{n}\log\mu_\omega(B_n(\omega,x,\epsilon))\\
  \end{align*} and
  \begin{align*}
   \limsup_{n\to\infty}\frac{1}{n}\sum_{i=0}^{n-1}\varphi(\Theta^i(\omega,x))&\leq    \lim_{\epsilon\to 0}\limsup_{n\to\infty}\frac{1}{n}\log \pi_F(\varphi)(\omega,\epsilon,n)+\lim_{\epsilon\to 0}\limsup_{n\to\infty}\frac{1}{n}\log\mu_\omega(B_n(\omega,x,\epsilon))\\
        &=  \lim_{\epsilon\to 0}\limsup_{n\to\infty}\frac{1}{n}\log \pi_F(\varphi)(\omega,\epsilon,n)-\lim_{\epsilon\to 0}\liminf_{n\to\infty}-\frac{1}{n}\log\mu_\omega(B_n(\omega,x,\epsilon)).
  \end{align*}
Taking Lemma \ref{lemma fiber top pressure} and Lemma \ref{lemma local entropy formula} into account, we have
\begin{equation*}
   \liminf_{n\to\infty}\frac{1}{n}\sum_{i=0}^{n-1}\varphi(\Theta^i(\omega,x))\geq \pi_F(\varphi)-\bar{h}_\mu(F;\omega,x)
\end{equation*}and
\begin{equation*}
   \limsup_{n\to\infty}\frac{1}{n}\sum_{i=0}^{n-1}\varphi(\Theta^i(\omega,x))\leq \pi_F(\varphi)-\underbar{h}_\mu(F;\omega,x).
\end{equation*}Then $\underbar{h}_{\mu}(F;\omega,x)=\bar{h}_{\mu}(F;\omega,x)$ implies $\lim_{n\to\infty}\frac{1}{n}\sum_{i=0}^{n-1}\varphi(\Theta^i(\omega,x))$ exists. In a word, for $P-$a.s. $\omega\in\Omega$, any $x\in M$
\begin{equation}\label{eq iff}
  \liminf_{n\to\infty}\frac{1}{n}\sum_{i=0}^{n-1}\varphi(\Theta^i(\omega,x))=\limsup_{n\to\infty}\frac{1}{n}\sum_{i=0}^{n-1}\varphi(\Theta^i(\omega,x)) \mbox{ iff } \underbar{h}_{\mu}(F;\omega,x)=\bar{h}_{\mu}(F;\omega,x).
\end{equation}
Notice that $\mu$ is an ergodic measure by Corollary \ref{corollary eq}, therefore by Birkhoff ergodic theorem and Lemma \ref{lemma local entropy formula}, for $\mu$-a.s. $(\omega,x)\in\Omega\times M$, one has
\begin{equation*}
  h_\mu(F)=\pi_F(\varphi)-\int \varphi d\mu.
\end{equation*}The left statements are just corollary of \eqref{eq iff}, Theorem \ref{thm variational principle} and Theorem \ref{thm 3}. The proof of Corollary \ref{corollary 2} is complete.

\section{Proof of Lemmas}\label{section proof of lemmas} Proof of all lemmas are collected in this section.
\subsection{Proof of Lemmas in section \ref{section 1}}\label{5.1}
\begin{proof}[Proof of Lemma \ref{lemma htopZ exist}]
 For any $\epsilon_1>\epsilon_2>0$, let $\Gamma_\omega^{\epsilon_2}=\{B_{n_i}(\omega,x_i,\epsilon_2)\}$ be any finite or countable covering of $Z$ with $\min\{n_i\}\geq N$. By using same $\{n_i\}$ and $\{x_i\}$ in $\Gamma_\omega^{\epsilon
 _2}$, then $\tilde{\Gamma}_\omega^{\epsilon_1}=\{B_{n_i}(\omega,x_i,\epsilon_1)\}$ is also a cover of $Z$, which implies $m(Z,s,\omega,N,\epsilon_1)\leq m(Z,s,\omega,N,\epsilon_2).$ Let $N\rightarrow \infty$, we  obtain that $    m(Z,s,\omega,\epsilon_1)\leq m(Z,s,\omega,\epsilon_2),$ which implies
$    h_{top}(F,Z,\omega,\epsilon_1)\leq  h_{top}(F,Z,\omega,\epsilon_2).$
\end{proof}
\begin{proof}[Proof of Lemma \ref{lemma Kalpha nonempty}]
Note that $\Omega_\alpha=\{\omega\in\Omega:K_{\varphi,\alpha}(\omega)\not=\emptyset\}=\pi_\Omega(K_{\varphi,\alpha}),$ where $\pi_\Omega:\Omega\times M\rightarrow \Omega$ is the projection to the first coordinate. According to the Projection theorem (see \cite[Theorem 2.12]{Hans02}) and the $\sigma$-algebra $\mathcal{B}_P(\Omega)$ on $\Omega$ is complete with respect to $P$(see section \ref{section settings}), $\pi_\Omega(K_{\varphi,\alpha})$ is $\mathcal{B}_P(\Omega)$-measurable. Therefore, $\Omega_\alpha$ is measurable.
 Since $K_{\varphi,\alpha}$ is $\Theta$-invariant, $\Omega_\alpha$ is $\theta-$invariant. By ergodicity, either $P(\Omega_\alpha)=1$ or $P(\Omega_\alpha)=0$.
\end{proof}

\begin{proof}[Proof of Lemma \ref{lemma I alpha nonemptyset}]
   For $\alpha\in L_\varphi$, then by definition, $K_{\varphi,\alpha}(\omega)\not=\emptyset$ for $\omega\in\Omega_\alpha$ with $P(\Omega_\alpha)=1$. Pick any $\omega_0\in\Omega_\alpha$ and $x_0\in K_{\varphi,\alpha}(\omega_0)$, we consider the following sequence of measures:
  \begin{equation*}
    \delta_n=\frac{1}{n}\sum_{i=0}^{n-1}\delta_{\Theta^i(\omega_0,x_0)}\in Pr(\Omega\times M).
  \end{equation*}By the compactness of $Pr(\Omega\times M)$ with respect to the weak$^*$ topology, $\delta_n$ has a weak$^*$ convergent subsequence $\{\delta_{n_k}\}_{k\in \mathbb{N}}$ converging to a measure in $Pr(\Omega\times M)$, named $\mu.$  It is clear that $\mu$ is $\Theta$-invariant. Hence $\mu\in I_\Theta(\Omega\times M).$ On the one hand, by the definition of weak$^*$ topology, we have
  \begin{equation*}
    \lim_{k\to\infty}\frac{1}{n_k}\sum_{i=0}^{n_k-1}\int_{\Omega\times M} \varphi(\omega,x)d\delta_{\Theta^i(\omega_0,x_0)}=\int_{\Omega\times M}\varphi(\omega,x) d\mu(\omega,x).
  \end{equation*}On the other hand, notice that $(\omega_0,x_0)\in K_{\varphi,\alpha}(\omega_0)$, we have
  \begin{equation*}
    \lim_{k\to\infty}\frac{1}{n_k}\sum_{i=0}^{n_k-1}\int_{\Omega\times M} \varphi(\omega,x)d\delta_{\Theta^i(\omega_0,x_0)}=\lim_{k\to\infty}\frac{1}{n_k}\sum_{i=0}^{n_k-1}\varphi(\Theta^i(\omega_0,x_0))=\alpha.
  \end{equation*}Therefore, $\mu\in I_\Theta(\Omega\times M,\varphi,\alpha).$ Hence, $I_\Theta(\Omega\times M,\varphi,\alpha)\not=\emptyset$ for $\alpha\in L_\varphi$. It is clear that $I_\Theta(\Omega\times M,\varphi,\alpha)$ is convex and closed.

  By Birkhoff ergodic theorem, for any $\mu\in I^e_\Theta(\Omega\times M)$, i.e. ergodic $\Theta$-invariant measure, we have $\int\varphi d\mu\in L_\varphi$. Note that $I^e_\Theta(\Omega\times M)$ are exactly the extremal points of $I_\Theta(\Omega\times M)$. Since $I_\Theta(\Omega\times M)$ is compact, $I^e_\Theta(\Omega\times M)$ is nonempty, and therefore, $L_\varphi$ is nonempty.
  We note that $L_\varphi\subset[-\|\varphi\|_{C^0},\|\varphi\|_{C^0}],$ hence bounded.
\end{proof}
\subsection{Proof of Lemmas in section \ref{section preliminary}}\label{5.2}
\begin{proof}[Proof of Lemma \ref{lemma katok entropy}]
 Given $\delta\in(0,1)$, we first prove that $\omega\mapsto S(\omega,n,\epsilon,\delta)$ is $\mathcal{B}_P(\Omega)$-measurable. For any integer $L\in\mathbb{N}_+$, we notice that following fact
  \begin{equation*}
    \{\omega\in\Omega:\ S(\omega,n,\epsilon,\delta)\leq L\}=Pr_\Omega\{(\omega,x_1,...,x_L):\ \mu_\omega(\cup_{i=1}^LB_{n}(\omega,x_i,\epsilon))\geq 1-\delta\},
  \end{equation*}where $Pr_\Omega:\Omega\times M^L\rightarrow\Omega$ is the projection to the first coordinate. Once
  \begin{equation*}
    \{(\omega,x_1,...,x_L):\ \mu_\omega(\cup_{i=1}^LB_{n}(\omega,x_i,\epsilon))\geq 1-\delta\}\in \mathcal{B}_P(\Omega)\otimes\mathcal{B}(M^L),
  \end{equation*}then $\{\omega\in\Omega:\ S(\omega,n,\epsilon,\delta)\leq L\}$ is $\mathcal{B}_P(\Omega)$-measurable by the Projection theorem (see, e.g., \cite[Theorem 2.12]{Hans02}). Therefore, we need to show the measurability of
  \begin{equation}\label{measurable mapping}
    (\omega,x_1,...,x_L)\mapsto \mu_\omega(\cup_{i=1}^LB_{n}(\omega,x_i,\epsilon))=\int_M1_{\cup_{i=1}^LB_{n}(\omega,x_i,\epsilon)}(y)d\mu_\omega(y).
  \end{equation}We use the same strategy as the proof of Proposition 3.3 (i) in \cite{Hans02}. Put
  \begin{align*}
    \mathcal{H}&=\left\{h:\Omega\times M^{L+1}\to \mathbb{R}:\ h\ \mbox{bounded and measurable, and}\right.\\
    &\ \ \ \ \ \ \ \ \ \ \ \ \ \ \ \ \ \ \ \ \ \left. (\omega,x_1,...,x_L)\mapsto \int_Mh(\omega,x_1,...,x_L,y)d\mu_\omega(y)\mbox{ is measurable}\right\}.
  \end{align*}Then $\mathcal{H}$ is a vector space. For any set $D=A\times B_1\times \cdots \times B_{L+1}\in\mathcal{B}_P(\Omega)\times \mathcal{B}(M)^{L+1}$, we have
   \begin{equation}\label{L+1measurable}
     (\omega,x_1,...,x_L)\mapsto \int_M1_D(\omega,x_1,...,x_L,y)d\mu_\omega(y)=1_A(\omega)\cdot 1_{B_1}(x_1)\cdots 1_{B_L}(x_L)\cdot \mu_\omega(B_{L+1}).
   \end{equation}Note that $\omega\mapsto \mu_\omega(B_{L+1})$ is measurable by Definition 3.1 and Proposition 3.6 in \cite{Hans02}. Therefore, the mapping \eqref{L+1measurable} is measurable. If $0\leq h_n\in\mathcal{H}$, $n\in\mathbb{N}$ with $h_n\uparrow h$ for some bounded $h$, then $h\in\mathcal{H}$ by monotone convergence theorem. By a monotone class argument, $\mathcal{H}$ contains all bounded $\sigma(\mathcal{B}_P(\Omega)\times \mathcal{B}(M)^{L+1})$-measurable functions, and therefore, the function $(\omega,x_1,...,x_L,y)\mapsto 1_{\cup_{i=1}^LB_{n}(\omega,x_i,\epsilon)}(y)$ lies in $\mathcal{H}$. Hence, \eqref{measurable mapping} is measurable.

  In the left, we prove the third and fourth equalities of \eqref{eq katok delta} for $P$-a.s. $\omega\in\Omega$. Since the value $\limsup_{n\rightarrow\infty}\frac{1}{n}\log S(\omega,n,\epsilon,\delta)$ do not decrease as $\epsilon\to 0$, by the first and second equalities of \eqref{eq katok delta}, which can be found in  \cite[Theorem 3.1]{zhutwonotes} and \cite[Theorem A]{Lizhimingdelta} respectively, one has
  \begin{align*}
  h_\mu(F)&=\lim_{\epsilon\to0}\limsup_{n\rightarrow\infty}\frac{1}{n}\log S(\omega,n,\epsilon,\delta)=\lim_{\epsilon\to0}\liminf_{n\rightarrow\infty}\frac{1}{n}\log S(\omega,n,\epsilon,\delta)\\
  &\geq\limsup_{n\rightarrow\infty}\frac{1}{n}\log S(\omega,n,\eta,\delta)\geq
   \liminf_{n\rightarrow\infty}\frac{1}{n}\log S(\omega,n,\eta,\delta)
  \end{align*}for $P$-a.s. $\omega\in\Omega$. Next, we prove another direction. Pick any countable sequence $\epsilon_k\to 0$, then $$h_\mu(F)=\lim_{k\to \infty}\liminf_{n\rightarrow\infty}\frac{1}{n}\log S(\omega,n,\epsilon_k,\delta)$$P-a.s.. Note that $\omega\mapsto\liminf_{n\rightarrow\infty}\frac{1}{n}\log S(\omega,n,\epsilon_k,\delta)$ is measurable for any given $k$. By Egorov's theorem, for any $\zeta\in (0,1)$, there exists a measurable set $\Omega_\zeta\subset \Omega$ with $P(\Omega_\zeta)>1-\zeta$ such that
  \begin{equation*}
    \liminf_{n\rightarrow\infty}\frac{1}{n}\log S(\omega,n,\epsilon_k,\delta)\to h_\mu(F) \mbox{ uniformly for all }\omega\in\Omega_\zeta\mbox{ as }k\to \infty.
  \end{equation*}
By using Birkhorff's ergodic theorem, there exists a $P$-full measure set $\Omega^\prime\subset\Omega$ such that
  \begin{equation}\label{Omega prime lie in zeta}
    \lim_{n\rightarrow\infty}\frac{1}{n}\sum_{i=0}^{n-1}1_{\Omega_\zeta}(\theta^i\omega)=P(\Omega_\zeta)\geq 1-\zeta\mbox{ for all }\omega\in\Omega^\prime.
  \end{equation}Besides, since $\mu\in I_\Theta^e(\Omega\times M)$, then there exists a $\theta$-invariant $P-$full measure set $\Omega_\mu\subset \Omega$ such that
  \begin{equation}\label{eq theta invariant}
    (F_\omega)_*\mu_\omega=\mu_{\theta\omega}\mbox{, for all }\omega\in\Omega_\mu.
  \end{equation}
  It is sufficient to show that for any $\gamma>0$,
  \begin{equation}\label{lininf geq hmu gamma}
    \liminf_{n\rightarrow\infty}\frac{1}{n}\log S(\omega^*,n,\eta,\delta)\geq h_\mu(F)-\gamma,\mbox{ for all $\omega^*\in\Omega^\prime\cap\Omega_\mu$.}
  \end{equation}  From now on, we fix any $\omega^*\in\Omega^\prime\cap\Omega_\mu$.
By uniformly convergence on $\Omega_\zeta$, we pick $N$ large enough such that
  \begin{equation}\label{special Omegazeta}
    \liminf_{n\rightarrow\infty}\frac{1}{n}\log S(\omega,n,\epsilon_N,\delta)\geq h_\mu(F)-\gamma\mbox{ for all }\omega\in\Omega_\zeta.
  \end{equation}By Lemma \ref{lemma pick L epsilon}, there exists $L(\epsilon_N)\in\mathbb{N}$ such that for any $k\geq L(\epsilon_N)$, if $\max_{|n|\leq k}d_M(F_\omega^n(x),F_\omega^n(y))\leq \eta$, then $d_M(x,y)<\epsilon_N$. By \eqref{Omega prime lie in zeta}, we can pick $k>L(\epsilon_N)$ such that $\theta^k\omega^*\in\Omega_\zeta$, then \eqref{special Omegazeta} implies that
  \begin{equation}\label{liminf geq hug}
     \liminf_{n\rightarrow\infty}\frac{1}{n}\log S(\theta^k\omega^*,n,\epsilon_N,\delta)\geq h_\mu(F)-\gamma.
  \end{equation}For any $n\in\mathbb{N}$, we pick any Bowen balls $\{B_{n+2k}(\omega^*,x_i,\eta)\}_{i\in I}$, which covers a set of $\mu_{\omega^*}$ measure at least $1-\delta$, i.e., there exists $A\subset \cup_{i\in I}B_{n+2k}(\omega^*,x_i,\eta)$ and $\mu_{\omega^*}(A)\geq 1-\delta.$ Then it must have that $\{B_{n}(\theta^k\omega^*,F_{\omega^*}^k(x_i),\epsilon_N)\}_{i\in I}$ covers $F_{\omega^*}^kA$. In fact, for any $y\in A$, there exists $x_i$ such that $d_{\omega^*}^{n+2k}(y,x_i)<\eta$. Therefore, by the choice of $k>L(\epsilon_N)$ and Lemma \ref{lemma pick L epsilon}, one has $d_{\theta^k\omega^*}^n(F_{\omega^*}^k(y),F_{\omega^*}^k(x_i))<\epsilon_N$. Now by \eqref{eq theta invariant}, one has $\mu_{\theta^k\omega^*}(F_{\omega^*}^kA)=\mu_{\omega^*}(A)\geq 1-\delta$. By the definition of $S(\theta^k\omega^*,n,\epsilon_N,\delta)$, one has
  \begin{equation*}
    S(\theta^k\omega^*,n,\epsilon_N,\delta)\leq S(\omega^*,n+2k,\eta,\delta)\mbox{ for all }n\in\mathbb{N}.
  \end{equation*}Taking $\log$, dividing by $n$ and sending $n\rightarrow\infty$, we obtain
  \begin{equation*}
  \liminf_{n\rightarrow\infty}\frac{1}{n}\log S(\omega^*,n,\eta,\delta)\geq   \liminf_{n\rightarrow\infty}\frac{1}{n}\log S(\theta^k\omega^*,n,\epsilon_N,\delta)
   \overset{\eqref{liminf geq hug}} \geq h_\mu(F)-\gamma,
    \end{equation*}i.e., \eqref{lininf geq hmu gamma} is proved.
\end{proof}

\subsection{Proof of Lemmas in section \ref{section proofs}}\label{5.3}
\begin{proof}[Proof of Lemma \ref{lemma infsup}]
For $\alpha\in L_\varphi$, we first show that $\omega\mapsto M(\alpha,\delta,n,\epsilon,\omega)$ is measurable for each fixed $\delta,\epsilon>0$ and $n\in\mathbb{N}$. Note that $M(\alpha,\delta,n,\epsilon,\omega)\geq 1$, we only need to check the measurability of $\{\omega: M(\alpha,\delta,n,\epsilon,\omega)\geq l\}$ for $l\geq 2$. For any integer $l\geq 2            $, we let
\begin{align*}
  P^l(\alpha,\delta,n)&=\{(\omega,x_1,x_2,...,x_l)\in \Omega\times M^l:\ x_1,...,x_l\in P(\alpha,\delta,n,\omega)\},\\
&=\left\{(\omega,x_1,x_2,...,x_l)\in \Omega\times M^l:\ \max_{1\leq i\leq l}\left\{|\frac{1}{n}\sum_{i=0}^{n-1}\varphi(\Theta^i(\omega,x_i))-\alpha|\right\}<\delta\right\}.
\end{align*}Then $P^l(\alpha,\delta,n)$ is Borel measurable subset of $\Omega\times M^l$. It is clear that
\begin{equation*}
  E^{n,\epsilon}_{l}:=\{(\omega,x_1,...,x_l)\in \Omega\times M^l|\ d_\omega^n(x_i,x_j)> \epsilon\mbox{ if }1\leq  i\not=j\leq l\}
\end{equation*}is Borel measurable subset of $\Omega\times M^l$. By the definition of $M(\alpha,\delta,n,\epsilon,\omega)$, we have
\begin{equation*}
  \{\omega: M(\alpha,\delta,n,\epsilon,,\omega)\geq l\}=Pr_\Omega(P^l(\alpha,\delta,n)\cap E_l^{n,\epsilon}),
\end{equation*}where $Pr_\Omega:\Omega\times M^l\rightarrow \Omega$ is the projection to the first coordinate. By the Projection theorem (see, e.g., \cite[Theorem 2.12]{Hans02}) and the completion of $\sigma$-algebra $\mathcal{B}_P(\Omega)$, $Pr_\Omega(P^l(\alpha,\delta,n)\cap E_l^{n,\epsilon})$ is measurable. Therefore, $\omega\mapsto M(\alpha,\delta,n,\epsilon,\omega)$ is measurable.
\end{proof}
\begin{proof}[Proof of Lemma \ref{Lemma lower estimates} ]
Fix any $\gamma$ and $\delta$ such that  $0<\gamma< \min\{1,\frac{1}{10}h_{\mu_0}(F)\}$ and  $0<\delta<\frac{1}{2}\min\{\gamma,\eta\}$.

 Step 1, we construct $\Omega_\delta$ in the statement of Lemma \ref{Lemma lower estimates}. By using Lemma \ref{lemma measure approxi}, we can approximate $\mu_0$ by an $\Theta-$invariant measure $\nu$ in the following sense:
	there exist $\nu\in I_\Theta(\Omega\times M)$ satisfying
	\begin{enumerate}
		\item $\nu=\sum_{i=1}^{k}\lambda_i\nu_i$, where $\lambda_i>0$, $\sum_{i=1}^{k}\lambda_i=1$ and $\{\nu_i\}_{i=1}^k\in I^e_\Theta(\Omega\times M)$;
		\item $h_\nu(f)\geq h_{\mu_0}(f)-\delta$;
		\item $|\int_{\Omega\times M} \varphi d\mu_0-\int_{\Omega\times M}\varphi d\nu|=|\alpha-\int_{\Omega\times M}\varphi d\nu|<\delta.$
	\end{enumerate}
	Since $\nu_i$ is ergodic for each $i\in\{1,...,k\}$, the following set has $\nu_i-$full measure
	\begin{equation*}
			\bigg\{(\omega,x)\in \Omega\times M:\lim_{n\rightarrow\infty}\frac{1}{n}\sum_{i=0}^{n-1}\varphi(\Theta^i(\omega,x))=\int\varphi d\nu_i\bigg\}.
		\end{equation*}By measure disintegration, there exists a $P$-full measure set $\Omega_{\nu_i}$ and for $\omega\in\Omega_{\nu_i}$, we have
		\begin{equation*}
			(\nu_i)_\omega\bigg(\bigg\{x\in M:\lim_{n\rightarrow\infty}\frac{1}{n}\sum_{i=0}^{n-1}\varphi(\Theta^i(\omega,x))=\int\varphi d\nu_i\bigg\}\bigg)=1.
		\end{equation*}Therefore, for $i\in\{1,...,k\}$, the following sequence of measurable functions
		\begin{equation*}
			g_N^i(\omega)=(\nu_i)_\omega\bigg(\bigg\{x\in M:\ \bigg|\frac{1}{n}\sum_{i=0}^{n-1}\varphi(\Theta^i(\omega,x))-\int\varphi d\nu_i\bigg|<\delta,\ \forall n\geq N\bigg\}\bigg)\rightarrow 1
		\end{equation*}pointwisely for all $\omega\in\cap_{i=1}^k\Omega_{\nu_i}$ as $N\rightarrow\infty.$ By using Egorov's theorem, for any $\xi>0$ sufficiently small such that
\begin{equation}\label{xi estimation}
			\xi\leq \min\left\{ \frac{\delta}{4(\|\varphi\|_{C^0}+1)}, \ \frac{1}{4}\left(1-\frac{h_{\mu_0}(F)-4\gamma}{h_{\mu_0}(F)-3\gamma}\right)\right\},
		\end{equation}
there exists a measurable set $\Omega^1\subset \cap_{i=1}^k\Omega_{\nu_i}$ satisfying  $P(\Omega^1)\geq 1-\xi/3$ and $g_N^i(\omega)\rightarrow1$ uniformly for all $\omega\in\Omega^1$ and $i\in\{1,\dots,k\}.$  We pick $N_i^1\in\mathbb{N}$ such that
		\begin{equation*}
			g_n^i(\omega)\geq 1-\gamma, \ \forall n\geq N_i^1\mbox{ and } \forall \omega\in\Omega^1.
		\end{equation*}
		
		By Lemma \ref{lemma katok entropy}, for any $i\in\{1,2,...,k\}$, we have
		\begin{equation*}
			h_{\nu_i}(F)=\lim_{n\rightarrow \infty}\frac{1}{n}\log S_i(\omega,n,\eta,\gamma)\ \mbox{for} \ P-a.s. \omega\in\Omega,
		\end{equation*}where $S_i(\omega,n,\eta,\gamma)$ is the minimal cardinality of any $(\omega,\eta,n)-$spanning set which covers a set with $(\nu_i)_{\omega}$-measure at least $1-\gamma$. Using Egorov's theorem again, one can find a set $\Omega^2\subset \Omega^1$ with $P(\Omega^2)\geq 1-2\xi/3$ and
		\begin{equation*}
			\frac{1}{n}\log S_i(\omega,n,\eta,\gamma)\rightarrow h_{\nu_i}(F)\mbox{ uniformly for $\omega\in\Omega^2$ and for all $i\in\{1,...,k\}$.}
		\end{equation*} We pick $N_i^2\in\mathbb{N}$ such that 
		\begin{equation}\label{number of span set}
			S_i(\omega,n,\eta,\gamma)\geq e^{n(h_{\nu_i}(F)-\gamma)}\mbox{ for all $\omega\in\Omega^2$, $i\in\{1,...,k\}$ and $n\geq N_i^2$ }.
		\end{equation}We denote $N_i=\max\{N_i^1,N_i^2\}.$
		For $\omega\in \Omega^2,$  we note that
		\begin{equation}\label{deviation nonempty}
			g_{n}^i(\omega)\geq 1-\gamma,\forall n\geq N_i,
		\end{equation}
		therefore
		\begin{equation}\label{def deviation set}
			\bigg\{x\in M:\bigg|\frac{1}{n}\sum_{i=0}^{n-1}\varphi(\Theta^i(\omega,x))-\int\varphi d\nu_i\bigg|<\delta,\forall n\geq N_i\bigg\}\neq\emptyset.
		\end{equation}
\begin{remark}\label{remark mid step}
			For any $\omega\in\Omega^2$, as long as $n\geq N_i$, let $C(\omega,n,\eta)$ be some maximal $(\omega,\eta,n)$-separated set in $\{x\in M:|\frac{1}{n}\sum_{i=0}^{n-1}\varphi(\Theta^i(\omega,x))-\int\varphi d\nu_i|<\delta,\forall n\geq N_i\}$ with largest cardinality, then by \eqref{deviation nonempty} and the definition of $S_i(\omega,n,\eta,\gamma)$, we have
		\begin{equation}\label{eq number C geq e}
			\# C(\omega,n,\eta)\geq S_i(\omega,n,\eta,\gamma)\overset{\eqref{number of span set}}\geq e^{n(h_{\nu_i}(F)-\gamma)},\forall \omega\in\Omega^2.
		\end{equation}
\end{remark}
	Now by Birkhoff's Ergodic theorem, there exists a $P$-full measure set $\Omega_{\delta}$, such that for any $\omega\in\Omega_{\delta}$, we have
		\begin{equation}\label{measure of omega2}
			\lim_{n\to\infty}\frac{1}{n}\sum_{i=0}^{n-1}1_{\Omega^2}(\theta^i\omega)=P(\Omega^2)\geq 1-\frac{2}{3}\xi.
		\end{equation}
	In the following, we show Lemma \ref{Lemma lower estimates} holds for all $\omega\in\Omega_\delta$.

Step 2, we construct $\hat{n}:\Omega_\delta\to\mathbb{N}$.
		For any $\omega\in{\Omega}_{\delta},$ let  $N_3(\omega)\geq 1$ be the first integer  such that for any $n\geq N_3(\omega)$
		\begin{equation}\label{proportion of good omega}		\frac{1}{n}\sum_{i=0}^{n-1}1_{\Omega^2}(\theta^i\omega)\geq 1-\xi.
		\end{equation}Then $N_3:\Omega_\delta\to\mathbb{N}$ is measurable by noticing the following fact that for any $j\in\mathbb{N}_{\geq 2}$,
\begin{align*}
  &\ \ \ \ \{\omega\in\Omega_\delta:\ N_3(\omega)=j\}\\
  &=\left(\bigcap_{n\geq j}\{\omega\in\Omega_\delta:\ \frac{1}{n}\sum_{i=0}^{n-1}1_{\Omega^2}(\theta^i\omega)\geq 1-\xi\}\right)\cap\{\omega\in\Omega_\delta: \frac{1}{j-1}\sum_{i=0}^{j-2}1_{\Omega^2}(\theta^i\omega)< 1-\xi\}.
\end{align*} By continuity of $\varphi:\Omega\times M\to\mathbb{R}$, we can pick $\epsilon\in (0,\delta/2)$ such that
\begin{equation}\label{eq pick of epsilon lemmaproof}
      \mbox{if $d_M(x,y)<\epsilon,$ then $|\varphi(\omega,x)-\varphi(\omega,y)|<\frac{\delta}{2}$.}
\end{equation}By Theorem \ref{thm random specification},	let $m=m(\epsilon/4)$ to be the space in the fiber specification property corresponding to $\epsilon/4$, i.e., one can find a point $(\omega,\epsilon/4)$-shadowing any $m(\epsilon/4)$-spaced $\omega$-specification.
		Besides, we can choose a sufficiently large integer $N_0\in\mathbb{N}$, such that for any $n\geq N_0$ and $i\in\{1,\dots,k\}$,
		\begin{equation}\label{def of ni}
				n_i=[\lambda_in]\geq N_i,\mbox{ and }
  \frac{(k-1)m(\frac{\epsilon}{4})}{n}\leq \xi.
		\end{equation}where $\{\lambda_i\}_{i=1}^k$ are given at the beginning of proof and $[\lambda_i n]$ denotes the largest integer less than or equal to $\lambda_in$.

Pick $N_4(\delta)\in\mathbb{N}$ such that for any $n\geq N_4(\delta)$, one has
\begin{equation}\label{pick N6}
  \left|n-\sum_{i=1}^{k}[\lambda_i n]\right|\leq \frac{\delta}{2(\|\varphi\|_{C^0}+1)}n.
\end{equation}
We let
\begin{equation}\label{eq pick of n}
  \bar{n}(\omega)=\max\{N_3(\omega),N_0,N_4(\delta)\}\mbox{ for }\omega\in\Omega_\delta.
\end{equation}Then $\bar{n}:\Omega_\delta\rightarrow \mathbb{N}$ is measurable since $N_3(\omega)$ is measurable. Denote $\bar{n}_i(\cdot)=[\lambda_i\bar{n}(\cdot)]:\Omega_\delta\to\mathbb{N}$ for $i=\{1,...,k\}$, which is also measurable.

For any $\omega\in\Omega_\delta$, let $l_1(\omega)$ be the firsr integer $i\geq 0$ such that
		\begin{equation*}
			\theta^i\omega\in\Omega^2,
		\end{equation*}
		such $i$ exists due to (\ref{measure of omega2}).
		Denote $l_0^{\prime}(\omega)=0,$ and let
		\begin{equation*} l_1^{\prime}(\omega)=l_0^\prime(\omega)+l_1(\omega)+\bar{n}_1(\omega)+m(\epsilon/4).
		\end{equation*}
	For $j\in\{1,2,\dots,k-2\},$ once $l_j^{\prime}(\omega)$ is defined, let $l_{j+1}(\omega)$ be the first integer $i\geq 0$ such that $\theta^{l_j^{\prime}(\omega)+i}\omega\in\Omega^2,$ and denote
		\begin{equation*} l_{j+1}^{\prime}(\omega)=l_j^{\prime}(\omega)+l_{j+1}(\omega)+\bar{n}_{j+1}(\omega)+m(\epsilon/4).
		\end{equation*}
		 Finally, let $l_k(\omega)$ be the first integer $i\geq 0$ such that $\theta^{l_{k-1}^{\prime}(\omega)+i}\omega\in\Omega^2$, and denote
		\begin{equation}\label{def of hat n}
			\hat{n}(\omega)=l_{k-1}^{\prime}(\omega)+l_k(\omega)+\bar{n}_k(\omega).
		\end{equation}
	We note that $\hat{n}(\omega)\geq \bar{n}(\omega)$ for $\omega\in\Omega_\delta$ by construction. Note that in orbits $\{\omega,\theta\omega,...,\theta^{\hat{n}(\omega)-1}\omega\}$, there are at least $l_1(\omega)+\cdots l_k(\omega)$ points not lying in $\Omega^2$.	By (\ref{proportion of good omega}) and  \eqref{eq pick of n}, we have
		\begin{equation}\label{proportion of l}
			\frac{l_1(\omega)+\dots+l_k(\omega)}{\hat{n}(\omega)}\leq \xi.
		\end{equation}
Last step, for any $\omega\in\Omega_\delta$, we construct $(\omega,\eta/2,\hat{n}(\omega))$-separated set in $P(\alpha,4\delta,\hat{n}(\omega),\omega)$ satisfying \eqref{eq 1 hatn M geq} by using orbit gluing technique.
By the construction of $l_i(\omega)$ for $i\in\{1,\dots,k\},$  $\theta^{l_{i-1}^{\prime}(\omega)+l_i(\omega)}\omega\in\Omega^2.$ By Remark \ref{remark mid step} and $\bar{n}_i(\omega)\geq N_i$, $C(\theta^{l_{i-1}^{\prime}(\omega)+l_i(\omega)}\omega,\bar{n}_i(\omega),\eta)$ is a maximal $(\theta^{l_{i-1}^{\prime}(\omega)+l_i(\omega)}\omega,\eta,\bar{n}_i(\omega))$-separated set in the deviation set
		\begin{equation*}
			\bigg\{x\in M:\bigg|\frac{1}{n}\sum_{i=0}^{n-1}\varphi(\Theta^i(\theta^{l_{i-1}^{\prime}(\omega)+l_i(\omega)}\omega,x))-\int\varphi d\nu_i\bigg|<\delta,\forall n\geq N_i\bigg\}
		\end{equation*}satisfying that $\# C(\theta^{l_{i-1}^{\prime}(\omega)+l_i(\omega)}\omega,\bar{n}_i(\omega),\eta)\geq e^{\bar{n}_i(\omega)(h_{\nu_i}(F)-\gamma)}$. For every $k$-tuple $(x_1,\dots,x_k)$ with $x_i\in C(\theta^{l_{i-1}^{\prime}(\omega)+l_i(\omega)}\omega,\bar{n}_i(\omega),\eta)$, by Theorem \ref{thm random specification}, there is a point $y=y(x_1,\dots,x_k)\in M$, which is $(\omega,\epsilon/4)$-shadowing the pieces of orbits
		\begin{equation*}
			\{x_1,F_{\theta^{l_1(\omega)}\omega}(x_1),\dots,F^{\bar{n}_1(\omega)-1}_{\theta^{l_1(\omega)}\omega}(x_1)\},\dots,\{x_k,F_{\theta^{l_{k-1}^{\prime}(\omega)+l_k(\omega)}\omega}(x_k),\dots,F_{\theta^{l_{k-1}^{\prime}(\omega)+l_k(\omega)}\omega}^{\bar{n}_k(\omega)-1}(x_k)\}
		\end{equation*}
		with gaps $m(\epsilon/4)+l_2(\omega)$,...,$m(\epsilon/4)+l_k(\omega)$, i.e.,
\begin{equation}\label{eq lemma y shadowing xi}
  d_{\theta^{l_{i-1}^\prime(\omega)+l_i(\omega)}\omega}^{\bar{n}_i(\omega)}(x_i,F_{\omega}^{l_{i-1}^\prime(\omega)+l_i(\omega)}(y))<\frac{\epsilon}{4} \mbox{ for }i=1,...,k.
\end{equation}
 For different tuples $(x_1,\dots,x_k)$, $(x_1^{\prime},\dots,x_k^{\prime})$ with $x_i,x_i^\prime\in C(\theta^{l_{i-1}^{\prime}(\omega)+l_i(\omega)}\omega,\bar{n}_i(\omega),\eta)$, we claim that
		\begin{equation}\label{eq y eta/2 separated}		d_{\omega}^{\hat{n}(\omega)}(y(x_1,\dots,x_k),y(x_1^{\prime},\dots,x_k^{\prime}))> \eta-\frac{\epsilon}{4}\times 2= \frac{\eta}{2}.
		\end{equation}In fact, if $x_i\not=x_i^\prime$ for some $i\in\{1,...,k\}$, and denote $y_1=y(x_1,...,x_k), $ $y_2=y(x_1^\prime,...,x_k^\prime)$, then we have
\begin{align*}
   & d_{\omega}^{\hat{n}(\omega)}(y(x_1,\dots,x_k),y(x_1^{\prime},\dots,x_k^{\prime}))\\
   &\geq d_{\theta^{l_{i-1}^\prime(\omega)+l_i(\omega)}\omega}^{\bar{n}_i(\omega)}(x_i,x_i^\prime)-d_{\theta^{l_{i-1}^\prime(\omega)+l_i(\omega)}\omega}^{\bar{n}_i(\omega)}(x_i,F_\omega^{l_{i-1}^\prime(\omega)+l_i(\omega)}(y_1))-
   d_{\theta^{l_{i-1}^\prime(\omega)+l_i(\omega)}\omega}^{\bar{n}_i(\omega)}(x_i^\prime,F_\omega^{l_{i-1}^\prime(\omega)+l_i(\omega)}(y_2))\\
 &\overset{\eqref{eq lemma y shadowing xi}}>\eta-\frac{\epsilon}{4}-\frac{\epsilon}{4}=\frac{\eta}{2}.
\end{align*}
We also claim that
\begin{equation}\label{eq y subset P}
  \begin{split}
  &\ \ \ \ \{y=y(x_1,...,x_k):\ x_i\in C(\theta^{l_{i-1}^{\prime}(\omega)+l_i(\omega)}\omega,\bar{n}_i(\omega),\eta), i=1,...,k \}\\
  &\subset P(\alpha,4\delta,\hat{n}(\omega),\omega)=\left\{x\in M:\ \left|\frac{1}{\hat{n}(\omega)}\sum_{i=0}^{\hat{n}(\omega)-1}\varphi(\Theta^i(\omega,x))-\alpha\right|<4\delta\right\}.
  \end{split}
\end{equation}	
Next, we prove \eqref{eq   y subset P} for any fixed $\omega\in\Omega_\delta$. Since $\omega$ is fixed, we omit $\omega$ in $l_i(\omega)$, $l_i^\prime(\omega)$, $\hat{n}(\omega)$ and $\bar{n}_i(\omega)$ for short. On interval $\cup_{i=1}^k[l_{i-1}^\prime+l_i,l_{i-1}^\prime+l_i+\bar{n}_i-1]$, we use shadowing property \eqref{eq lemma y shadowing xi} and triangle inequality; while on other intervals, we note that
\begin{align*}
  &\#\left([0,\hat{n}-1]\backslash\cup_{i=1}^k[l_{i-1}^\prime+l_i,l_{i-1}^\prime+l_i+\bar{n}_i-1]\right)\\
  =&\#\left(\left(\cup_{i=1}^k[l_{i-1}^\prime,l_{i-1}^\prime+l_i-1]\right)\cup\left(\cup_{i=1}^{k-1}[l_{i-1}^\prime+l_i+\bar{n}_i,l_{i-1}^\prime+l_i+\bar{n}_i+m(\frac{\epsilon}{4})-1]\right)\right)\\
  =&\sum_{i=1}^{k}l_i+(k-1)m(\frac{\epsilon}{4}),
\end{align*}and we use inequality $|\varphi-\alpha|\leq 2\|\varphi\|_{C^0}$ to obtain
\begin{align*}
   & \ \ \ \bigg|\sum_{p=0}^{\hat{n}-1}\varphi(\Theta^p(\omega,y))-\hat{n}\alpha\bigg|\label{esimation 1}\\
   &\leq \left|\sum_{i=1}^{k}\left(\sum_{p=l_{i-1}^\prime+l_i}^{l_{i-1}^\prime+l_i+\bar{n}_i-1}\varphi(\Theta^p(\omega,y))-\bar{n}_i\alpha\right)\right|\\
   &\quad\quad\quad+\left|\sum_{p\in[0,\hat{n}-1]\backslash\cup_{i=1}^k[l_{i-1}^\prime+l_i,l_{i-1}^\prime+l_i+\bar{n}_i-1]}\varphi(\Theta^p(\omega,y))-\left(\sum_{i=1}^{k}l_i+(k-1)m(\frac{\epsilon}{4})\right)\alpha\right|\\
   &\leq \left|\sum_{i=1}^{k}\left(\sum_{p=l_{i-1}^\prime+l_i}^{l_{i-1}^\prime+l_i+\bar{n}_i-1}\varphi(\Theta^p(\omega,y))-\sum_{p=0}^{\bar{n}_i-1}\varphi(\Theta^p(\theta^{l_{i-1}^\prime+l_i}\omega,x_i))+\sum_{p=0}^{\bar{n}_i-1}\varphi(\Theta^p(\theta^{l_{i-1}^\prime+l_i}\omega,x_i))-\bar{n}_i\int\varphi d\nu_i\right.\right.\\
   &\quad\quad\left.\left.+\bar{n}_i\int\varphi d\nu_i-\bar{n}_i\alpha\right)\right|
   +2\Big((l_1+\cdots+l_k)+(k-1)m\Big(\frac{\epsilon}{4}\Big)\Big)\|\varphi\|_{C^0}\\
&\leq\bigg|\sum_{p=l_1}^{l_1+\bar{n}_1-1}\varphi(\Theta^p(\omega,y))-\sum_{p=0}^{\bar{n}_1-1}\varphi(\Theta^p(\theta^{l_1}\omega,x_1))\bigg|+\bigg|\sum_{p=0}^{\bar{n}_1-1}\varphi(\Theta^p(\theta^{l_1}\omega),x_1)-\bar{n}_1\int\varphi d\nu_1\bigg|+\cdots\nonumber\\
&+\bigg|\sum_{p=l_{k-1}^{\prime}+l_k}^{l_{k-1}^{\prime}+l_k+\bar{n}_k-1}\varphi(\Theta^p(\omega,y))-\sum_{p=0}^{\bar{n}_k-1}\varphi(\Theta^p(\theta^{l_{k-1}^{\prime}+l_k}\omega,x_k))\bigg|+\bigg|\sum_{p=0}^{\bar{n}_k-1}\varphi(\Theta^p(\theta^{l_{k-1}^{\prime}+l_k}\omega),x_k)-\bar{n}_k\int\varphi d\nu_k\bigg|\nonumber\\
&+\bigg|\bar{n}_1\int\varphi d\nu_1+\cdots+\bar{n}_k\int \varphi d\nu_k-(\bar{n}_1+\dots+\bar{n}_k)\alpha\bigg|+2\Big((l_1+\cdots+l_k)+(k-1)m\Big(\frac{\epsilon}{4}\Big)\Big)\|\varphi\|_{C^0}.\nonumber
\end{align*}
Combining the shadowing property and \eqref{eq pick of epsilon lemmaproof}, we have
		\begin{equation}\label{estimation 2}
			\bigg|\sum_{p=l_{j-1}^{\prime}+l_j}^{l_{j-1}^{\prime}+l_j+\bar{n}_j-1}\varphi(\Theta^p(\omega,y))-\sum_{p=0}^{\bar{n}_j-1}\varphi(\Theta^p(\theta^{l_{j-1}^{\prime}+l_j}\omega,x_j))\bigg|\leq \bar{n}_j\cdot\frac{\delta}{2},\ \forall j\in\{1,\dots,k\}.
		\end{equation}
		Besides, since $\theta^{l_{j-1}^{\prime}+l_j}\omega\in\Omega^2$ and $x_j\in C(\theta^{l_{j-1}^{\prime}+l_j}\omega,\bar{n}_j(\omega),\eta)$, we have
		\begin{equation}\label{estimation 3}
			\bigg|\sum_{p=0}^{\bar{n}_j-1}\varphi(\Theta^p(\theta^{l_{j-1}^{\prime}+l_j}\omega,x_j)-\bar{n}_j\int\varphi d\nu_j\bigg|\leq \bar{n}_j\delta,\ \forall j\in\{1,\dots,k\}.
		\end{equation}Moreover, we have
\begin{equation}\label{eq third item}
  \begin{split}
\bigg|\bar{n}_1\int\varphi d\nu_1+\cdots+\bar{n}_k\int \varphi d\nu_k-(\bar{n}_1+\dots+\bar{n}_k)\alpha\bigg|
 &\leq |\bar{n}\int \varphi d\nu-\bar{n}\alpha|+2(\bar{n}-\sum_{i=1}^{k}\bar{n}_i)\|\varphi\|_{C^0}\\
 &\leq \bar{n}\delta+2(\bar{n}-\sum_{i=1}^{k}\bar{n}_i)\|\varphi\|_{C^0}\\
 &\overset{\eqref{pick N6}}\leq 2\bar{n}\delta.
  \end{split}
\end{equation}
		Combining (\ref{proportion of l}), (\ref{estimation 2}), (\ref{estimation 3}) and \eqref{eq third  item} together, we have
		\begin{equation*}
			\bigg|\sum_{p=0}^{\hat{n}-1}\varphi(\Theta^p(\omega,y))-\hat{n}\alpha\bigg|\leq \hat{n}\cdot\frac{7\delta}{2}+2\xi\hat{n}\|\varphi\|_{C^0}+2(k-1)m\Big(\frac{\epsilon}{4}\Big)\|\varphi\|_{C^0}
		\end{equation*}
We conclude that	\begin{equation}
			\frac{1}{\hat{n}}\bigg|\sum_{p=0}^{\hat{n}-1}\varphi(\Theta^p(\omega,y))-\hat{n}\alpha\bigg|\leq \frac{7\delta}{2}+\frac{\delta}{4}+\frac{2(k-1)m(\frac{\epsilon}{4})}{\hat{n}}\|\varphi\|_{C^0}\leq 4\delta,
		\end{equation}provided \eqref{xi estimation} and $N_0$ in  \eqref{eq pick of n} sufficiently large satisfying
	\begin{equation*} \label{eq pick N delta 1epsilon} \frac{2(k-1)m(\frac{\epsilon}{4})}{n}\|\varphi\|_{C^0}\leq \frac{\delta}{4},\ \forall n\geq N_0.
	\end{equation*} Therefore, \eqref{eq   y subset P} is proved.

	Note that by \eqref{eq y eta/2 separated}, different $y=y(x_1,\dots,x_k)$ are $(\omega,\eta/2,\hat{n}(\omega))$ separated. Therefore
\begin{align*}
   M(\alpha,4\delta,\hat{n}(\omega),\frac{\eta}{2},\omega)&\geq\# C(\theta^{l_0^\prime(\omega)+l_1(\omega)}\omega,\bar{n}_1(\omega),\eta)\times\cdots\times\# C(\theta^{l_{k-1}^{\prime}(\omega)+l_k(\omega)}\omega,\bar{n}_k(\omega),\eta)\\
&\overset{\eqref{eq number C geq e}}\geq\exp\big([\lambda_1\bar{n}(\omega)](h_{\nu_1}(F)-\gamma)+\cdots+[\lambda_k\bar{n}(\omega)](h_{\nu_k}(F)-\gamma)\big)\\
&\geq\exp\big(\bar{n}(\omega)(h_{\nu}(F)-2\gamma)\big)\\
&\geq\exp\big(\bar{n}(\omega)(h_{\mu_0}(F)-3\gamma)\big),
\end{align*}provided $N_0$ in \eqref{eq pick of n} sufficiently large satisfying
\begin{equation*}
  [\lambda_1n](h_{\nu_1}(F)-\gamma)+\cdots+[\lambda_kn](h_{\nu_k}(F)-\gamma)\geq n(h_{\nu}(F)-2\gamma),\ \forall n\geq N_0.
\end{equation*} Thus, we obtain
	\begin{equation*}
		\begin{split}
		\frac{1}{\hat{n}(\omega)}\log M(\alpha,4\delta,\hat{n}(\omega),\frac{\eta}{2},\omega)&\geq
\frac{\bar{n}(\omega)(h_{\mu_0}(F)-3\gamma)}{\bar{n}_1(\omega)+\cdots+\bar{n}_k(\omega)}\cdot\frac{\bar{n}_1(\omega)+\cdots+\bar{n}_k(\omega)}{\hat{n}(\omega)}\\
&= \frac{\bar{n}(\omega)(h_{\mu_0}(F)-3\gamma)}{\bar{n}_1(\omega)+\cdots+\bar{n}_k(\omega)}\cdot\frac{\hat{n}(\omega)-(l_1(\omega)+\cdots+l_k(\omega))-(k-1)m(\frac{\epsilon}{4})}{\hat{n}(\omega)}\\
&\geq (h_{\mu_0}(F)-3\gamma)\cdot\bigg(\frac{\hat{n}(\omega)-(l_1(\omega)+\dots+l_k(\omega))}{\hat{n}(\omega)}-\frac{(k-1)m(\frac{\epsilon}{4})}{\hat{n}(\omega)}\bigg)\\
&\overset{\eqref{def of ni},\eqref{proportion of l}}\geq (1-2\xi)(h_{\mu_0}(F)-3\gamma)\\
&\overset{\eqref{xi estimation}}\geq h_{\mu_0}(F)-4\gamma,
		\end{split}
	\end{equation*}i.e., \eqref{eq 1 hatn M geq} is proved.
	It remains to show that $\hat{n}=\hat{n}(\omega):\Omega_{\delta}\to\mathbb{N}$ is measurable. Notice that
	\begin{displaymath}
		\hat{n}(\omega)=\sum_{i=1}^kl_i(\omega)+\sum_{i=1}^k\bar{n}_i(\omega)+(k-1)m\Big(\frac{\epsilon}{4}\Big),
	\end{displaymath} and we already showed the measurability of $\bar{n}_i:\Omega_\delta\to\mathbb{N}$.
    Hence, it suffices to prove that $l_i:\Omega_\delta\rightarrow\mathbb{Z}_{\geq 0}$ is measurable.
For $i=1$, notice that for any $k\geq 1$, we have
\begin{equation*}
  \{\omega\in\Omega_{\delta}:l_1(\omega)=k\}=\{\omega\in\Omega_{\delta}:\theta^k\omega\in \Omega^2\}\cap(\cap_{i=0}^{k-1}\{\omega\in\Omega_{\delta}:\theta^{i}\omega\notin\Omega^2\})
\end{equation*}
hence $l_1:\Omega_{\delta}\to\mathbb{N}$ is measurable. It follows that $l_1^{\prime}=l_1+\bar{n}_1+m(\frac{\epsilon}{4}):\Omega_{\delta}\to\mathbb{N}$ is measurable. Once the measurability of $l_{j}$ and $l_{j}^{\prime}$ is proven, for any $p\geq 1$, we have
\begin{displaymath}
	\begin{split}
		&\{\omega\in\Omega_{\delta}:l_{j+1}(\omega)=p\}
=\{\omega\in\Omega_{\delta}:\theta^{l_j^{\prime}(\omega)+p}\omega\in\Omega^2\}\cap\cap_{i=0}^{p-1}\{\omega\in\Omega_{\delta}:\theta^{l^{\prime}_{j}(\omega)+i}\omega\notin \Omega^2\}.
	\end{split}
\end{displaymath}
In fact, we have the following decomposition
\begin{displaymath}
	\begin{split} \{\omega\in\Omega_{\delta}:\theta^{l_j^{\prime}(\omega)+p}\omega\in \Omega^2\}&=\bigcup_{m\geq 0}\{\omega\in\Omega_{\delta}:l_j^{\prime}(\omega)=m,\theta^{p+m}(\omega)\in \Omega^2\}\\&=\bigcup_{m\geq 0}\left(\{\omega\in\Omega_{\delta}:l_j^{\prime}(\omega)=m\}\cap\{\omega\in\Omega_{\delta}:\theta^{p+m}\omega\in\Omega^2\}\right),
	\end{split}
\end{displaymath}
hence $\{\omega\in\Omega_{\delta}:\theta^{l_j^{\prime}(\omega)+p}\omega\in\Omega^2\}$ is measurable due to the measurability of $l_j^{\prime}.$ Similarly, one can prove that $\{\omega\in\Omega_{\delta}:\theta^{l_j^{\prime}(\omega)+i}\omega\notin\Omega^2\}$ is also measurable for $i\in\{0,\dots,p-1\}.$ Therefore, $l_{j+1}:\Omega_{\delta}\to\mathbb{N}$ is measurable and $l_{j+1}^{\prime}=l_j^{\prime}+l_{j+1}+\bar{n}_{j+1}+m(\frac{\epsilon}{4})$ is also measurable as a sum of measurable functions. By induction, $l_i$ is measurable for $i\in\{1,\dots,k\}$. The proof of Lemma \ref{Lemma lower estimates} is complete.
\end{proof}
\begin{proof}[Proof of Lemma \ref{lemma separated metric}]
	Since  $(x_0^1,\dots,x_{N_1(\omega)-1}^1)\neq (z_0^1,\dots,z_{N_1(\omega)-1}^1)$, there exists $j\in\{0,1,\dots,N_1(\omega)-1\},$ such that $x_j^1\neq z_j^1.$ By \eqref{eq shadowingproperty y}, one has
$$
  d_{\theta^{T_j^1(\omega)}\omega}^{\hat{n}_1(\theta^{T_j^1(\omega)}\omega)}\Big(x_j^1,F_{\omega}^{T_j^1(\omega)}y(x_0^1,\dots,x_{N_1(\omega)-1})\Big)<\frac{\eta}{2^{4+1}},$$ and $$
  d_{\theta^{T_j^1(\omega)}\omega}^{\hat{n}_1(\theta^{T_j^1(\omega)}\omega)}\Big(z_j^1,F_{\omega}^{T_j^1(\omega)}y(z_0^1,\dots,z_{N_1(\omega)-1})\Big)<\frac{\eta}{2^{4+1}}.
$$
Note that $x_j^1,z_j^1\in C_1(\theta^{T_j^1(\omega)}\omega)$ are $(\theta^{T_j^1(\omega)}\omega,\frac{\eta}{2},\hat{n}_1(\theta^{T_j^1(\omega)}\omega))$-separated, hence
\begin{align*}
&\ \ \ \ d_{\omega}^{T_{N_1(\omega)}^1(\omega)}(y(x_0^1,\dots,x_{N_1(\omega)-1}^1),y(z_1^1,\dots,z_{N_1(\omega)-1}^1))\\
&\geq d_{\theta^{T_j^1(\omega)}\omega}^{\hat{n}_1(\theta^{T_j^1(\omega)}\omega)}\Big(F_{\omega}^{T_j^1(\omega)}y(x_0^1,\dots,x_{N_1(\omega)-1}^1),F_{\omega}^{T_j^1(\omega)}y(z_0^1,\dots,z_{N_1(\omega)-1}^1)\Big)\\
&\geq d_{\theta^{T_j^1(\omega)}\omega}^{\hat{n}_1(\theta^{T_j^1(\omega)}\omega)}(x_j^1,z_j^1)- d_{\theta^{T_j^1(\omega)}\omega}^{\hat{n}_1(\theta^{T_j^1(\omega)}\omega)}\Big(x_j^1,F_{\omega}^{T_j^1(\omega)}y(x_0^1,\dots,x_{N_1(\omega)-1}^1)\Big)\\
&\ \ \ \ \ \ \ \ \ \ \ \ \ -d_{\theta^{T_j^1(\omega)}\omega}^{\hat{n}_1(\theta^{T_j^1(\omega)}\omega)}\Big(z_j^1,F_{\omega}^{T_j^1(\omega)}y(z_0^1,\dots,z_{N_1(\omega)-1}^1)\Big)\\
&> \frac{\eta}{2}-\frac{\eta}{2^4+1}\times 2.
\end{align*}
The proof of Lemma \ref{lemma separated  metric} is complete.
\end{proof}
\begin{proof}[Proof of Lemma \ref{lemma Hk separated}]
For any $x\in H_k(\omega),y,y^{\prime}\in D_{k+1}(\theta^{T_0^{k+1}(\omega)}\omega),$ by \eqref{z shadow xy}, one has
\begin{align*}
   d_{\omega}^{T_{N_k(\omega)}^k(\omega)}(z(x,y),z(x,y^{\prime}))&\leq  d_{\omega}^{T_{N_k(\omega)}^k(\omega)}(z(x,y),x)+ d_{\omega}^{T_{N_k(\omega)}^k(\omega)}(x,z(x,y^{\prime}))<\frac{\eta}{2^{4+k+1}}\times 2,
\end{align*}i.e., \eqref{eq 1 lemma 6.7} holds. If $y\not=y^\prime$, again by \eqref{z shadow xy} and Lemma \ref{lemma separated metric 2}, one has
\begin{align*}
 &\ \ \ \ d_{\omega}^{T_{N_{k+1}(\omega)}^{k+1}(\omega)}(z(x,y),z(x,y^{\prime}))\\
  &\geq  d_{\theta^{T_0^{k+1}(\omega)}\omega}^{T_{N_{k+1}(\omega)}^{k+1}(\omega)-T_0^{k+1}(\omega)}(y,y^{\prime})-
  d_{\theta^{T_0^{k+1}(\omega)}\omega}^{T_{N_{k+1}(\omega)}^{k+1}(\omega)-T_0^{k+1}(\omega)}(y^\prime,F_\omega^{T_0^{k+1}(\omega)}z(x,y^{\prime}))\\
  &\ \ \ \ \ \ \ \ \ \ \ -
  d_{\theta^{T_0^{k+1}(\omega)}\omega}^{T_{N_{k+1}(\omega)}^{k+1}(\omega)-T_0^{k+1}(\omega)}(y,F_\omega^{T_0^{k+1}(\omega)}z(x,y))\\
   &> \frac{\eta}{2}-\frac{\eta}{2^{4+k+1}}\times 2-\frac{\eta}{2^{4+k+1}}-\frac{\eta}{2^{4+k+1}}\geq \frac{3\eta}{8},
\end{align*}i.e., \eqref{eq 2 lemma 6.7} holds. The proof of Lemma \ref{lemma Hk separated} is complete.
\end{proof}

\begin{proof}[Proof of Lemma \ref{lemma decreased Xk}]
(1)	By \eqref{eq 3 lemma 6.7}, for any $x\not=x^\prime\in H_k(\omega)$, one has $d_{\omega}^{T_{N_k(\omega)}^k(\omega)}(x,x^{\prime})> \frac{3\eta}{8},$ hence
	\begin{displaymath}
	\overline{B}_{T_{N_k(\omega)}^k(\omega)}\Big(\omega,x,\frac{\eta}{2^{4+k}}\Big)\cap \overline{B}_{T_{N_k(\omega)}^k(\omega)}\Big(\omega,x^{\prime},\frac{\eta}{2^{4+k}}\Big)=\emptyset.
	\end{displaymath}

(2) Let $z=z(x,y)\in H_{k+1}(\omega)$ for $x\in H_k(\omega)$ and $y\in D_{k+1}(\theta^{T_0^{k+1}(\omega)}\omega).$ Pick any point $t\in\overline{B}_{T_{N_{k+1}(\omega)}^{k+1}(\omega)}(\omega,z,\frac{\eta}{2^{4+k+1}})$, we have
\begin{equation*}
  d_{\omega}^{T_{N_k(\omega)}^k(\omega)}(t,x)\leq d_{\omega}^{T_{N_k(\omega)}^k(\omega)}(t,z)+d_{\omega}^{T_{N_k(\omega)}^k(\omega)}(z,x)
  \overset{\eqref{z shadow xy}}<\frac{\eta}{2^{4+k+1}}+\frac{\eta}{2^{4+k+1}}
  =\frac{\eta}{2^{4+k}},
\end{equation*}
 Therefore, $\overline{B}_{T_{N_{k+1}(\omega)}^{k+1}(\omega)}\Big(\omega,z,\frac{\eta}{2^{4+k+1}}\Big)\subseteq \overline{B}_{T_{N_k(\omega)}^k(\omega)}\Big(\omega,x,\frac{\eta}{2^{4+k}}\Big).
$ The proof of Lemma \ref{lemma decreased Xk} is complete.
\end{proof}
\begin{proof}[Proof of Lemma \ref{lemma contained}]
For any $x\in\mathfrak{X}(\omega)$, we need to show $\lim_{n\to\infty}\frac{1}{n}\sum_{k=0}^{n-1}\varphi(\Theta^i(\omega,x))=\alpha$. So we have to estimate $|\sum_{k=0}^{n-1}\varphi(\Theta^i(\omega,x))-n\alpha|$ for $x\in \mathfrak{X}(\omega)$ and $n\in\mathbb{N}$. We are going to employ the shadowing property, triangle inequality and the estimation for points in $C_k(\theta^{T_j^k(\omega)}\omega)$, $D_k(\theta^{T_0^k(\omega)}\omega)$ and $H_k(\omega)$. The estimation is a 4-step process.

	\emph{Step1, estimation on $D_1(\omega)=H_1(\omega)$.} For any $y\in D_1(\omega),$ there exists some $x_j^1$ in $C_1(\theta^{T_j^1(\omega)}\omega)$ for $j\in\{0,\dots,N_1(\omega)-1\},$ such that
\begin{equation}\label{eq shadowing 1 lemma 6.9}
  d_{\theta^{T_j^1}\omega}^{\hat{n}_1(\theta^{T_j^1(\omega)}\omega)}\Big(x_j^1,F_{\omega}^{T_j^1(\omega)}y\Big)<\frac{\eta}{2^{4+1}}.
\end{equation}
we break the interval $[0,T_{N_1(\omega)}^1(\omega)-1]$ into
\begin{align*}
  &[0,T_0^1(\omega)-1]\cup[T_0^1(\omega),T_0^1(\omega)+\hat{n}_1(\theta^{T_0^1(\omega)}\omega)-1]\cup[T_0^1(\omega)+\hat{n}_1(\theta^{T_0^1(\omega)}\omega),T_1^1(\omega)-1]\\
  &\ \ \ \ \ \ \ \cup\cdots
  \cup[ T_{N_1(\omega)-1}^1(\omega),T_{N_1(\omega)-1}^1(\omega)+\hat{n}_1(\theta^{T_{N_1(\omega)-1}^1(\omega)}\omega)-1].
\end{align*}
On intervals $\cup_{j=0}^{N_1(\omega)-1}[T_j^1(\omega),T_j^1(\omega)+\hat{n}_1(\theta^{T_j^1(\omega)}\omega)-1]$, we use triangle inequality, shadowing property \eqref{eq shadowing 1 lemma 6.9} and notice $x_j^1\in C_1(\theta^{T_j^1(\omega)}\omega)\subset P(\alpha,4\delta_1,\hat{n}_1(\theta^{T_j^1(\omega)}\omega),\theta^{T_j^1(\omega)}\omega)$, while on other intervals we use inequality $|\varphi-\alpha|\leq 2\|\varphi\|_{C^0}$ to obtain
\begin{equation}\label{estimation on D1}
	\begin{split}		&\ \ \ \ \left|\sum_{i=0}^{T_{N_1(\omega)}^1(\omega)-1}\varphi(\Theta^i(\omega,y))-T_{N_1(\omega)}^1(\omega)\alpha\right|\\
&\leq \sum_{j=0}^{N_1(\omega)-1}\left|\sum_{i=T_j^1(\omega)}^{T_j^1(\omega)+\hat{n}_1(\theta^{T_j^1(\omega)}\omega)-1}\left(\varphi(\Theta^i(\omega,y))-\varphi(\Theta^{i-T_j^1(\omega)}(\theta^{T_j^1(\omega)}\omega,x_j^1))\right)\right|\\
&\quad\quad +\sum_{j=0}^{N_1(\omega)-1}\left|\sum_{i=T_j^1(\omega)}^{T_j^1(\omega)+\hat{n}_1(\theta^{T_j^1(\omega)}\omega)-1}\varphi(\Theta^{i-T_j^1(\omega)}(\theta^{T_j^1(\omega)}\omega,x_j^1))-\hat{n}_1(\theta^{T_j^1(\omega)}\omega)\alpha\right|\\
&\quad\quad\quad+2\left(T_0^1(\omega)+\sum_{j=1}^{N_1(\omega)-1}(m_1+l_j^1(\omega))\right)\|\varphi\|_{C^0}\\
&\leq\sum_{j=0}^{N_1(\omega)-1}\hat{n}_1(\theta^{T_j^1(\omega)}\omega)\mathrm{var}\Big(\varphi,\frac{\eta}{2^{4+1}}\Big)+4\sum_{j=0}^{N_1(\omega)-1}\hat{n}_1(\theta^{T_j^1(\omega)}\omega)\delta_1 \\
&\ \ \ \ \ \ \ \ \ \ \ \ \ \ \ \ \ \ \ \ \ \ +2\left(T_0^1(\omega)+\sum_{j=1}^{N_1(\omega)-1}(m_1+l_j^1(\omega))\right)\|\varphi\|_{C^0},
	\end{split}
\end{equation}where $\mathrm{var}(\varphi,c)=\sup\{|\varphi(\omega,x)-\varphi(\omega^\prime,x^\prime)|:\ d((\omega,x),(\omega^\prime,x^\prime))<c\}$ for some $c>0$.

\emph{Step 2, estimation on $D_k(\theta^{T_0^k(\omega)}\omega)$ for $k\geq 2$.} Suppose $y\in D_k(\theta^{T_0^k(\omega)}\omega),$ let us estimate the following difference
\begin{equation*}
  \Bigg|\sum_{i=T_0^k(\omega)}^{T_{N_k(\omega)}^k(\omega)-1}\varphi(\Theta^{i-T_0^k(\omega)}(\theta^{T_0^k(\omega)}\omega,y))-(T_{N_k(\omega)}^k(\omega)-T_0^k(\omega))\alpha\Bigg|.
\end{equation*}
 By the construction of $D_k(\theta^{T_0^k(\omega)}\omega),$ there exists $x_j^k\in C_k(\theta^{T_j^k(\omega)}\omega)$ for $j=\{0,\dots,N_k(\omega)-1\},$ such that
 \begin{equation}\label{eq shadowing 2 lemma 6.9}
   d_{\theta^{T_j^{k}(\omega)}\omega}^{\hat{n}_{k}(\theta^{T_j^{k}(\omega)}\omega)}\Big(x_j^{k}, F_{\theta^{T_0^{k}(\omega)}\omega}^{T_{j}^{k}(\omega)-T_0^{k}(\omega)}y\Big)<\frac{\eta}{2^{4+k}}.
 \end{equation}
we break the interval $[T_0^k(\omega),T_{N_k(\omega)}^k(\omega)-1]$ into
\begin{displaymath}
	\begin{split}
		&[T_0^k(\omega),T_0^k(\omega)+\hat{n}_k(\theta^{T_0^k(\omega)}\omega)-1]\cup[T_0^k(\omega)+\hat{n}_k(\theta^{T_0^k(\omega)}\omega),T_1^k(\omega)-1]\\
& \cup[T_1^k(\omega),T_1^k(\omega)+\hat{n}_k(\theta^{T_1^k(\omega)}\omega)-1]\cup[ T_{1}^k(\omega)+\hat{n}_k(\theta^{T_1^k(\omega)}\omega),T_{2}^k(\omega)-1]\\
&\ \ \ \ \cup\cdots\cup[T_{N_k(\omega)-1}^k(\omega),T_{N_k(\omega)-1}^k(\omega)+\hat{n}_k(\theta^{T_{N_k(\omega)-1}^k(\omega)}\omega)-1].
	\end{split}
\end{displaymath}
Similar as the estimation \eqref{estimation on D1}, on intervals $\cup_{j=0}^{N_k(\omega)-1}[T_j^k(\omega),T_j^k(\omega)+\hat{n}_k(\theta^{T_j^k(\omega)}\omega)-1]$, we use triangle inequality, shadowing property \eqref{eq shadowing 2 lemma 6.9} and notice $x_j^k\in C_k(\theta^{T_j^k(\omega)}\omega)\subset P(\alpha,4\delta_k,\hat{n}_k(\theta^{T_j^k(\omega)}\omega),\theta^{T_j^k(\omega)}\omega)$, while on other intervals we use $|\varphi-\alpha|\leq 2\|\varphi\|_{C^0}$ to obtain
\begin{align*}
& \Bigg|\sum_{i=T_0^k(\omega)}^{T_{N_k(\omega)}^k(\omega)-1}\varphi(\Theta^{i-T_0^k(\omega)}(\theta^{T_0^k(\omega)}\omega,y))-(T_{N_k(\omega)}^k(\omega)-T_0^k(\omega))\alpha\Bigg|\label{eq estimate on Dk}\\
\leq &\sum_{j=0}^{N_k(\omega)-1}\left|\sum_{i=T_j^k(\omega)}^{T_j^k(\omega)+\hat{n}_k(\theta^{T_j^k(\omega)}\omega)-1}\left(\varphi(\Theta^{i-T_0^k(\omega)}(\theta^{T_0^k(\omega)}\omega,y))-\varphi(\Theta^{i-T_j^k(\omega)}(\theta^{T_j^k(\omega)}\omega,x_j^k))\right)\right|\\
&\quad+\sum_{j=0}^{N_k(\omega)-1}\left|\sum_{i=T_j^k(\omega)}^{T_j^k(\omega)+\hat{n}_k(\theta^{T_j^k(\omega)}\omega)-1}\varphi(\Theta^{i-T_j^k(\omega)}(\theta^{T_j^k(\omega)}\omega,x_j^k))-\hat{n}_k(\theta^{T_j^k(\omega)}\omega)\alpha\right|\\
&\quad\quad+2\sum_{j=1}^{N_k(\omega)-1}(m_k+l_j^k(\omega))\|\varphi\|_{C^0}\\
\leq & \sum_{j=0}^{N_k(\omega)-1}\hat{n}_k(\theta^{T_j^k(\omega)}\omega)\mathrm{var}\Big(\varphi,\frac{\eta}{2^{4+k}}\Big)+4\sum_{j=0}^{N_k(\omega)-1}\hat{n}_k(\theta^{T_j^k(\omega)}\omega)\delta_k+2\sum_{j=1}^{N_k(\omega)-1}(m_k+l_j^k(\omega))\|\varphi\|_{C^0}.\nonumber	
\end{align*}

\emph{Step 3, estimation on $H_k(\omega)$.} Denote
\begin{equation}\label{eq def Rk}
  R_k(\omega)=\max_{z\in H_k(\omega)}\Bigg|\sum_{i=0}^{T_{N_k(\omega)}^k(\omega)-1}\varphi(\Theta^i(\omega,z))-T_{N_k(\omega)}^k(\omega)\alpha\Bigg|.
\end{equation}We show
\begin{equation}\label{eq estimation Rk}
	\begin{split} R_k(\omega)&\leq 2\sum_{i=1}^{k}T_{N_i(\omega)}^i(\omega)\Big(\mathrm{var}\Big(\varphi,\frac{\eta}{2^{4+i}}\Big)+2\delta_i\Big)+2\sum_{i=1}^{k}N_i(\omega)m_i\|\varphi\|_{C^0}\\
&\ \ \ \ +2\left(\sum_{i=0}^{k-1}l^{i,i+1}(\omega)+\sum_{i=1}^{k}\sum_{j=1}^{N_i-1}l_j^i(\omega)\right)\|\varphi\|_{C^0},
	\end{split}
\end{equation}
by induction on $k\in\mathbb{N}$.
When $k=1$, inequality (\ref{estimation on D1}) gives the desired estimation. Suppose now we have an upper estimation on $R_k(\omega),$ let us estimation $R_{k+1}(\omega).$ For any $z\in H_{k+1}(\omega),$ by \eqref{z shadow xy}, there exist $x\in H_k(\omega)$ and $y\in D_{k+1}(\theta^{T_0^{k+1}(\omega)}\omega),$ such that
\begin{equation}\label{eq shadowing 3 lemma 6.9}
  d_{\omega}^{T_{N_k(\omega)}^k(\omega)}(x,z)<\frac{\eta}{2^{4+k+1}},\mbox{ and } d_{\theta^{T_0^{k+1}(\omega)}\omega}^{T_{N_{k+1}(\omega)}^{k+1}(\omega)-T_0^{k+1}(\omega)}\big(y,F_{\omega}^{T_0^{k+1}(\omega)}z\big)<\frac{\eta}{2^{4+k+1}}.
\end{equation}
 We break $[0,T_{N_{k+1}(\omega)}^{k+1}(\omega)-1]$ into $$[0,T_{N_k(\omega)}^k(\omega)-1]\cup[T_{N_k(\omega)}^k(\omega),T_0^{k+1}(\omega)-1]\cup[T_0^{k+1}(\omega),T_{N_{k+1}(\omega)}^{k+1}(\omega)-1].$$ On interval $[0,T_{N_k(\omega)}^k(\omega)-1]$, we use triangle inequality, the first inequality of \eqref{eq shadowing 3 lemma 6.9}, and $R_k(\omega)$. On interval $[T_{N_k(\omega)}^k(\omega),T_0^{k+1}(\omega)-1]=[T_{N_k(\omega)}^k(\omega),T_{N_k(\omega)}^k(\omega)+m_{k+1}+l^{k,k+1}(\omega)-1]$, we use estimate $|\varphi-\alpha|\leq 2\|\varphi\|_{C^0}$. On interval $[T_0^{k+1}(\omega),T_{N_{k+1}(\omega)}^{k+1}(\omega)-1]$, we use triangle inequality, the second inequality of \eqref{eq shadowing 3 lemma 6.9} and estimate on $D_{k+1}(\theta^{T_0^{k+1}(\omega)}\omega)$. Therefore
\begin{displaymath}
	\begin{split}
&R_{k+1}(\omega)=\Bigg|\sum_{i=0}^{T_{N_{k+1}(\omega)}^{k+1}(\omega)-1}\varphi(\Theta^i(\omega,z))-T_{N_{k+1}(\omega)}^{k+1}(\omega)\alpha\Bigg|\\
\leq &\left|\sum_{i=0}^{T_{N_{k}(\omega)}^{k}(\omega)-1}\left(\varphi(\Theta^i(\omega,z))-\varphi(\Theta^i(\omega,x))\right)\right|+\left|\sum_{i=0}^{T_{N_{k}(\omega)}^{k}(\omega)-1}\varphi(\Theta^i(\omega,x))-T_{N_k(\omega)}^k(\omega)\alpha\right| \\
&\quad+2(m_{k+1}+l^{k,k+1}(\omega))\|\varphi\|_{C^0}+\left|\sum_{i=T_{0}^{k+1}(\omega)}^{T_{N_{k+1}(\omega)}^{k+1}(\omega)-1}\left(\varphi(\Theta^i(\omega,z))-\varphi(\Theta^{i-T_0^{k+1}(\omega)}(\theta^{T_0^{k+1}(\omega)}\omega,y))\right)\right|\\
&\quad\quad+\left|\sum_{i=T_{0}^{k+1}(\omega)}^{T_{N_{k+1}(\omega)}^{k+1}(\omega)-1}\varphi(\Theta^{i-T_0^{k+1}(\omega)}(\theta^{T_0^{k+1}(\omega)}\omega,y))-(T_{N_{k+1}(\omega)}^{k+1}(\omega)-T_0^{k+1}(\omega))\alpha\right|\\
\leq & T_{N_k(\omega)}^k(\omega)\mathrm{var}\Big(\varphi,\frac{\eta}{2^{4+k+1}}\Big)+R_k(\omega)+2(m_{k+1}+l^{k,k+1}(\omega))\|\varphi\|_{C^0}\\
&\ \ \ \ +\Big(T_{N_{k+1}(\omega)}^{k+1}(\omega)-T_0^{k+1}(\omega)\Big)\mathrm{var}\Big(\varphi,\frac{\eta}{2^{4+k+1}}\Big)+\sum_{j=0}^{N_{k+1}(\omega)-1}\hat{n}_{k+1}(\theta^{T_j^{k+1}(\omega)}\omega)\mathrm{var}\Big(\varphi,\frac{\eta}{2^{4+k+1}}\Big)\\
&\ \ \ \ \ \ \ \ +4\sum_{j=0}^{N_{k+1}(\omega)-1}\hat{n}_{k+1}(\theta^{T_j^{k+1}(\omega)}\omega)\delta_{k+1}+2\sum_{j=1}^{N_{k+1}(\omega)-1}(m_{k+1}+l_j^{k+1}(\omega))\|\varphi\|_{C^0}\\
\leq & R_k(\omega)+ 2T_{N_{k+1}(\omega)}^{k+1}(\omega)\mathrm{var}\Big(\varphi,\frac{\eta}{2^{4+k+1}}\Big)+2\Big(N_{k+1}(\omega)m_{k+1}+l^{k,k+1}(\omega)+\sum_{j=1}^{N_{k+1}-1}l_j^{k+1}(\omega)\Big)\|\varphi\|_{C^0}\\
&\ \ \ \ +4T_{N_{k+1}(\omega)}^{k+1}(\omega)\delta_{k+1}.
\end{split}
\end{displaymath}By the induction step, we have
\begin{displaymath}
\begin{split}
R_{k+1}(\omega)
\leq &2\sum_{i=1}^{k+1}T_{N_i(\omega)}^i(\omega)\Big(\mathrm{var}\Big(\varphi,\frac{\eta}{2^{4+i}}\Big)+2\delta_i\Big)+2\sum_{i=1}^{k+1}N_i(\omega)m_i\|\varphi\|_{C^0}\\
&\quad +2\left(\sum_{i=0}^{k}l^{i,i+1}(\omega)+\sum_{i=1}^{k+1}\sum_{j=1}^{N_i-1}l_j^i(\omega)\right)\|\varphi\|_{C^0}.
\end{split}
\end{displaymath}By induction, we obtain \eqref{eq estimation Rk}.
 Next, we claim that
  \begin{equation}\label{eq R OVER t GOES 0}
    \mbox{$R_k(\omega)/T_{N_k(\omega)}^k(\omega)\to 0$ as $k\to\infty.$}
  \end{equation} In fact, by \eqref{estimate 1 bad point} and \eqref{estimate of k bad point}, one can inductively estimate
 \begin{align*}
   T_{N_{k-1}(\omega)}^{k-1}(\omega)&\leq N_1(\omega)(\hat{n}^M_1+m_1)\prod_{i=1}^{k-1}(1-\xi_i)^{-1}+N_2(\omega)(\hat{n}^M_2+m_2)\prod_{i=2}^{k-1}(1-\xi_i)^{-1}+\\
   &\ \ \ \ \cdots+ N_{k-1}(\omega)(\hat{n}_{k-1}^M+m_{k-1})(1-\xi_{k-1})^{-1}.
 \end{align*}
By using (\ref{choose of Nk}), we have
\begin{equation}\label{quotient of T} \limsup_{k\rightarrow\infty}\frac{T_{N_{k-1}(\omega)}^{k-1}(\omega)}{T_{N_k(\omega)}^k(\omega)}\leq \limsup_{k\rightarrow\infty}\frac{ \sum_{j=1}^{k-1}(N_j(\omega)\cdot(\hat{n}_j^M+m_j)\cdot\prod_{i=j}^{k-1}(1-\xi_i)^{-1})}{N_k(\omega)}= 0.
\end{equation}Using Stolz's theorem, since the sequence $\{T_{N_k(\omega)}^k(\omega)\}_k$ is strictly increasing and  $T_{N_k(\omega)}^k(\omega)\to\infty$ as $k\to\infty$, we have
\begin{equation*}\label{quottient term 1}
  \lim_{k\to\infty}\frac{2\sum_{i=1}^{k}T_{N_i(\omega)}^i(\omega)(\mathrm{var}(\varphi,\frac{\eta}{2^{4+i}})+2\delta_i)}{T_{N_k(\omega)}^k(\omega)}=\lim_{k\to\infty}\frac{2T_{N_k(\omega)}^k(\omega)(\mathrm{var}(\varphi,\frac{\eta}{2^{4+k}})+2\delta_k)}{T_{N_k(\omega)}^k(\omega)-T_{N_{k-1}(\omega)}^{k-1}(\omega)}=0.
\end{equation*}
Noticing that $\hat{n}_k(\omega)\geq 2^{m_k}$ for all $k\in\mathbb{N}$, so by construction of $T_{N_k(\omega)}^k(\omega)$, we have $T_{N_k(\omega)}^k(\omega)\geq \sum_{i=1}^k N_i(\omega)\cdot 2^{m_i}.$ Hence, by using Stolz's theorem again, we have
\begin{equation*}\label{quotient term 2}
  \lim_{k\to\infty}\frac{2\sum_{i=1}^k N_i(\omega)m_i\|\varphi\|_{C^0}}{T_{N_k(\omega)}^k(\omega)}\leq \lim_{k\to\infty}\frac{2\sum_{i=1}^k N_i(\omega)m_i\|\varphi\|_{C^0}}{\sum_{i=1}^kN_i(\omega)\cdot 2^{m_i}}= \lim_{k\to\infty}\frac{2N_k(\omega)m_k\|\varphi\|_{C^0}}{N_k(\omega)\cdot 2^{m_k}}=0.
\end{equation*}
Finally, using  (\ref{estimate of k bad point}) and \eqref{quotient of T}, we have
\begin{align*}
 &\ \ \ \ \limsup_{k\to\infty}\frac{2(\sum_{i=0}^{k-1}l^{i,i+1}(\omega)+\sum_{i=1}^{k}\sum_{j=1}^{N_i-1}l_j^i(\omega))\|\varphi\|_{C^0}}{T_{N_k(\omega)}^k(\omega)}\\
 &=\limsup_{k\to\infty}\frac{2(\sum_{i=0}^{k-2}l^{i,i+1}(\omega)+\sum_{i=1}^{k-1}\sum_{j=1}^{N_i-1}l_j^i(\omega))\|\varphi\|_{C^0}+2(l^{k-1,k}(\omega)+\sum_{j=1}^{N_k-1}l_j^k(\omega))\|\varphi\|_{C^0}}{T_{N_k(\omega)}^k(\omega)}\\\\
 &\leq\limsup_{k\to\infty}\frac{2T_{N_{k-1}(\omega)}^{k-1}(\omega)\|\varphi\|_{C^0}+2(l^{k-1,k}(\omega)+\sum_{j=1}^{N_k-1}l_j^k(\omega))\|\varphi\|_{C^0}}{T_{N_k(\omega)}^k(\omega)}\\
 &\leq\limsup_{k\to\infty}2\left(\frac{T_{N_{k-1}(\omega)}^{k-1}(\omega)}{T_{N_k(\omega)}^k(\omega)}+\xi_{k+1}\right)\|\varphi\|_{C^0}=0.
\end{align*}
Combining all above together, we arrive ${R_k}(
\omega)/{T_{N_{k}(\omega)}^k(\omega)}\to 0$ as $k\to\infty.$

\emph{Last step, estimation on $\mathfrak{X}(\omega).$} Given $x\in\mathfrak{X}(\omega)$, we are going to estimate $|\frac{1}{n}\sum_{k=0}^{n-1}\varphi(\Theta^i(\omega,x))-\alpha|$ for $n\in\mathbb{N}$ large enough by using previous estimates. Note that for any $n\geq T_0^2(\omega)$, there exists a unique $k\geq 1$, such that
\begin{displaymath}
	T_{N_k(\omega)}^k(\omega)<n\leq T_{N_{k+1}(\omega)}^{k+1}(\omega).
\end{displaymath} More precisely, there are two cases:
\begin{enumerate}
	\item [Case 1.] $T_{N_k(\omega)}^k(\omega)<n\leq T_{N_k(\omega)}^k(\omega)+l^{k,k+1}(\omega)+m_{k+1}=T^{k+1}_0(\omega);$
	\item [Case 2.] $T_0^{k+1}(\omega)<n\leq T_{N_{k+1}(\omega)}^{k+1}(\omega).$
\end{enumerate}

In case 1: One has
\begin{equation*}
  n- T_{N_k(\omega)}^k(\omega)\leq m_{k+1}+l^{k,k+1}(\omega)\overset{\eqref{estimate of k bad point}}\leq  m_{k+1}+\xi_{k+1}n.
\end{equation*}
 Since $x\in\mathfrak{X}(\omega),$ by the construction of $\mathfrak{X}(\omega),$ there exists some $z\in H_k(\omega)$ such that $d_\omega^{T_{N_k(\omega)}^k(\omega)}(z,x)<\frac{\eta}{2^{4+k}}$.
Now by using triangle inequality, one has
\begin{align*}
   &\quad\Bigg|\frac{1}{n}\sum_{i=0}^{n-1}\varphi(\Theta^i(\omega,x))-\alpha\Bigg|\\
   &\leq \frac{1}{n}\Bigg(\Bigg|\sum_{i=0}^{T_{N_k(\omega)}^k(\omega)-1}\varphi(\Theta^i(\omega,x))-\sum_{i=0}^{T_{N_k(\omega)}^k(\omega)-1}\varphi(\Theta^i(\omega,z))\Bigg|+\Bigg|\sum_{i=0}^{T_{N_k(\omega)}^k(\omega)-1}\varphi(\Theta^i(\omega,z))-T_{N_k(\omega)}^k(\omega)\alpha\Bigg|\\
   &\quad\quad+\Bigg|\sum_{i=T_{N_k(\omega)}^k(\omega)}^{n-1}\varphi(\Theta^i(\omega,x))-(n-T_{N_k(\omega)}^k(\omega))\alpha\Bigg|\Bigg)\\
   &\leq\frac{1}{n}\Big(T_{N_k(\omega)}^k(\omega)\mathrm{var}\Big(\varphi,\frac{\eta}{2^{4+k}}\Big)+R_k(\omega)+2(\xi_{k+1}n
		+m_{k+1})\|\varphi\|_{C^0}\Big)\\
&\leq \mathrm{var}\Big(\varphi,\frac{\eta}{2^{4+k}}\Big)+R_k(\omega)/T_{N_k(\omega)}^k(\omega)+2(\xi_{k+1}+\frac{m_{k+1}}{N_k})\|\varphi\|_{C^0},
\end{align*}
which tends to zero as $n\to\infty$ by \eqref{eq R OVER t GOES 0} and \eqref{choose of Nk}.

In case 2: there exists some $j\in\{0,1,\dots,N_{k+1}(\omega)-1\}$ such that $ T_{j}^{k+1}(\omega)<n\leq T_{j+1}^{k+1}(\omega).$
Since $x\in\mathfrak{X}(\omega),$ by construction of $\mathfrak{X}(\omega)$ \eqref{construction of hua H}, there exists some $z\in H_{k+1}(\omega)$ such that
\begin{equation}\label{distance zx}
	d_{\omega}^{T_{N_{k+1}(\omega)}^{k+1}(\omega)}\big(z,x\big)<\frac{\eta}{2^{4+k+1}}.
\end{equation}
For such $z\in H_{k+1}(\omega)$, by \eqref{z shadow xy}, there exist some $\bar{x}\in H_k(\omega)$ and $y\in D_{k+1}(\theta^{T_0^{k+1}(\omega)}\omega)$ such that
\begin{equation}\label{distance zyxbar}
	d_{\omega}^{T_{N_k(\omega)}^k(\omega)}(z,\bar{x})<\frac{\eta}{2^{4+k+1}},\quad\mbox{and } d_{\theta^{T_0^{k+1}(\omega)}\omega}^{T_{N_{k+1}(\omega)}^{k+1}(\omega)-T_0^{k+1}(\omega)}\Big(y,F_{\omega}^{T_0^{k+1}(\omega)}z\Big)<\frac{\eta}{2^{4+k+1}}.
\end{equation}
Combining (\ref{distance zx}) and (\ref{distance zyxbar}) together, we have
\begin{equation}\label{x shadowed by x bar}
  d_{\omega}^{T_{N_k(\omega)}^k(\omega)}(x,\bar{x})<\frac{\eta}{2^{4+k}}.
\end{equation}
On the other hand, since $y\in D_{k+1}(\theta^{T_0^{k+1}(\omega)}\omega),$ there exists some $x_i^{k+1}\in C_{k+1}(\theta^{T_i^{k+1}(\omega)}\omega)$ such that
\begin{equation}\label{distance xy}
	d_{\theta^{T_i^{k+1}(\omega)}\omega}^{\hat{n}_{k+1}(\theta^{T_i^{k+1}(\omega)}\omega)}\Big(x_i^{k+1},F_{\theta^{T_0^{k+1}(\omega)}\omega}^{T_i^{k+1}(\omega)-T_0^{k+1}(\omega)}y\Big)<\frac{\eta}{2^{4+k+1}}, \mbox{ for } i\in\{0,\dots,N_{k+1}-1\}.
\end{equation}
For $i\in\{0,1,\dots,j-1\}$, combining (\ref{distance zx}), (\ref{distance zyxbar}) and (\ref{distance xy}) together gives
\begin{equation}\label{x shadowed by xi}
  d_{\theta^{T_i^{k+1}(\omega)}\omega}^{\hat{n}_{k+1}(\theta^{T_i^{k+1}(\omega)}\omega)}\Big(F_{\omega}^{T_{i}^{k+1}(\omega)}x,x_i^{k+1}\Big)<\frac{\eta}{2^{4+k+1}}+\frac{\eta}{2^{4+k+1}}+\frac{\eta}{2^{4+k+1}}<\frac{\eta}{2^{3+k}}.
\end{equation}
In case $T_{j}^{k+1}(\omega)<n\leq T_{j+1}^{k+1}(\omega)$, then either
\begin{equation}\label{1 subcase 1}
  T_j^{k+1}(\omega)<n\leq T_j^{k+1}(\omega)+\hat{n}_{k+1}(\theta^{T_j^{k+1}(\omega)}\omega)
\end{equation} or
\begin{equation}\label{1 subcase 2}
  T_{j}^{k+1}(\omega)+\hat{n}_{k+1}(\theta^{T_j^{k+1}(\omega)}\omega)<n\leq T_{j+1}^{k+1}(\omega)=T_{j}^{k+1}(\omega)+\hat{n}_{k+1}(\theta^{T_j^{k+1}(\omega)}\omega)+m_{k+1}+l_{j+1}^{k+1}(\omega).
\end{equation}
In subcases \eqref{1 subcase 1}: on interval $[0,T_{N_k(\omega)}^k(\omega)-1]$, we use shadowing \eqref{x shadowed by x bar} and estimate on $R_k(\omega)$; on interval $\cup_{i=0}^{j-1}[T_i^{k+1}(\omega),T_{i}^{k+1}(\omega)+\hat{n}_{k+1}(\theta^{T_i^{k+1}(\omega)}\omega)-1]$ (with the convention that $\cup_{i=0}^{-1}=\emptyset$), we use shadowing \eqref{x shadowed by xi} and the fact $x_i^{k+1}\in C_{k+1}(\theta^{T_i^{k+1}(\omega)}\omega)\subset P(\alpha,4\delta_{k+1},\hat{n}_{k+1}(\theta^{T_i^{k+1}(\omega)}\omega),\theta^{T_i^{k+1}(\omega)}\omega)$; on other intervals, we use inequality $|\varphi-\alpha|\leq 2\|\varphi\|_{C^0}$ to obtain
\begin{align*}
  &\quad \frac{1}{n}\bigg|\sum_{i=0}^{n-1}\varphi(\Theta^i(\omega,x))-n\alpha\bigg|\\
  &\leq \frac{1}{n}\left\{\Bigg|\sum_{i=0}^{T_{N_k(\omega)}^k(\omega)-1}\varphi(\Theta^i(\omega,x))-\sum_{i=0}^{T_{N_k(\omega)}^k(\omega)-1}\varphi(\Theta^i(\omega,\bar{x}))\Bigg|+\Bigg|\sum_{i=0}^{T_{N_k(\omega)}^k(\omega)-1}\varphi(\Theta^i(\omega,\bar{x}))-T_{N_k(\omega)}^k(\omega)\alpha\Bigg|\right.\\
&\quad+\sum_{i=0}^{j-1}\left|\sum_{p=T_i^{k+1}(\omega)}^{T_i^{k+1}(\omega)+\hat{n}_{k+1}(\theta^{T_i^{k+1}(\omega)}\omega)-1}(\varphi(\Theta^p(\omega,x))-\varphi(\Theta^{p-T_i^{k+1}(\omega)}(\theta^{T_i^{k+1}(\omega)}\omega,x_i^{k+1})))\right|\\
&\quad\quad \left.+2l^{k,k+1}(\omega)\|\varphi\|_{C^0} +2\sum_{i=1}^jl_i^{k+1}(\omega)\|\varphi\|_{C^0}+2(j+1)m_{k+1}\|\varphi\|_{C^0}+2(n-T_j^{k+1}(\omega))\|\varphi\|_{C^0}\right\}\\
  &\leq\frac{1}{n}\left\{\Bigg|\sum_{i=0}^{T_{N_k(\omega)}^k(\omega)-1}\varphi(\Theta^i(\omega,x))-\sum_{i=0}^{T_{N_k(\omega)}^k(\omega)-1}\varphi(\Theta^i(\omega,\bar{x}))\Bigg|+\Bigg|\sum_{i=0}^{T_{N_k(\omega)}^k(\omega)-1}\varphi(\Theta^i(\omega,\bar{x}))-T_{N_k(\omega)}^k(\omega)\alpha\Bigg|\right.\\
&\quad+2l^{k,k+1}(\omega)\|\varphi\|_{C^0}+2\sum_{i=1}^jl_i^{k+1}(\omega)\|\varphi\|_{C^0}+2(j+1)m_{k+1}\|\varphi\|_{C^0}\\
&\quad\quad\left.+\sum_{i=0}^{j-1}\hat{n}_{k+1}(\theta^{T_i^{k+1}(\omega)}\omega)\Big[\mathrm{var}\Big(\varphi,\frac{\eta}{2^{3+k}}\Big)+4\delta_{k+1}\Big]+2\hat{n}_{k+1}^M\|\varphi\|_{C^0}\right\}\\
&\leq\frac{1}{n}\left\{T_{N_k(\omega)}^k(\omega)\mathrm{var}\Big(\varphi,\frac{\eta}{2^{4+k}}\Big)+R_k(\omega)+2\xi_{k+1}n\|\varphi\|_{C^0}+n\Big[\mathrm{var}\Big(\varphi,\frac{\eta}{2^{3+k}}\Big)+4\delta_{k+1}\Big]\right.\\
&\quad\quad\left.+2(\hat{n}_{k+1}^M+(j+1)m_{k+1})\|\varphi\|_{C^0}\right\}\\
&\leq \mathrm{var}\Big(\varphi,\frac{\eta}{2^{4+k}}\Big)+\frac{R_k(\omega)}{T_{N_k(\omega)}^k(\omega)}+\mathrm{var}\Big(\varphi,\frac{\eta}{2^{3+k}}\Big)+4\delta_{k+1}+2(\xi_{k+1}+\frac{\hat{n}_{k+1}^M}{N_k(\omega)}+\frac{(j+1)m_{k+1}}{N_k(\omega)+j2^{m_{k+1}}})\|\varphi\|_{C^0},
\end{align*} with the convention that $\sum_{i=0}^{-1}=\sum_{i=1}^{0}=0$.
Hence $|\frac{1}{n}\sum_{i=0}^{n-1}\varphi(\Theta^i(\omega,x))-\alpha|\to 0$ as $n\to \infty$ by using previous estimates.

In subcase \eqref{1 subcase 2}, the estimate \eqref{x shadowed by xi} holds for $i\in\{0,1,...,j\}$. Moreover, we notice $n- T_{j}^{k+1}(\omega)-\hat{n}_{k+1}(\theta^{T_j^{k+1}(\omega)}\omega)\leq m_{k+1}+\xi_{k+1}n$. Use the same strategy as in subcase \eqref{1 subcase 1}, we obtain
\begin{align*}
&\quad \frac{1}{n}\bigg|\sum_{i=0}^{n-1}\varphi(\Theta^i(\omega,x))-n\alpha\bigg|\\
&\leq\frac{1}{n}\left\{\Bigg|\sum_{i=0}^{T_{N_k(\omega)}^k(\omega)-1}\varphi(\Theta^i(\omega,x))-\sum_{i=0}^{T_{N_k(\omega)}^k(\omega)-1}\varphi(\Theta^i(\omega,\bar{x}))\Bigg|+\Bigg|\sum_{i=0}^{T_{N_k(\omega)}^k(\omega)-1}\varphi(\Theta^i(\omega,\bar{x}))-T_{N_k(\omega)}^k(\omega)\alpha\Bigg|\right.\\
&\quad+2l^{k,k+1}(\omega)\|\varphi\|_{C^0}+2\sum_{i=1}^jl_i^{k+1}(\omega)\|\varphi\|_{C^0}+2\xi_{k+1}n\|\varphi\|_{C^0}+2(\hat{n}_{k+1}^M+m_{k+1})\|\varphi\|_{C^0}\\
&\quad\quad\left.+2(j+1)m_{k+1}\|\varphi\|_{C^0}+\sum_{i=0}^{j}\hat{n}_{k+1}(\theta^{T_i^{k+1}(\omega)}\omega)\Big[\mathrm{var}\Big(\varphi,\frac{\eta}{2^{3+k}}\Big)+4\delta_{k+1}\Big]\right\}\\
&\leq \frac{1}{n}\left\{T_{N_k(\omega)}^k(\omega)\mathrm{var}\Big(\varphi,\frac{\eta}{2^{4+k}}\Big)+R_k(\omega)+4\xi_{k+1}n\|\varphi\|_{C^0}+n\Big[\mathrm{var}\Big(\varphi,\frac{\eta}{2^{3+k}}\Big)+4\delta_{k+1}\Big]\right.\\
&\quad\quad \left.+2(\hat{n}_{k+1}^M+(j+2)m_{k+1})\|\varphi\|_{C^0}\right\}\\
&\leq \frac{R_k(\omega)}{T_{N_k(\omega)}^k(\omega)}+2\mathrm{var}\left(\varphi,\frac{\eta}{2^{3+k}}\right)+4\delta_{k+1}+2\left(2\xi_{k+1}+\frac{\hat{n}_{k+1}^M}{N_k(\omega)}+\frac{(j+2)m_{k+1}}{N_k(\omega)+j2^{m_{k+1}}}\right)\|\varphi\|_{C^0},
\end{align*}with the convention that $\sum_{i=1}^{0}=0$.
Hence $|\frac{1}{n}\sum_{i=0}^{n-1}\varphi(\Theta^i(\omega,x))-\alpha|\to 0$ as $n\to \infty$ by using previous estimates. Combining Case 1 and Case 2, we obtain $\mathfrak{X}(\omega)\subset K_{\varphi,\alpha}(\omega)$. The proof of lemma \ref{lemma contained} is complete.
\end{proof}
\begin{proof}[Proof of Lemma \ref{lemma limit measure}]
	It's sufficient to show that $\{\int_M\psi d\mu_{k,\omega}\}_{k=1}^{\infty}$ is a Cauchy sequence. Given $\delta>0$, for any sufficiently large $K$ such that $\mathrm{var}(\psi,\frac{\eta}{2^{4+K}})<\delta$, let $k_2> k_1\geq K$ be any numbers, then
\begin{equation*}
  \int_M\psi d\mu_{k_i,\omega}=\frac{1}{\# H_{k_i}(\omega)}\sum_{x\in H_{k_i}(\omega)}\psi(x)\mbox{ for }i=1,2.
\end{equation*}
 For any $x\in H_{k_1}(\omega),$ denote $Z(x)$ to be the collection of points $z$ in $H_{k_2}(\omega)$ such that $z$ descends from $x$, i.e., there is a sequence $z_{k_1+1}\in H_{k_1+1}(\omega),\dots,z_{k_2-1}\in H_{k_2-1}(\omega)$ satisfying
 \begin{equation}\label{def descend}
 	\begin{split}
 		z&=z(z_{k_2-1},y_{k_2}),\ \exists\ y_{k_2}\in D_{k_2}(\theta^{T_0^{k_2}(\omega)}\omega),\\
 z_{k_2-1}&=z(z_{k_2-2}, y_{k_2-1}),\ \exists
 \ y_{k_2-1}\in D_{k_2}(\theta^{T_0^{k_2-1}(\omega)}\omega),\\
 &\cdots\\
 z_{k_1+1}&=z(x,y_{k_1+1}),\ \exists\ y_{k_1+1}\in D_{k_1+1}(\theta^{T_0^{k_1+1}(\omega)}\omega).
 	\end{split}
 \end{equation}
It follows that $\#Z(x)=\# D_{k_1+1}(\theta^{T_0^{k_1+1}(\omega)}\omega)\cdots\# D_{k_2}(\theta^{T_0^{k_2}(\omega)}\omega)$, and by \eqref{number Hk},
\begin{equation}\label{Hk2=HK1ZX}
  \# H_{k_2}(\omega)=\# H_{k_1}(\omega)\cdot\# Z(x),\forall x\in H_{k_1}(\omega).
\end{equation}
By \eqref{z shadow xy}, for any $z\in Z(x),$ we have
\begin{equation}\label{distance descend}
	\begin{split}
		d_M(x,z)
\leq& d_{\omega}^{T_{N_{k_1}(\omega)}^{k_1}(\omega)}(x,z)\\
\leq & d_{\omega}^{T_{N_{k_1}(\omega)}^{k_1}(\omega)}(x,z_{k_1+1})+d_{\omega}^{T_{N_{k_1+1}(\omega)}^{k_1+1}(\omega)}(z_{k_1+1},z_{k_1+2})+\cdots+ d_{\omega}^{T_{N_{k_2-1}(\omega)}^{k_2-1}(\omega)}(z_{k_2-1},z)\\
\leq& \frac{\eta}{2^{4+k_1+1}}+\frac{\eta}{2^{4+k_1+2}}+\cdots+\frac{\eta}{2^{4+k_2}}\\
\leq &\frac{\eta}{2^{4+k_1}}.
	\end{split}
\end{equation}
Hence
\begin{displaymath}
	\begin{split}
		\Big|\int_M\psi d\mu_{k_1,\omega}-\int_M\psi d\mu_{k_2,\omega}\Big|&\overset{\eqref{Hk2=HK1ZX}}\leq \frac{1}{\# H_{k_2}(\omega)}\sum_{x\in H_{k_1}(\omega)}\sum_{z\in Z(x)}|\psi(x)-\psi(z)|\\&\overset{\eqref{distance descend}}\leq \mathrm{var}\Big(\psi,\frac{\eta}{2^{4+k_1}}\Big)\leq \mathrm{var}\Big(\psi,\frac{\eta}{2^{4+K}}\Big)<\delta.
	\end{split}
\end{displaymath}
 As a consequence, $\{\int_M\psi d\mu_{k,\omega}\}_{k=1}^{\infty}$ is a Cauchy sequence. The proof of Lemma \ref{lemma limit measure} is complete.
\end{proof}
\begin{proof}[Proof of Lemma \ref{X full measure}]
	For any $k\in\mathbb{N}, p\geq 0$, since $\mathfrak{X}_{k+p}(\omega)\subseteq \mathfrak{X}_k(\omega)$ and $\mu_{k+p,\omega}(\mathfrak{X}_{k+p}(\omega))=1,$ therefore, $\mu_{k+p,\omega}(\mathfrak{X}_k(\omega))=1.$ Note that $\mathfrak{X}_k(\omega)$ is a closed set, by the Portmanteau theorem, we have
	\begin{displaymath}
	\mu_{\omega}(\mathfrak{X}_k(\omega))\geq \limsup_{p\to\infty}\mu_{k+p,\omega}(\mathfrak{X}_k(\omega))=1.
	\end{displaymath}
    Since $\mathfrak{X}(\omega)=\cap_{k\geq 1}\mathfrak{X}_k(\omega),$ it follows that $\mu_{\omega}(\mathfrak{X}(\omega))=1.$ The proof is complete.
\end{proof}

\begin{proof}[Proof of Lemma \ref{lemma entropy distribution}]
 We first define $N(\omega)$ in the statement of Lemma \ref{lemma entropy distribution}.

\emph{Pick of $N(\omega)$.} For $\omega\in\tilde{\Omega}_\gamma$, any $k\in\mathbb{N}$ and $j=\{0,...,N_{k}(\omega)-1\}$, by the previous construction, we have $2^{m_k}\leq \hat{n}_k(\theta^{T_j^k(\omega)}\omega)\leq \hat{n}_k^M$ and $N_k(\omega)\geq 2^{N_1(\omega)(\hat{n}_1^M+m_1)+\cdots+ N_{k-1}(\omega)(\hat{n}_{k-1}^M+m_{k-1})}$. Then for any $\omega\in\tilde{\Omega}_\gamma$, there exists some $k_1(\omega)\in\mathbb{N},$ such that for any $k\geq k_1(\omega),$ and any $i\in\{0,1,...,N_{k+1}-1\}$, one has
\begin{equation}\label{estimation 1.9}
\begin{split}
&\quad\exp\Big((h_{\mu_0}(F)-4\gamma)\Big(\sum_{j=0}^{N_1(\omega)-1}\hat{n}_1(\theta^{T_j^1(\omega)}\omega)+\cdots+\sum_{j=0}^{N_k(\omega)-1}\hat{n}_k(\theta^{T_j^k(\omega)}\omega)+\sum_{j=0}^{i}\hat{n}_{k+1}(\theta^{T_j^{k+1}(\omega)}\omega)\Big)\Big)\\
&\geq \exp\Big(\Big(h_{\mu_0}(F)-\frac{9\gamma}{2}\Big)\Big(\sum_{j=0}^{N_1(\omega)-1}(\hat{n}_1(\theta^{T_j^1(\omega)}\omega)+m_1)+\cdots+\sum_{j=0}^{N_k(\omega)-1}(\hat{n}_k(\theta^{T_j^k(\omega)}\omega)+m_k)\\
&\quad\quad+\sum_{j=0}^{i}(\hat{n}_{k+1}(\theta^{T_j^{k+1}(\omega)}\omega)+m_{k+1})\Big)\Big).
\end{split}
\end{equation}
Note also that by (\ref{quotient of T}), we have $T_{N_{k-1}(\omega)}^{k-1}(\omega)/T_{N_{k}(\omega)}^{k}(\omega)\to 0$ as $k\to\infty,$ hence there exists some $k_2(\omega)>k_1(\omega)$, such that for any $k\geq k_2(\omega)$, we have
\begin{equation}\label{estimation 1.10} \frac{T_{N_{k-1}(\omega)}^{k-1}(\omega)}{T_{N_k(\omega)}^k(\omega)}+\xi_k+\xi_{k+1}+\frac{\hat{n}_{k+1}^M+m_{k+1}}{N_k(\omega)}<\frac{{\gamma}/{2}}{(h_{\mu_0}(F)-4\gamma)-\gamma/2}.
\end{equation}
Define $N(\omega)=T_{N_{k_2(\omega)}(\omega)}^{k_2(\omega)}(\omega)+1$.

Now, we start proving lemma \ref{lemma entropy distribution} for $n\geq N(\omega)$. Note that $n\geq N(\omega),$ then there exists some $k\geq k_2(\omega)$ such that	$T_{N_k(\omega)}^k(\omega)<n\leq T_{N_{k+1}(\omega)}^{k+1}(\omega).$	For any open set $B_n(\omega,x,\frac{\eta}{2^4})$, by the weak$^*$ convergence of measure, we have
	\begin{displaymath}
		\begin{split} \mu_{\omega}\Big(B_n\Big(\omega,x,\frac{\eta}{2^4}\Big)\Big)&\leq \liminf_{p\to\infty}\mu_{k+p,\omega}\Big(B_n\Big(\omega,x,\frac{\eta}{2^4}\Big)\Big).
		\end{split}
	\end{displaymath}
Next, we wish to estimate
\begin{equation}\label{mu k+p omega Bn}
  \mu_{k+p,\omega}\Big(B_n\Big(\omega,x,\frac{\eta}{2^4}\Big)\Big)=\frac{1}{\# H_{k+p}(\omega)}\cdot \#\Big\{z\in H_{k+p}(\omega):z\in B_n\Big(\omega,x,\frac{\eta}{2^4}\Big)\Big\}.
\end{equation}
As in the proof of lemma \ref{lemma contained}, there are $3$ cases:
 \begin{enumerate}
 	\item [Case 1.]$T_{N_k(\omega)}^k(\omega)<n\leq T_{N_k(\omega)}^k+l^{k,k+1}(\omega)+m_{k+1}=T_{0}^{k+1}(\omega);$
 	\item[Case 2.] There exists $j\in \{0,\dots,N_{k+1}(\omega)-1\},$ such that $T_{j}^{k+1}(\omega)<n\leq T_j^{k+1}(\omega)+\hat{n}_{k+1}(\theta^{T_j^{k+1}(\omega)}\omega);$
 	\item[Case 3.] There exists   $j\in\{0,\dots,N_{k+1}(\omega)-1\},$ such that $$T_j^{k+1}(\omega)+\hat{n}_{k+1}(\theta^{T_j^{k+1}(\omega)}\omega)<n\leq T_{j+1}^{k+1}(\omega)=T_j^{k+1}+\hat{n}_{k+1}(\theta^{T_j^{k+1}(\omega)}\omega)+m_{k+1}+l^{k+1}_{j+1}(\omega).$$
 \end{enumerate}

\emph{In Case 1.} We divide the proof into $3$ steps.
Firstly, we show $\#\{z\in H_k(\omega):z\in B_n(\omega,x,\frac{\eta}{2^4})\}\leq 1$. If there are $z_1\neq z_2\in H_k(\omega),$ such that $z_1,z_2\in B_n(\omega,x,\frac{\eta}{2^4}),$ then
\begin{equation*}
  	d_\omega^n(z_1,z_2)<\frac{\eta}{2^4}\times 2=\frac{\eta}{2^3}.
\end{equation*}
However, we have $d_{\omega}^n(z_1,z_2)\geq d_{\omega}^{T_{N_k(\omega)}^k(\omega)}(z_1,z_2)> \frac{3\eta}{8}$ by \eqref{eq 3 lemma 6.7}, which leads to a contradiction.

Secondly, we show for any $p\geq 1 $,
\begin{equation}\label{case 1 step 2}
  \#\{z\in H_{k+p}(\omega):z\in B_n(\omega,x,\frac{\eta}{2^4})\}\leq \# D_{k+1}(\theta^{T_0^{k+1}}\omega)\cdots\# D_{k+p}(\theta^{T_0^{k+p}}\omega).
\end{equation}
 If there are different points $z_1, z_2\in H_{k+p}(\omega)\cap B_n(\omega,x,\frac{\eta}{2^4})$ such that $z_1$ descends from $x_1\in H_k(\omega)$ and $z_2$ descends from $x_2\in H_k(\omega)$  defined as \eqref{def descend}, then we claim that $x_1=x_2$. In fact, if $x_1\neq x_2,$ by \eqref{eq 3 lemma 6.7}, we have $d_{\omega}^{T_{N_k(\omega)}^k(\omega)}(x_1,x_2)>\frac{3\eta}{8}.$ But, we also have
\begin{align*}
   d_{\omega}^{T_{N_k(\omega)}^k(\omega)}(x_1,x_2)&\leq d_{\omega}^{T_{N_k(\omega)}^k(\omega)}(x_1,z_1)+d_{\omega}^{T_{N_k(\omega)}^k(\omega)}(z_1,x)+d_{\omega}^{T_{N_k(\omega)}^k(\omega)}(x,z_2)+d_{\omega}^{T_{N_k(\omega)}^k(\omega)}(z_2,x_2)\\
   &\overset{\eqref{distance descend}}< \frac{\eta}{2^{4+k}}+\frac{\eta}{2^{4}}+\frac{\eta}{2^{4}}+\frac{\eta}{2^{4+k}}\\
   &\leq \frac{\eta}{4},
\end{align*}which leads to a contradiction.
As a consequence of \eqref{mu k+p omega Bn} and \eqref{case 1 step 2}, we have
\begin{displaymath} \mu_{k+p,\omega}\Big(B_n\Big(\omega,x,\frac{\eta}{2^4}\Big)\Big)\leq \frac{\# D_{k+1}(\theta^{T_0^{k+1}}\omega)\cdots\# D_{k+p}(\theta^{T_0^{k+p}}\omega)}{\# H_{k+p}(\omega)}=\frac{1}{\# H_k(\omega)}.
\end{displaymath}

Thirdly, we claim $\# H_k(\omega)\geq \exp((h_{\mu_0}(F)-5\gamma)n)$. In fact,
\begin{align*}
		\# H_k(\omega)&=\# D_1(\omega)\cdot \# D_2(\theta^{T_0^2(\omega)}\omega)\cdots\# D_k(\theta^{T_0^k(\omega)}\omega)=\prod_{i=1}^{k}\prod_{j=0}^{N_i(\omega)-1}\# C_i(\theta^{T_j^i(\omega)}\omega)\\
&= \prod_{i=1}^{k}\prod_{j=0}^{N_i-1}M(\alpha,4\delta_i,\hat{n}_i(\theta^{T_j^i(\omega)}\omega),\frac{\eta}{2},\theta^{T_j^i(\omega)}\omega)\overset{\eqref{frac 1 hat n M entropy}}\geq \prod_{i=1}^{k}\prod_{j=0}^{N_i-1}\exp (\hat{n}_i(\theta^{T_j^i(\omega)}\omega)(h_{\mu_0}(F)-4\gamma))\\
&\overset{\eqref{estimation 1.9}}\geq \exp\Big(\Big(h_{\mu_0}(F)-\frac{9\gamma}{2}\Big)\Big(\sum_{j=0}^{N_1-1}(\hat{n}_1(\theta^{T_j^1(\omega)}\omega)+m_1)+\cdots+\sum_{j=0}^{N_k-1}(\hat{n}_k(\theta^{T_j^k(\omega)}\omega)+m_k)\Big)\Big).
\end{align*} We notice the following fact
\begin{align*}
 &\quad  n-\Big(\sum_{j=0}^{N_1(\omega)-1}(\hat{n}_1(\theta^{T_j^1(\omega)}\omega)+m_1)+\cdots+\sum_{j=0}^{N_k(\omega)-1}(\hat{n}_k(\theta^{T_j^k(\omega)}\omega)+m_k)\Big)\\
 &\leq  T_{N_{k-1}(\omega)}^{k-1}(\omega)+\xi_kT_{N_k(\omega)}^k(\omega)+m_{k+1}+\xi_{k+1}n,
\end{align*}and therefore,
\begin{align*}
  &\frac{n-\left(\sum_{j=0}^{N_1(\omega)-1}(\hat{n}_1(\theta^{T_j^1(\omega)}\omega)+m_1)+\cdots+\sum_{j=0}^{N_k(\omega)-1}(\hat{n}_k(\theta^{T_j^k(\omega)}\omega)+m_k)\right)}{n}\\
  \leq &\frac{T_{N_{k-1}(\omega)}^{k-1}(\omega)+\xi_kT_{N_k(\omega)}^k(\omega)+m_{k+1}+\xi_{k+1}n}{n}\\
 \leq & \frac{T_{N_{k-1}(\omega)}^{k-1}(\omega)}{T_{N_{k}(\omega)}^{k}(\omega)}+\xi_k+\xi_{k+1}+\frac{m_{k+1}}{N_k(\omega)}\\
  \overset{\eqref{estimation 1.10}}\leq & \frac{\gamma/2}{h_{\mu_0}(F)-\frac{9}{2}\gamma}.
\end{align*}So we have
\begin{equation*}
  \Big(h_{\mu_0}(F)-\frac{9\gamma}{2}\Big)\Big(\sum_{j=0}^{N_1-1}(\hat{n}_1(\theta^{T_j^1(\omega)}\omega)+m_1)+\cdots+\sum_{j=0}^{N_k-1}(\hat{n}_k(\theta^{T_j^k(\omega)}\omega)+m_k)\Big)\geq n(h_{\mu_0}(F)-5\gamma).
\end{equation*}As a consequence, we obtain $\#H_k(\omega)\geq \exp(n(h_{\mu_0}(F)-5\gamma))$.
Hence, in case 1, we have
\begin{equation*}
 \mu_{k+p,\omega}\Big(B_n\Big(\omega,x,\frac{\eta}{2^4}\Big)\Big)\leq \exp(-n(h_{\mu_0}(F)-5\gamma)).
\end{equation*}

\emph{In Case 2.} i.e., $\exists j\in \{0,\dots,N_{k+1}(\omega)-1\}$ such that $T_{j}^{k+1}(\omega)<n\leq T_j^{k+1}(\omega)+\hat{n}_{k+1}(\theta^{T_j^{k+1}(\omega)}\omega).$ We divide the proof into $4$ steps.
Firstly, we show $\#\{z\in H_k(\omega):z\in B_n(\omega,x,\frac{\eta}{2^3})\}\leq 1.$ If there are $z_1\neq z_2\in H_k(\omega)\cap B_n(\omega,x,\frac{\eta}{2^3}),$ then
\begin{displaymath}
	d_\omega^n(z_1,z_2)<2\times \frac{\eta}{2^3}=\frac{\eta}{4}.
\end{displaymath}
But by \eqref{eq 3 lemma 6.7}, we have $d_\omega^n(z_1,z_2)\geq d_{\omega}^{T_{N_k(\omega)}^k(\omega)}(z_1,z_2)> \frac{3\eta}{8},$ which leads to a contradiction.

Secondly, we show $\#\{z\in H_{k+1}(\omega):z\in B_n(\omega,x,\frac{\eta}{2^3})\}\leq \prod_{i=j}^{N_{k+1}(\omega)-1} \# C_{k+1}(\theta^{T_i^{k+1}(\omega)}\omega).$ If there are $z_1\neq z_2\in H_{k+1}(\omega)\cap B_n(\omega,x,\frac{\eta}{2^3}),$ with
\begin{displaymath}
	\begin{split}
		z_1&=z(x_1,y_1), x_1\in H_k(\omega),y_1\in D_{k+1}(\theta^{T_0^{k+1}(\omega)}\omega), y_1=y\big(a_0^{k+1},\dots,a_{N_{k+1}(\omega)-1}^{k+1}\big),\\	
		z_2&=z(x_2,y_2), x_2\in H_k(\omega),y_2\in D_{k+1}(\theta^{T_0^{k+1}(\omega)}\omega), y_2=y\big(b_0^{k+1},\dots,b_{N_{k+1}(\omega)-1}^{k+1}\big),
	\end{split}
\end{displaymath}
where $a_i^{k+1},b_i^{k+1}\in C_{k+1}(\theta^{T_i^{k+1}(\omega)}\omega)$ for $i\in\{0,\dots,N_{k+1}(\omega)-1\}.$ We claim that $x_1=x_2$ and $a_i^{k+1}=b_i^{k+1}$ for $i\in \{0,1,\dots,j-1\}.$ If $x_1\neq x_2\in H_k(\omega)$, by \eqref{eq 3 lemma 6.7}, we have $d_{\omega}^{T_{N_k(\omega)}^k(\omega)}(x_1,x_2)>\frac{3\eta}{8},$. But
\begin{displaymath}
	\begin{split}
		d_{\omega}^{T_{N_k(\omega)}^k(\omega)}(x_1,x_2)&\leq d_{\omega}^{T_{N_k(\omega)}^k(\omega)}(x_1,z_1)+d_{\omega}^{T_{N_k(\omega)}^k(\omega)}(z_1,x)+d_{\omega}^{T_{N_k(\omega)}^k(\omega)}(x,z_2)+d_{\omega}^{T_{N_k(\omega)}^k(\omega)}(z_2,x_2)\\
&\overset{\eqref{z shadow xy}}\leq \frac{\eta}{2^{4+k+1}}+\frac{\eta}{2^3}+\frac{\eta}{2^3}+\frac{\eta}{2^{4+k+1}}\\&\leq \frac{5\eta}{16},
	\end{split}
\end{displaymath}
which leads to a contradiction. Hence $x_1=x_2.$ Next, we prove that $a_i^{k+1}=b_i^{k+1}$ for $i\in\{0,1,\dots,j-1\}$. If $j=0,$ there is nothing to prove. Suppose $j\geq 1$ and there exists $i$ with $0\leq i\leq j-1$ such that $a_{i}^{k+1}\neq b_i^{k+1}.$ On the one hand, by \eqref{eq y shadowing xk+1} and \eqref{z shadow xy}, one has
\begin{displaymath}
	\begin{split}
		&\quad d_{\theta^{T_i^{k+1}(\omega)}\omega}^{\hat{n}_{k+1}(\theta^{T_i^{k+1}(\omega)}\omega)}\Big(a_i^{k+1},b_i^{k+1}\Big)\\
&\leq d_{\theta^{T_i^{k+1}(\omega)}\omega}^{\hat{n}_{k+1}(\theta^{T_i^{k+1}(\omega)}\omega)}\Big(a_i^{k+1},F_{\theta^{T_0^{k+1}(\omega)}\omega}^{T_i^{k+1}(\omega)-T_0^{k+1}(\omega)}y_1\Big)+d_{\theta^{T_i^{k+1}(\omega)}\omega}^{\hat{n}_{k+1}(\theta^{T_i^{k+1}(\omega)}\omega)}\Big(F_{\theta^{T_0^{k+1}(\omega)}\omega}^{T_i^{k+1}(\omega)-T_0^{k+1}(\omega)}y_1,F_{\omega}^{T_i^{k+1}(\omega)}z_1\Big)\\
&\quad+  d_{\theta^{T_i^{k+1}(\omega)}\omega}^{\hat{n}_{k+1}(\theta^{T_i^{k+1}(\omega)}\omega)}\Big(F_{\omega}^{T_i^{k+1}(\omega)}z_1,F_{\omega}^{T_i^{k+1}(\omega)}z_2\Big)+  d_{\theta^{T_i^{k+1}(\omega)}\omega}^{\hat{n}_{k+1}(\theta^{T_i^{k+1}(\omega)}\omega)}\Big(F_{\omega}^{T_i^{k+1}(\omega)}z_2,F_{\theta^{T_0^{k+1}(\omega)}\omega}^{T_i^{k+1}(\omega)-T_0^{k+1}(\omega)}y_2\Big)   \\&\quad\quad+   d_{\theta^{T_i^{k+1}(\omega)}\omega}^{\hat{n}_{k+1}(\theta^{T_i^{k+1}(\omega)}\omega)}\Big(F_{\theta^{T_0^{k+1}(\omega)}\omega}^{T_i^{k+1}(\omega)-T_0^{k+1}(\omega)}y_2,b_i^{k+1}\Big)\\
&\leq \frac{\eta}{2^{4+k+1}}+\frac{\eta}{2^{4+k+1}}+\frac{\eta}{2^3}\times 2+\frac{\eta}{2^{4+k+1}}+\frac{\eta}{2^{4+k+1}}\leq \frac{\eta}{2^2}+\frac{\eta}{2^{3+k}}<\frac{3\eta}{8}.
  	\end{split}
\end{displaymath}
On the other hand, $a_i^{k+1}\neq b_i^{k+1}\in C_{k+1}(\theta^{T_i^{k+1}}\omega)$ are $(\theta^{T_i^{k+1}}\omega,\frac{\eta}{2},\hat{n}_{k+1}(\theta^{T_i^{k+1}}\omega))$-separated, which leads to a contradiction. Hence, there are at most $\prod_{i=j}^{N_{k+1}(\omega)-1} \# C_{k+1}(\theta^{T_i^{k+1}(\omega)}\omega)$ points lying in $H_{k+1}(\omega)\cap B_n(\omega,x,\frac{\eta}{2^3}).$

Thirdly, we show for any $p\geq 1$,
\begin{equation*}
  \# \{z\in H_{k+p}(\omega):z\in B_n(\omega,x,\frac{\eta}{2^4})\}\leq \left( \prod_{i=j}^{N_{k+1}(\omega)-1}\# C_{k+1}(\theta^{T_i^{k+1}(\omega)}\omega)\right)\cdot \left(\prod_{i=2}^{p}\# D_{k+i}(\theta^{T_0^{k+i}(\omega)}\omega)\right).
\end{equation*} We prove it by showing that $z\in H_{k+p}(\omega)\cap B_n(\omega,x,\frac{\eta}{2^4})$ must descend from the points of $H_{k+1}(\omega)\cap B_n(\omega,x,\frac{\eta}{2^3}).$ Suppose that we have $z_1\in H_{k+1}(\omega)$ and $z_p\in H_{k+p}(\omega)\cap B_{n}(\omega,x,\frac{\eta}{2^4}),$ where $z_p$ descends from $z_1$.
 Denote $z_p=z(z_{p-1},y_p)$ for $z_{p-1}\in H_{k+p-1}(\omega),y_p\in D_{k+p}(\theta^{T_0^{k+p}}\omega),\dots,z_2=z(z_1,y_2)$ for $y_2\in D_{k+2}(\theta^{T_0^{k+2}}\omega).$ Then by \eqref{z shadow xy}, one has
 \begin{displaymath}
 	\begin{split}
 		d_{\omega}^n(z_1,z_p)&\leq d_{\omega}^{T_{N_{k+1}(\omega)}^{k+1}(\omega)}(z_1,z_2)+d_{\omega}^{T_{N_{k+2}(\omega)}^{k+2}(\omega)}(z_2,z_3)+\cdots+ d_{\omega}^{T_{N_{k+p-1}(\omega)}^{k+p-1}(\omega)}(z_{p-1},z_p)\\
 &\leq \frac{\eta}{2^{4+k+2}}+\frac{\eta}{2^{4+k+3}}+\cdots+\frac{\eta}{2^{k+p}}\\
 &\leq\frac{\eta}{2^{4+k+1}}.
 	\end{split}
 \end{displaymath}
Hence $d_{\omega}^n(x,z_1)\leq d_{\omega}^n(x,z_p)+d_{\omega}^n(z_1,z_p)<\frac{\eta}{2^4}+\frac{\eta}{2^{4+k+1}}<\frac{\eta}{2^3},$ which implies that $z_1\in B_n(\omega,x,\frac{\eta}{2^3}).$ Therefore
\begin{displaymath}
	\begin{split}
	&\quad	\# \Big\{z\in H_{k+p}(\omega):z\in B_n\Big(\omega,x,\frac{\eta}{2^4}\Big)\Big\}\\
&\leq \# H_{k+1}(\omega)\cap B_n(\omega,x,\frac{\eta}{2^3})\# D_{k+2}(\theta^{T_0^{k+2}(\omega)}\omega)\cdots\# D_{k+p}(\theta^{T_0^{k+p}(\omega)}\omega)\\
&\leq\left(\prod_{i=j}^{N_{k+1}(\omega)-1}\# C_{k+1}(\theta^{T_j^{k+1}(\omega)}\omega)\right)\cdot\prod_{i=2}^p D_{k+i}(\theta^{T_0^{k+i}(\omega)}\omega),
	\end{split}
\end{displaymath}where the last inequality is given by the second step.
It follows by \eqref{number Hk} and \eqref{number Dk+1} that
\begin{equation}\label{estimation 1.12}
	\begin{split}
	\mu_{k+p,\omega}\Big(B_n\Big(\omega,x,\frac{\eta}{2^4}\Big)\Big)&\leq\frac{\left(\prod_{i=j}^{N_{k+1}(\omega)-1}\# C_{k+1}(\theta^{T_i^{k+1}(\omega)}\omega)\right)\cdot\left(\prod_{i=2}^{p}D_{k+i}(\theta^{T_0^{k+i}(\omega)}\omega)\right)}{\# H_{k+p}(\omega)}\\&=\frac{1}{\# H_k(\omega)\prod_{i=0}^{j-1}\# C_{k+1}(\theta^{T_i^{k+1}(\omega)}\omega)},
	\end{split}
\end{equation}with the convention that $\prod_{i=0}^{-1}=1.$

Fourthly, we show
\begin{equation}\label{estimate number hkck1}
  \# H_k(\omega)\prod_{i=0}^{j-1}\# C_{k+1}(\theta^{T_i^{k+1}}\omega)\geq \exp((h_{\mu_0}(F)-5\gamma)n).
\end{equation}  In fact,
\begin{align*}
 &\quad\# H_k(\omega)\cdot\prod_{i=0}^{j-1} C_{k+1}(\theta^{T_i^{k+1}}\omega)=\left(\prod_{t=1}^{k}\prod_{i=0}^{N_t(\omega)-1}\# C_t(\theta^{T_i^t(\omega)}\omega)\right)\cdot \left(\prod_{i=0}^{j-1} \# C_{k+1}(\theta^{T_i^{k+1}(\omega)}\omega)\right)\nonumber\\
 &=\prod_{t=1}^{k}\prod_{i=0}^{N_t(\omega)-1}M(\alpha,4\delta_t,\hat{n}_t(\theta^{T_i^t(\omega)}\omega),\frac{\eta}{2},\theta^{T_i^t(\omega)}\omega)\cdot \prod_{i=0}^{j-1}M(\alpha,4\delta_{k+1},\hat{n}_{k+1}(\theta^{T_i^{k+1}(\omega)}\omega),\frac{\eta}{2},\theta^{T_i^{k+1}(\omega)}\omega) \nonumber\\
 &\geq\exp\Big((h_{\mu_0}(F)-4\gamma)\Big(\sum_{i=0}^{N_1-1}\hat{n}_1(\theta^{T_i^1(\omega)}\omega)+\cdots+\sum_{i=0}^{N_k-1}\hat{n}_k(\theta^{T_i^k(\omega)}\omega)+\sum_{i=0}^{j-1}\hat{n}_{k+1}(\theta^{T_i^{k+1}(\omega)}\omega)\Big)\Big)\nonumber\\
 &\overset{\eqref{estimation 1.9}}\geq \exp\Big(\Big(h_{\mu_0}(F)-\frac{9\gamma}{2}\Big)\Big(\sum_{i=0}^{N_1(\omega)-1}(\hat{n}_1(\theta^{T_i^1(\omega)}\omega)+m_1)+\cdots+\sum_{i=0}^{N_k(\omega)-1}(\hat{n}_k(\theta^{T_i^k(\omega)}\omega)+m_k)\nonumber\\
 &\quad\quad\quad\quad\quad\quad\quad\quad\quad\quad\quad\quad\quad\quad+\sum_{i=0}^{j-1}(\hat{n}_{k+1}(\theta^{T_i^{k+1}(\omega)}\omega)+m_{k+1})\Big)\Big).
\end{align*}We notice the following fact
\begin{align*}
 & n-\left(\sum_{i=0}^{N_1(\omega)-1}(\hat{n}_1(\theta^{T_i^1(\omega)}\omega)+m_1)+\cdots+\sum_{i=0}^{N_k(\omega)-1}(\hat{n}_k(\theta^{T_i^k(\omega)}\omega)+m_k)\right.\\
 &\quad\quad\left.+\sum_{i=0}^{j-1}(\hat{n}_{k+1}(\theta^{T_i^{k+1}(\omega)}\omega)+m_{k+1})\right)\\
\leq & T_{N_{k-1}(\omega)}^{k-1}(\omega)+l^{k-1,k}(\omega)+\sum_{i=1}^{N_k(\omega)-1}l_i^k(\omega)+l^{k,k+1}(\omega)+\sum_{i=1}^{j}l_i^{k+1}(\omega)+n-T_{j}^{k+1}(\omega)\nonumber\\
\leq &T_{N_{k-1}(\omega)}^{k-1}(\omega)+\xi_kn+\xi_{k+1}n+\hat{n}_{k+1}^M
\leq \left(\frac{T_{N_{k-1}(\omega)}^{k-1}(\omega)}{T_{N_{k}(\omega)}^{k}(\omega)}+\xi_k+\xi_{k+1}+\frac{\hat{n}_{k+1}^M}{N_k(\omega)}\right)n\\
\overset{\eqref{estimation 1.10}}\leq & \frac{\gamma/2}{h_{\mu_0}(F)-\frac{9\gamma}{2}}\cdot n
\end{align*}
with the convention that $\sum_{i=1}^{0}=0$. As a consequence, we have
\begin{align*}
   & \Big(h_{\mu_0}(F)-\frac{9\gamma}{2}\Big)\Big(\sum_{i=0}^{N_1(\omega)-1}(\hat{n}_1(\theta^{T_i^1(\omega)}\omega)+m_1)+\cdots+\sum_{i=0}^{N_k(\omega)-1}(\hat{n}_k(\theta^{T_i^k(\omega)}\omega)+m_k)\nonumber\\
 &\quad\quad\quad\quad\quad\quad\quad\quad\quad\quad\quad\quad\quad\quad+\sum_{i=0}^{j-1}(\hat{n}_{k+1}(\theta^{T_i^{k+1}(\omega)}\omega)+m_{k+1})\Big)\\
 \geq & (h_{\mu_0}(F)-5\gamma )n
\end{align*}and \eqref{estimate number hkck1} is proved. Therefore, in case 2, by \eqref{estimation 1.12} and \eqref{estimate number hkck1}, we have
\begin{displaymath} \mu_{k+p,\omega}\Big(B_n\Big(\omega,x,\frac{\eta}{2^4}\Big)\Big)\leq \exp(-n(h_{\mu_0}(F)-5\gamma)).
\end{displaymath}

\emph{In Case 3.} i.e., there exists   $j\in\{0,\dots,N_{k+1}(\omega)-1\},$ such that $$T_j^{k+1}(\omega)+\hat{n}_{k+1}(\theta^{T_j^{k+1}(\omega)}\omega)<n\leq T_{j+1}^{k+1}(\omega)=T_j^{k+1}+\hat{n}_{k+1}(\theta^{T_j^{k+1}(\omega)}\omega)+m_{k+1}+l^{k+1}_{j+1}(\omega).$$
We also divide our proof into $4$ steps. Firstly, same as the proof of step 1 in case 2, we can show that $\#\{z\in H_k(\omega):z\in B_n(\omega,x,\frac{\eta}{2^3})\}\leq 1.$

Secondly, we show that $\#\{z\in H_{k+1}(\omega):z\in B_n(\omega,x,\frac{\eta}{2^3})\}\leq \prod_{i=j+1}^{N_{k+1}(\omega)-1} \# C_{k+1}(\theta^{T_i^{k+1}(\omega)}\omega)$ with the convention that $\prod_{i=N_{k+1}(\omega)}^{N_{k+1}(\omega)-1}=1$. If there are $z_1\neq z_2\in H_{k+1}(\omega)\cap B_n(\omega,x,\frac{\eta}{2^3}),$ with
\begin{displaymath}
	\begin{split}
		z_1&=z(x_1,y_1), x_1\in H_k(\omega),y_1\in D_{k+1}(\theta^{T_0^{k+1}(\omega)}\omega), y_1=y\big(a_0^{k+1},\dots,a_{N_{{k+1}}(\omega)-1}^{k+1}\big),\\	
		z_2&=z(x_2,y_2), x_2\in H_k(\omega),y_2\in D_{k+1}(\theta^{T_0^{k+1}(\omega)}\omega), y_2=y\big(b_0^{k+1},\dots,b_{N_{{k+1}}(\omega)-1}^{k+1}\big),
	\end{split}
\end{displaymath}
where $a_i^{k+1},b_i^{k+1}\in C_{k+1}(\theta^{T_i^{k+1}(\omega)}\omega)$ for $i\in\{0,\dots,N_{k+1}(\omega)-1\}.$ We claim that $x_1=x_2$ and $a_i^{k+1}=b_i^{k+1}$ for $i\in \{0,1,\dots,j\}.$ In fact, the same proof as in step 2 of case 2 indicates that $x_1=x_2$ and $a_i^{k+1}=b_i^{k+1}$ for $i\in\{0,1,\dots,j-1\}$. It remains to show that $a_j^{k+1}=b_j^{k+1}.$ If $a_j^{k+1}\neq b_j^{k+1}$, on the one hand, by \eqref{eq y shadowing xk+1} and \eqref{z shadow xy}, one has
\begin{displaymath}
	\begin{split}
		&\quad d_{\theta^{T_j^{k+1}(\omega)}\omega}^{\hat{n}_{k+1}(\theta^{T_j^{k+1}(\omega)}\omega)}\Big(a_j^{k+1},b_j^{k+1}\Big)\\
&\leq d_{\theta^{T_j^{k+1}(\omega)}\omega}^{\hat{n}_{k+1}(\theta^{T_j^{k+1}(\omega)}\omega)}\Big(a_j^{k+1},F_{\theta^{T_0^{k+1}(\omega)}\omega}^{T_j^{k+1}(\omega)-T_0^{k+1}(\omega)}y_1\Big)+d_{\theta^{T_j^{k+1}(\omega)}\omega}^{\hat{n}_{k+1}(\theta^{T_j^{k+1}(\omega)}\omega)}\Big(F_{\theta^{T_0^{k+1}(\omega)}\omega}^{T_j^{k+1}(\omega)-T_0^{k+1}(\omega)}y_1,F_{\omega}^{T_j^{k+1}(\omega)}z_1\Big) \\ &\quad+  d_{\theta^{T_j^{k+1}(\omega)}\omega}^{\hat{n}_{k+1}(\theta^{T_j^{k+1}(\omega)}\omega)}\Big(F_{\omega}^{T_j^{k+1}(\omega)}z_1,F_{\omega}^{T_j^{k+1}(\omega)}z_2\Big) +  d_{\theta^{T_j^{k+1}(\omega)}\omega}^{\hat{n}_{k+1}(\theta^{T_j^{k+1}(\omega)}\omega)}\Big(F_{\omega}^{T_j^{k+1}(\omega)}z_2,F_{\theta^{T_0^{k+1}}\omega}^{T_j^{k+1}(\omega)-T_0^{k+1}(\omega)}y_2\Big)\\   &\quad+   d_{\theta^{T_i^{k+1}(\omega)}\omega}^{\hat{n}_{k+1}(\theta^{T_j^{k+1}(\omega)}\omega)}\Big(F_{\theta^{T_0^{k+1}(\omega)}\omega}^{T_j^{k+1}(\omega)-T_0^{k+1}(\omega)}y_2,b_j^{k+1}\Big)\\
&\leq \frac{\eta}{2^{4+k+1}}+\frac{\eta}{2^{4+k+1}}+\frac{\eta}{2^3}\times 2+\frac{\eta}{2^{4+k+1}}+\frac{\eta}{2^{4+k+1}}\leq \frac{\eta}{2^2}+\frac{\eta}{2^{3+k}}<\frac{3\eta}{8}.
	\end{split}
\end{displaymath}
On the other hand, $a_j^{k+1}\neq b_j^{k+1}\in C_{k+1}(\theta^{T_j^{k+1}(\omega)}\omega)$ are $(\theta^{T_j^{k+1}(\omega)}\omega,\frac{\eta}{2},\hat{n}_{k+1}(\theta^{T_j^{k+1}(\omega)}\omega))$-separated, which leads to a contradiction.

Thirdly, we show that for any $p\geq 2$,
\begin{equation*}
  \# \{z\in H_{k+p}(\omega):z\in B_n(\omega,x,\frac{\eta}{2^4})\}\leq \left( \prod_{i=j+1}^{N_{k+1}(\omega)-1}\# C_{k+1}(\theta^{T_i^{k+1}(\omega)}\omega)\right)\cdot\left(\prod_{i=2}^{p}\# D_{k+i}(\theta^{T_0^{k+i}(\omega)}\omega)\right),
\end{equation*}with the convention that $\prod_{i=N_{k+1}(\omega)}^{N_{k+1}(\omega)-1}=1$.
Exactly same proof as in step $3$ of cases $2$ indicates that $z\in H_{{k+p}}(\omega)\cap B_n(\omega,x,\frac{\eta}{2^4})$ must descends from some point in $H_{k+1}(\omega)\cap B_n(\omega,x,\frac{\eta}{2^3}).$
Therefore,
\begin{displaymath}
	\begin{split}
		&\quad\# \Big\{z\in H_{k+p}(\omega):z\in B_n\Big(\omega,x,\frac{\eta}{2^4}\Big)\Big\}\\
&\leq \# H_{k+1}(\omega)\cap B_n(\omega,x,\frac{\eta}{2^3})\# D_{k+2}(\theta^{T_0^{k+2}(\omega)}\omega)\cdots\# D_{k+p}(\theta^{T_0^{k+p}(\omega)}\omega)\\
&\leq\left(\prod_{i=j+1}^{N_{k+1}(\omega)-1}\# C_{k+1}(\theta^{T_j^{k+1}(\omega)}\omega)\right)\cdot\left(\prod_{i=2}^p D_{k+i}(\theta^{T_0^{k+i}(\omega)}\omega)\right).
	\end{split}
\end{displaymath}
It follows that
\begin{equation}
	\begin{split}
		\mu_{k+p,\omega}\Big(B_n\Big(\omega,x,\frac{\eta}{2^4}\Big)\Big)&\leq\frac{\prod_{i=j+1}^{N_{k+1}(\omega)-1}\# C_{k+1}(\theta^{T_i^{k+1}(\omega)}\omega)\cdot\prod_{i=2}^{p}D_{k+i}(\theta^{T_0^{k+i}(\omega)}\omega)}{\# H_{k+p}(\omega)}\\&=\frac{1}{\# H_k(\omega)\prod_{i=0}^{j}\# C_{k+1}(\theta^{T_i^{k+1}(\omega)}\omega)}.
	\end{split}
\end{equation}

Fourthly, we show $\# H_k(\omega)\prod_{i=0}^{j}\# C_{k+1}(\theta^{T_i^{k+1}(\omega)}\omega)\geq \exp((h_{\mu_0}(F)-5\gamma)n).$
 Using (\ref{estimation 1.9}), we have
\begin{displaymath}
	\begin{split}
		&\quad\# H_k(\omega)\cdot\prod_{i=0}^{j} C_{k+1}(\theta^{T_i^{k+1}(\omega)}\omega)=\prod_{t=1}^{k}\prod_{i=0}^{N_t(\omega)-1}\# C_t(\theta^{T_i^t(\omega)}\omega)\cdot \prod_{i=0}^{j} \# C_{k+1}(\theta^{T_i^{k+1}(\omega)}\omega)\nonumber\\
 &=\prod_{t=1}^{k}\prod_{i=0}^{N_t(\omega)-1}M(\alpha,4\delta_t,\hat{n}_t(\theta^{T_i^t(\omega)}\omega),\frac{\eta}{2},\theta^{T_i^t(\omega)}\omega)\cdot \prod_{i=0}^{j}M(\alpha,4\delta_{k+1},\hat{n}_{k+1}(\theta^{T_i^{k+1}(\omega)}\omega),\frac{\eta}{2},\theta^{T_i^{k+1}(\omega)}\omega) \\
&\geq\exp\Big((h_{\mu_0}(F)-4\gamma)\Big(\sum_{i=0}^{N_1(\omega)-1}\hat{n}_1(\theta^{T_i^1(\omega)}\omega)+\cdots+\sum_{i=0}^{N_k(\omega)-1}\hat{n}_k(\theta^{T_i^k(\omega)}\omega)+\sum_{i=0}^{j}\hat{n}_{k+1}(\theta^{T_i^{k+1}(\omega)}\omega)\Big)\Big)\\&\geq \exp\Big(\Big(h_{\mu_0}(F)-4\gamma-\frac{\gamma}{2}\Big)\Big(\sum_{i=0}^{N_1(\omega)-1}(\hat{n}_1(\theta^{T_i^1(\omega)}\omega)+m_1)+\cdots+\sum_{i=0}^{N_k(\omega)-1}(\hat{n}_k(\theta^{T_i^k(\omega)}\omega)+m_k)\\
&\quad+\sum_{i=0}^{j-1}(\hat{n}_{k+1}(\theta^{T_i^{k+1}(\omega)}\omega)+m_{k+1})+\hat{n}_{k+1}(\theta^{T_j^{k+1}(\omega)}\omega)\Big)\Big).
	\end{split}
\end{displaymath} We notice the following fact that in Case 3
\begin{displaymath}
	\begin{split}
		&\quad n-\left(\sum_{i=0}^{N_1(\omega)-1}(\hat{n}_1(\theta^{T_i^1(\omega)}\omega)+m_1)+\cdots+\sum_{i=0}^{N_k(\omega)-1}(\hat{n}_k(\theta^{T_i^k(\omega)}\omega)+m_k)\right.
\\ &\quad\quad\quad\quad\left.+\sum_{i=0}^{j-1}(\hat{n}_{k+1}(\theta^{T_i^{k+1}(\omega)}\omega)+m_{k+1})+\hat{n}_{k+1}(\theta^{T_j^{k+1}(\omega)}\omega)\right)\\
\leq & T_{N_{k-1}(\omega)}^{k-1}(\omega)+l^{k-1,k}(\omega)+\sum_{i=1}^{N_k(\omega)-1}l_i^k(\omega)+l^{k,k+1}(\omega)+\sum_{i=1}^{j}l_i^{k+1}(\omega)\\
&\quad\quad+(n-T_{j}^{k+1}(\omega)- \hat{n}_{k+1}(\theta^{T_j^{k+1}(\omega)}\omega))\\
\leq & T_{N_{k-1}(\omega)}^{k-1}(\omega)+\xi_kn+\xi_{k+1}n+m_{k+1}\\
\leq &\left(\frac{T_{N_{k-1}(\omega)}^{k-1}(\omega)}{T_{N_{k}(\omega)}^{k}(\omega)}+\xi_k+\xi_{k+1}+\frac{m_{k+1}}{N_{k}(\omega)}\right)\cdot n\\
\overset{\eqref{estimation 1.10}}\leq &\frac{\gamma/2}{h_{\mu_0}(F)-\frac{9\gamma}{2}}\cdot n.
	\end{split}
\end{displaymath} As a consequence,
\begin{align*}
  & \Big(h_{\mu_0}(F)-4\gamma-\frac{\gamma}{2}\Big)\Big(\sum_{i=0}^{N_1(\omega)-1}(\hat{n}_1(\theta^{T_i^1(\omega)}\omega)+m_1)+\cdots+\sum_{i=0}^{N_k(\omega)-1}(\hat{n}_k(\theta^{T_i^k(\omega)}\omega)+m_k)\\
&\quad+\sum_{i=0}^{j-1}(\hat{n}_{k+1}(\theta^{T_i^{k+1}(\omega)}\omega)+m_{k+1})+\hat{n}_{k+1}(\theta^{T_j^{k+1}(\omega)}\omega)\Big)\\
\geq &(h_{\mu_0}(F)-5\gamma)n,
\end{align*}and $\# H_k(\omega)\cdot\prod_{i=0}^{j} C_{k+1}(\theta^{T_i^{k+1}(\omega)}\omega)\geq \exp((h_{\mu_0}(F)-5\gamma)n)$.
Therefore,
\begin{displaymath}
	\mu_{k+p,\omega}\Big(B_n\Big(\omega,x,\frac{\eta}{2^4}\Big)\Big)\leq \exp(-n(h_{\mu_0}(F)-5\gamma)).
\end{displaymath}
In all three cases, we conclude
\begin{displaymath}
	\mu_{\omega}\Big(B_n(\omega,x,\frac{\eta}{2^4})\Big)\leq \varliminf_{p\to\infty}\mu_{k+p,\omega}\Big(B_n\Big(\omega,x,\frac{\eta}{2^4}\Big)\Big)\leq \exp(-n(h_{\mu_0}(F)-5\gamma)).
\end{displaymath}
The proof of Lemma \ref{lemma entropy   distribution} is complete.
\end{proof}
\subsection*{Acknowledgment}
The authors would like to thank Wen Huang for many valuable comments. Partial work was finished at the Brigham Young University and constitutes part of the second author's PhD thesis. The second author would like to thank Kening Lu for many useful discussions.

 \bibliographystyle{plain}

\end{document}